\documentclass[11pt]{article}

\usepackage[utf8]{inputenc}
\usepackage{amsmath}
\usepackage{amssymb,amsfonts,amsthm}
\usepackage{mathrsfs}
\usepackage{mathtools}
\usepackage{algorithm}
\usepackage{algorithmic}
\usepackage{xcolor}
\usepackage{tikz}
\usepackage{fullpage}
\usepackage{dsfont}
\usepackage{thmtools}
\usepackage{thm-restate}
\usepackage{enumitem}
\usepackage{dsfont}
\usepackage{graphicx}
\usepackage{sidecap}

\usepackage{accents}

\usepackage{hyperref}
\usepackage{xcolor}
\hypersetup{
  colorlinks   = true, 
  urlcolor     = blue, 
  linkcolor    = blue, 
  citecolor   = blue 
}
\usepackage[capitalize, nameinlink]{cleveref}

\usepackage{tikz}
\usetikzlibrary{decorations.pathreplacing}
\usetikzlibrary{arrows,positioning}
\DeclareRobustCommand\shape{
 \lower5pt\hbox{
 \hskip-7pt
  \tikzset{circ/.style={circle, draw, fill=black, scale=.15}}
  \begin{tikzpicture}[semithick,scale=.3]
  \node (l1) at (0,.5) [circ]{};
  \node (l3) at (0.5,0.3) [circ]{};
  \draw[-] (l1) to node [auto] {} (l3);
    \end{tikzpicture}
  \hskip-8pt}
}

\newcommand{\bernoulli}{\mathsf{Bern}}

\newcommand{\hzero}{\mathsf{H}_0}
\newcommand{\hone}{\mathsf{H}_1}

\newtheorem{conjecture}{Conjecture}[]
\newtheorem{theorem}{Theorem}[section]
\newtheorem{challenge}{Challenge}[]

\newtheorem{lemma}{Lemma}[section]

\newtheorem{corollary}[lemma]{Corollary}

\newtheorem{claim}[lemma]{Claim}

\newtheorem{proposition}[theorem]{Proposition}
\theoremstyle{definition}
\newtheorem{remark}{Remark}[]
\newtheorem{definition}{Definition}
\newtheorem{problem}{Problem}[]

\newcommand{\bfx}{\mathbf{x}}

\newcommand{\Star}{\mathsf{Star}}
\newcommand{\cycle}{\mathsf{Cyc}}
\newcommand{\complete}{\mathsf{K}}
\newcommand{\connectedgraphs}{\mathsf{CGraphs}}
\newcommand{\graphs}{\mathsf{Graphs}}

\newcommand{\trace}{\mathsf{tr}}

\newcommand{\Fourier}{\Phi}

\DeclareMathOperator*{\var}{{\rm Var}}
\newcommand{\Var}{\var}

\newcommand{\cov}{{\rm Cov}}

\DeclareMathOperator*{\E}{{\rm I}\kern-0.18em{\rm E}}

\newcommand{\expect}{\E}
\renewcommand{\Pr}{\,{\rm I}\kern-0.18em{\rm P}}

\newcommand{\bfG}{\mathbf{G}}

\newcommand{\signedcount}{\mathsf{SC}}

\newcommand{\edges}{\mathsf{E}}
\newcommand{\vertices}{\mathsf{V}}
\newcommand{\unsignedcount}{\mathsf{C}}
\newcommand{\tspecial}{\mathbf{T}}
\newcommand{\leaves}{\mathsf{L}}
\newcommand{\parent}{\mathsf{par}}

\newcommand{\ER}{{Erd\H{o}s-R\'enyi}}

\newcommand{\bfK}{\mathbf{K}}

\newcommand{\iidsim}{\stackrel{\mathrm{i.i.d.}}\sim}

\newcommand{\ergraph}{{\mathbb{G}(n,q)}}
\newcommand{\ergraphhalf}{{\mathbb{G}(n,1/2)}}

\newcommand{\SBM}{\mathbb{SBM}}

\newcommand{\indicator}{\mathds{1}}

\usepackage[titles]{tocloft}
\title{Graph Quasirandomness for Hypothesis Testing of \\ Stochastic Block Models}

\author{Kiril Bangachev\thanks{Dept. of EECS, MIT. \texttt{kirilb@mit.edu}} \quad Guy Bresler\thanks{Dept. of EECS, MIT. \texttt{guy@mit.edu}. Supported by NSF Grant CCF-2428619.}}

\begin{document}

\maketitle

\pagenumbering{gobble}

\begin{abstract}
The celebrated theorem of Chung, Graham, and Wilson on quasirandom graphs implies that if the 4-cycle and edge counts in a graph $G$ are both close to their typical number in $\mathbb{G}(n,1/2),$ then this also holds for the counts of subgraphs isomorphic to $H$ for any $H$ of constant size. We aim to prove a similar statement where the notion of close is whether the given (signed) subgraph count can be used as a test between $\mathbb{G}(n,1/2)$ and a stochastic block model $\mathbb{SBM}.$ 

Quantitatively, this is related to approximately maximizing $H \longrightarrow |\Phi(H)|^{\frac{1}{|\mathsf{V}(H)|}},$ where $\Phi(H)$ is the Fourier coefficient of $\mathbb{SBM}$, indexed by subgraph $H.$ This formulation turns out to be equivalent to approximately maximizing the partition function of a spin model over alphabet equal to the community labels in $\mathbb{SBM}.$

We resolve the approximate maximization when $\mathbb{SBM}$ satisfies one of four conditions:
1) the probability of an edge between any two vertices in different communities is exactly $1/2$; 2) the probability of an edge between two vertices from any two communities is at least $1/2$ (this case is also covered in a recent work of Yu, Zadik, and Zhang); 3) the probability of belonging to any given community is at least $c$ for some universal constant $c>0$; 4) $\mathbb{SBM}$ has two communities.
In each of these cases, we show that
there is an approximate maximizer of $|\Phi(H)|^{\frac{1}{|\mathsf{V}(H)|}}$ in $\mathcal{A} = \{\text{stars, 4-cycle}\}.$ This implies that if there exists a constant-degree polynomial test distinguishing $\mathbb{G}(n,1/2)$ and $\mathbb{SBM},$ then the two distributions can also be distinguished via the signed count of some graph in $\mathcal{A}.$ We conjecture that the same holds true for distinguishing $\mathbb{G}(n,1/2)$ and any graphon if we also add triangles to $\mathcal{A}.$
\end{abstract}

\begin{keywords}{}
Stochastic Block Models, Quasirandom Graphs, Low Degree Polynomial Tests.
\end{keywords}

\newpage

\setcounter{tocdepth}{2}
\tableofcontents

\clearpage
\pagenumbering{arabic}

\section{Introduction}
\subsection{Background}
\paragraph{Quasirandomness.}
In the seminal paper \cite{chung87},  Chung, Graham, and Wilson initiated the study of \emph{quasirandom graphs}. The central question addressed by quasirandomness is:
\begin{quote}
\begin{center}
\emph{Does the graph $G$ resemble a random graph?}
\end{center}
\end{quote}
\cite{chung87} establishes the equivalence of several seemingly disparate properties that make a graph similar to a sample from $\ergraphhalf$ or, in their terminology, \emph{quasirandom}. To describe the result, we denote by 
$\unsignedcount_H(G)$ the number of labeled copies of $H$ in a graph $G$ (so, for example,
$\unsignedcount_H(\complete_n)=n(n-1)\cdots (n -|\vertices(H)|+1)$ for any $H$) and by 
$\unsignedcount^*_H(G)$ the number of induced labeled copies.

 \begin{theorem}[\cite{chung87}]
\label{thm:classicquasirandomnessnew}
    The following conditions are equivalent for an $n$-vertex graph $G:$
    \begin{enumerate}
        \item[P1)] $|\unsignedcount^*_H(G)| =(1+o(1))n^{|\vertices(H)|}2^{-\binom{|\vertices(H)|}{2}},$ which is equivalent to 
        $|\unsignedcount_H(G)| =(1+o(1))n^{|\vertices(H)|}2^{-|\edges(H)|},$
        for all graphs $H$ of constant size.
        \item[P2)] $|\edges(G)|\ge (1 +o(1))\frac{n^2}{4}$ and $|\unsignedcount_{\cycle_t}(G)|\le \frac{n^t}{2^t}(1+o(1))$ for some cycle $\cycle_t$ of even length $t\ge 4.$
        \item[P3)] $|\edges(G)|\ge (1+o(1))\frac{n^2}{4}$ and $\lambda_1(G)=\frac{n}{2}(1+o(1)),\lambda_2(G)=o(n)$ where these are the largest two eigenvalues of $G$ by absolute value.
        \item[P4)] For each subset $S$ of $\vertices(H),$ the number of edges restricted to this subset is $\frac{|S|^2}{4} + o(|\vertices(H)|^2).$
        \item[P5)] For each subset $S$ of $\vertices(H)$ with $S=\lfloor n/2\rfloor,$ the number of edges restricted to this subset is $(1 +o(1))\frac{n^4}{16}.$
        \item[P6)]
        $\displaystyle 
        \sum_{v\neq v'\in \vertices(H)}
        \bigg|\sum_{u\in \vertices(H)} \indicator[G_{vu}= G_{v'u}] - n/2\bigg| = o(n^3). 
        $
        \item[P7)] 
        $\displaystyle 
        \sum_{v\neq v'\in \vertices(H)}
        \bigg|\sum_{u\in \vertices(H)} \indicator[G_{vu}= G_{v'u}= 1] - n/4\bigg| = o(n^3).
        $
    \end{enumerate}
\end{theorem}

In particular, if $G$ has a number of 4-cycles and edges close to that of a typical sample from $\ergraphhalf,$ then any other subgraph count as well as the spectrum of $G$ resemble that of \ER.

\paragraph{Hypothesis Testing for Graph Structure.} Subgraph counts and spectral statistics are commonly used to compare random graphs to \ER{} in a different framework -- \emph{when testing for hidden structure in random graph distributions}. Concretely, the following hypothesis testing question has received considerable attention from the probability, computer science, and theoretical statistics communities. For a family of distributions $(\mathbb{P}_n)_{n\in \mathbb{N}}$ over $n$-vertex graphs, one aims to solve the following hypothesis testing problem on input an $n$-vertex graph $G$:
\begin{equation}
\tag{HT}
\label{eq:hypothesistest}
    \hzero: \ergraphhalf\quad \text{ versus }\quad \hone: \mathbb{P}_n.
\end{equation}
Both the \emph{information-theoretic} question (is there a consistent test?) and the \emph{computational} question (is there a computationally-efficient test?) are active research areas for many choices of $\mathbb{P}_n$ in \eqref{eq:hypothesistest}.  

Popular choices of $\mathbb{P}_n$ in \eqref{eq:hypothesistest} include graph distributions with the following types of hidden structure: Community structure such as stochastic block models \cite{HOLLAND1983sbm,Bui1984GraphBA,DYER1989sbm,boppana87sbm,bollobas07sbm,reichardt08detectsbm,COJAOGHLAN2010sbm,decele11detect,massoulie14detection,mossel18proofsbm} (and many more in \cite{abbecommunitydetection}), in particular the planted clique \cite{kuvcera1995expected,jerrum1992large,alon1998finding,barak19soslowerbound,brennan2018reducibility,Brennan19Reductions,brennan2020reducibility,hirahara24pc}; Geometric structure such as random geometric graphs \cite{Devroye11,Bubeck14RGG,brennan2019phase,Brennan22AnisotropicRGG,Liu2021APV,Liu2022STOC,Liu21PhaseTransition,Friedrich24cliques,bangachev2024detection,bangachev24fourier,bangachev2023random}; Planted dense subgraph such as a matching, clique, or cycle \cite{hajek2015computational,Ding2021ThePM,mossel23shapthresholds,dhawan2023detection,yu2024counting} (in addition to planted clique references).

One natural and computationally efficient approach to \eqref{eq:hypothesistest} is to compare simple statistics of $\mathbb{P}_n$ and $\ergraphhalf.$ For example, if $\mathbb{P}_n$ has a lot more triangles than $\ergraphhalf$ (as in the case of random geometric graphs over the sphere, see \cite{Bubeck14RGG}), then one can use this fact to distinguish the two distributions. The semblance with the classic quasirandomness theory, initiated in \cref{thm:classicquasirandomnessnew}, is striking -- subgraph counts are used to compare with \ER.

\paragraph{Towards a Quasirandomness Theory for Hypothesis Testing for Graph Structure.} In the context of \eqref{eq:hypothesistest}, the question of graph quasirandomness takes the following form:
\begin{quote}
\begin{center}
\emph{
Is there a small set of simple graph statistics such that
if $\mathbb{P}_n$ and $\ergraphhalf$ are indistinguishable under these statistics, then they are also computationally / information-theoretically indistinguishable?}
\end{center}
\end{quote}

Of course, one may speculate that \cref{thm:classicquasirandomnessnew} already answers this question. If the 4-cycle count and edge count of $\mathbb{P}_n$ and $\ergraphhalf$ are sufficiently close, then so are all other subgraph counts. Hence, polynomials in the edges of the two graph distributions are also close. Within the increasingly popular framework of low-degree polynomial tests (see \cref{sec:ldpprelim} for definitions), this would give evidence for computational indistinguishability. 

The reason this argument fails is that \cref{thm:classicquasirandomnessnew} is quantitatively too weak in the $o(1)$ dependence on the number of vertices. For example, consider $\mathbb{P}_n = \mathbb{G}(n, 1/2 + 1/\sqrt{n}).$ According to \cref{thm:classicquasirandomnessnew}, (a sample from) this graph distribution is quasirandom. However, one can easily distinguish $\mathbb{P}_n$ and $\ergraphhalf$: A simple concentration argument shows that with high probability, the first distribution has more than $\frac{1}{2}\binom{n}{2} + \frac{1}{8}n^{1.5}$ edges, while the second less than $\frac{1}{2}\binom{n}{2} + \frac{1}{8}n^{1.5}.$ 
In \cref{sec:priortechniquesQUASI}, we 
explain why even careful book-keeping of the $o(1)$-dependencies in  \cref{thm:classicquasirandomnessnew} is not sufficient. There is a fundamental barrier to the techniques used for   \cref{thm:classicquasirandomnessnew} and a different approach is needed to develop a quasirandomness theory in the setting of hypothesis testing.

In the recent work \cite{yu2024counting}, the authors make a first step towards this end.\footnote{To the best of our knowledge, the connection of their work to graph quasirandomness was not known to the authors of \cite{yu2024counting}.} 

\begin{theorem}[\cite{yu2024counting}]
\label{thm:countingstars}
Let $(H_n)_{n \in \mathbb{N}}$ be a sequence of fixed
$n$-vertex graphs. The sequence of planted distributions
$\mathbb{P}_{n \in \mathbb{N}}$ over $n$-vertex graphs is defined as follows. To sample from $\mathbb{P}_n,$ one first samples $\bfG$ from $\ergraphhalf,$\footnote{The results of \cite{yu2024counting} and the current work extend appropriately for any $\ergraph$ when $q\in (0,1)$ is an absolute constant (independent on $n$). We choose to work with $q = 1/2$ as this makes nearly no difference in the arguments. We do note that if we allow $q$ to depend on $n,$ we expect the behavior to change dramatically, see \cref{sec:discussion}.} then adds the edges of $H_n,$ producing $\bfG\cup H_n,$ and finally applies a uniformly random permutation to the vertices of $\bfG\cup H_n.$

There exists a constant degree polynomial test that distinguishes $\hzero: \ergraphhalf$ and $\hone: \mathbb{P}_n$ with high probability if and only if there exists some signed star count that distinguishes $\hzero: \ergraphhalf$ and $\hone: \mathbb{P}_n$ with high probability.
\end{theorem}

We will explain the precise meaning of \emph{constant degree polynomial test} and \emph{signed} subgraph count in a moment. More important for now is the interpretation of the theorem from a quasirandomness perspective: if signed star counts between $\mathbb{P}_n$ and $\ergraphhalf$ are sufficiently similar, then so are also all other constant-degree polynomial statistics.

The goal of this work is to extend the above theorem to a more diverse family of distributions. Specifically, the choice of model in \cref{thm:countingstars} makes the following strong structural assumption: \textbf{Even conditioned on the latent planted graph $H,$ the edge $\bfG_{ij}$ has marginal probability of appearance at least $1/2$}. Both the methods and conclusion of \cref{thm:countingstars} break down in some of the simplest models that do not exhibit this structure. For example, the following distribution is easily distinguishable from $\ergraphhalf$ via counting signed 4-cycles but signed stars fail to distinguish it. Each of $n$ vertices receives an independent uniformly random label $1$ or 2. If $u$ and $v$ receive the same label, they are adjacent with probability $1/2+ n^{-1/10}.$ If $u$ and $v$ receive different labels, they are adjacent with probability $1/2- n^{-1/10}.$ One can show that the expected number of signed stars is the same as in $\ergraphhalf$: zero. The reason is that even conditioned on the latent label of $u,$ the edge $\bfG_{u,v}$ is distributed as $\bernoulli(1/2),$ due to the randomness of the label of $v.$ 

\paragraph{Our Conjecture and Results.}
We make the following quasirandomness conjecture for hypothesis testing of a stochastic block model (\cref{def:sbm} below) against \ER.

\begin{conjecture}
\label{conj:optimality}
Suppose that $\mathbb{P}_n$ is a stochastic block model (whose parameters may depend arbitrarily on $n$) which one aims to distinguish from $\ergraphhalf.$ There exists a constant degree polynomial test that distinguishes $\hzero: \ergraphhalf$ and $\hone: \mathbb{P}_n$ with high probability if and only if one of the following signed subgraph counts
$$
\{\textsf{edge, stars, triangle, 4-cycle}\}
$$
distinguishes $\hzero: \ergraphhalf$ and $\hone: \mathbb{P}_n$ with high probability.    
\end{conjecture}

We note that \cite{yu2024counting} implicitly proves the conjecture whenever, even conditioned on the labels, any two vertices have a probability of connection at least $1/2.$

Our main contribution is to prove this conjecture in the following additional cases:
\begin{enumerate}
    \item $\mathbb{P}_n$ is an arbitrary stochastic block model on 2 communities.
    \item $\mathbb{P}_n$ is an arbitrary stochastic block model on a constant number of communities and each community label appears with constant positive probability.
    \item $\mathbb{P}_n$ is a stochastic block model in which every two vertices in distinct communities are adjacent with probability $1/2.$
\end{enumerate}
The stochastic block model is of special interest for several reasons. First, it is perhaps the most widely studied family of random graph distributions after \ER. Second, as we will see in the forthcoming section, the above results for stochastic block models with $k$-communities can be equivalently phrased as an approximate maximization of the partition function of a certain spin-glass model with alphabet of size $k.$ The interaction matrix of the spin model is determined by the stochastic block model interaction matrix and external fields correspond to the 
community probabilities. This gives an alternative statistical-mechanics interpretation of our results in addition to the quasirandomness perspective. Finally, stochastic block models can approximate arbitrarily well smooth graphons. In fact, we believe that \cref{conj:optimality} should hold more generally for graphons.

\subsection{Implications and Interpretations of \texorpdfstring{\cref{conj:optimality}}{conjecture}}

\paragraph{Fine-Grained Running Times.}
A degree-$D$ polynomial in the edges of an $n$-vertex graph may take up to time $n^{2D}$ to compute and, hence, be completely impractical. Yet, 
\cref{conj:optimality} tells us that we only need to compute 
signed star counts or signed 3- and 4-cycle counts. Signed counts of stars can be computed in near-linear time as the signed star count with a fixed central vertex $v$ is a simple one-dimensional functional of the degree of the vertex (which can trivially be computed in time $O(D)$ where $D$ is the star size). One simply needs to evaluate and add this function over all possible central vertices. On the other hand, signed triangles and 4-cycle counts can be evaluated in $n^\omega$ time, where $\omega \in [2,2.371\ldots]$ is the matrix-multiplication time \cite{williams24matrixmult}. Indeed, the signed triangle count of 
a graph with adjacency matrix $A$ is
is $\mathsf{trace}((2A- 11^T)^3)$ and the signed 4-cycle count is 
$\mathsf{trace}((2A- 11^T)^4) - n(n-1)(2n-3).$ Hence, if there exists a constant-degree polynomial distinguisher, there also exists a practical one.

More succinctly, our result suggests a strong dichotomy for the complexity of testing stochastic block-models with low-degree polynomial tests. Either there exists a low-degree distinguisher implementable in nearly-quadratic time or there does not exist any constant-degree polynomial distinguisher.

\paragraph{On Finding A Distinguisher.}
The practical distinguisher is easy to find as noted in \cite{yu2024counting} -- it is a signed star, triangle, or 4-cycle count. Hence, a statistician in practice may also simply try each of these tests. They do not even need to know the specific parameters of the stochastic block model, but only the fact that it is part of a family satisfying \cref{conj:optimality}.

Conversely, the failure of a very small set of tests (the signed counts of triangles, 4-cycles, starts) implies the failure of ALL low-degree polynomial tests.

\paragraph{Our Conjecture And Results in The Context of Low-Degree Hardness.} The current work and \cite{yu2024counting} are preceded by a long line of low-degree polynomial hardness results on testing between \ER{} and random graph models
\cite{barak19soslowerbound,hopkins2017bayesian,hopkins18,bangachev2023random,bangachev2024detection,bangachev24fourier,mao23detectionrecoverygap,mao2024informationtheoretic,kothari2023planted,rush2022,mardia24ldpquery} (nearly all of which are graphons).. We remark that there is an important conceptual difference between these prior results and ours. Our work does not prove hardness for any specific model (and in fact does not prove explicitly any hardness result). Instead, our work and \cite{yu2024counting} aim to derive certain universality principles for the low-degree polynomial framework itself over a rich family of testing problems. In this light, our work more closely resembles the universal (near)-equivalence between low-degree polynomial tests and SQ algorithms in \cite{brennan2021statisticalquery}.

\subsection{Signed Subgraph Counts, Stochastic Block Models, and Partition Functions}
\label{sec:introformalizingsubgraphsbmpartition}

We now introduce the main ingredients necessary to describe our results. 
\paragraph{Signed Subgraph Counts.} To derive a quantitatively stronger form of \cref{thm:classicquasirandomnessnew}, we need to quantify what it means for a subgraph count to be different enough from that of \ER{} so that one can use it towards \eqref{eq:hypothesistest}. We borrow the notion from the framework of low-degree polynomial tests \cite{hopkins2017bayesian,hopkins18} as in \cite{yu2024counting}. Let $H$ be a graph. On input an $n$-vertex graph $G,$ compute the signed count of $H$,
\begin{equation}
    \label{eq:signedsubgraphcount}
    \signedcount_H(G)\coloneqq 
    \sum_{H_1\sim H} \prod_{(ij)\in \edges(H_1)}(2G_{ij}-1)\,,
\end{equation}
where the sum is over all isomorphic copies of $H_1$ in the complete graph $\complete_n$ and $G_{ij}$ is the indicator of the respective edge. Compare with the unsigned count $\unsignedcount_H(G)$ from \cref{thm:classicquasirandomnessnew}, which corresponds to $\unsignedcount_H(G)\coloneqq 
    \sum_{H_1\sim H} \prod_{(ij)\in \edges(H_1)}G_{ij}.$ The advantage of working with \emph{signed counts} in the hypothesis testing setting is that they have a smaller variance (as noted by \cite{Bubeck14RGG}) with respect to $\ergraphhalf$ due to the fact that the signs of different subgraphs are uncorrelated (see \cref{thm:orthofourier}). 

One condition that would guarantee that $\signedcount_H$ distinguishes $\hzero:\ergraphhalf$ and $\hone:\mathbb{P}_n$ is  
\begin{align}
\label{eq:meanseparationforsignedcounts}
\Big|
\expect_{\bfG\sim \hone}\signedcount_H(\bfG) - 
\expect_{\bfG\sim \hzero}\signedcount_H(\bfG)
\Big| = \omega\Big( \max\Big(
\var_{\bfG\sim \hzero}[\signedcount_H(\bfG)]^{1/2},
\var_{\bfG\sim \hone}[\signedcount_H(\bfG)]^{1/2}\Big)\Big).
\end{align}
If \eqref{eq:meanseparationforsignedcounts} holds, one can test between $\hone$ and $\hzero$ by evaluating $\signedcount_H$ on the input $G.$ If $G\sim\mathsf{H}_b$ for $b\in \{0,1\},$ then $\displaystyle\signedcount_H(G)\in [\expect_{\bfG\sim \mathsf{H}_b}\signedcount_H(\bfG) \pm C\times \var_{\bfG\sim\mathsf{H}_b}[\signedcount_H(\bfG)]^{1/2}]\eqqcolon I_b$ with high probability for large enough $C = \omega(1)$ by Chebyshev's inequality. Condition \eqref{eq:meanseparationforsignedcounts} ensures that $I_0, I_1$ are disjoint for some appropriate $C$ and, hence, one can test by reporting the membership of $\signedcount_H$ to one of $I_0,I_1.$

Under the $\ergraphhalf$ distribution, when $H$ has a constant number of edges,
it holds that 
\begin{align}
    \label{ergraph}
    \expect_{\bfG\sim\ergraphhalf}[\signedcount_H(\bfG)] = 0\quad \text{ and }
    \quad
    \Var_{\bfG\sim\ergraphhalf}[\signedcount_H(\bfG)] = 
     \sum_{H_1\sim H}  1 = \Theta(n^{|\vertices(H)|}).
\end{align}
Thus, one can rephrase \eqref{eq:meanseparationforsignedcounts} as 
$|\displaystyle
\expect_{\bfG\sim \mathbb{P}_n}\signedcount_H(\bfG) 
| = \omega\Big( \max\big(
n^{|\vertices(H)|/2},
\var_{\bfG\sim \mathbb{P}_n}[\signedcount_H(\bfG)]^{1/2}\big)\Big).$

If the distribution $\mathbb{P}_n$ is permutation invariant, that is any vertex permutation $\pi$ is measure-preserving with respect to $\mathbb{P}_n$, one can further observe that 
\begin{align*}
\expect_{\bfG\sim \mathbb{P}_n}\signedcount_H(\bfG) = 
\expect_{\bfG\sim \mathbb{P}_n}\prod_{(ij)\in \edges(H)}(2G_{ij}-1)\times \Bigg( \sum_{H_1\sim H} 1\Bigg) = \Theta\Bigg(n^{|\vertices(H)|}\times \expect_{\bfG\sim \mathbb{P}_n}\Bigg[\prod_{(ij)\in \edges(H)}(2G_{ij}-1)\Bigg]\Bigg).
\end{align*}
The quantity 
\begin{align}
    \label{eq:fourier}
    \Fourier_{\mathbb{P}_n}(H)\coloneqq \expect_{\bfG\sim \mathbb{P}_n}\Bigg[\prod_{(ij)\in \edges(H)}(2G_{ij}-1)\Bigg]
\end{align}
denotes the Fourier coefficient of the distribution (i.e., of the probability mass function of)
$\mathbb{P}_n$
corresponding to shape $H.$ Going back to \eqref{eq:meanseparationforsignedcounts}, the inequality becomes 
\begin{equation}
\label{eq:statisticalscaling}
    n^{|\vertices(H)|} |\Fourier_{\mathbb{P}_n}(H)| = \omega(n^{|\vertices(H)|/2})
    \Longleftrightarrow
    |\Fourier_{\mathbb{P}_n}(H)|^{\frac{1}{|\vertices(H)|}}n^{1/2} = \omega(1).   
\end{equation}
Leaving aside the fact that \eqref{eq:statisticalscaling} does not capture the variance under $\mathbb{P}_n$\footnote{In our setting, it turns out that \eqref{eq:statisticalscaling} 
can ``nearly'' capture the variance under $\mathbb{P}_n.$ We return to this in \cref{thm:beatingvarnullimpliesbeatingvarplanted}.
}, \eqref{eq:statisticalscaling} suggests that the ``most powerful'' signed subgraph counts for distinguishing $\ergraphhalf$ and $\mathbb{P}_n$ are
the (approximate) maximizers of $$H\longrightarrow \Psi_{\mathbb{P}_n}(H), \text{ where } 
\Psi_{\mathbb{P}_n}(H)\coloneqq |\Fourier_{\mathbb{P}_n}(H)|^{\frac{1}{|\vertices(H)|}}.$$This motivates the following problem:

\begin{problem}[Approximate Maximization of Scaled Fourier Coefficients / Approximate Maximization of SBM Partition Function]
\label{problem:maximizingfourriercoefficients}
For every $n\in \mathbb{N},$ let $\mathcal{F}_n$ be a family of random graph distributions over $n$-vertex graphs invariant under vertex permutations. Let 
$\mathcal{F}=\bigcup_{n\in \mathbb{N}}\mathcal{F}_n.$
Let $D\in \mathbb{N}$ be some fixed constant. Find some minimal set $\mathcal{A}_D$ of graphs on at most $D$ edges with the following property. 
There exists some constant $C_{D, \mathcal{F}}>0$ (depending only on $\mathcal{F}, D$) such that for any graph $H$ \emph{without isolated vertices} on at most $D$ edges and any $\mathbb{F}\in \mathcal{F},$
\begin{align}
\label{eq:comparisontomaximalfourierset}
\Psi_{\mathbb{F}}(H)
\le C_{D, \mathcal{F}}\times 
\max_{K\in \mathcal{A}_D}
\Psi_{\mathbb{F}}(K).
\end{align}
\end{problem}

To gain intuition about \cref{problem:maximizingfourriercoefficients}, consider the setting of dense planted subgraphs of \cite{yu2024counting}. The authors implicitly show that the set of stars on at most $D$ edges is such a set $\mathcal{A}_D$.

\paragraph{Stochastic Block Models.} One of the motivations of our work is to extend the results of \cite{yu2024counting} to a setting when, conditioned on the latent structure of the model, some of the edges might have a probability of appearance less than half. Stochastic block models are a canonical example.

\begin{definition}
\label{def:sbm}
A stochastic block model on $k$ communities is parametrized by a probability vector $p=(p_1, p_2,\ldots,p_k)\in (0,1]^k$ and a symmetric matrix $Q\in [-1,1]^k.$ To generate an $n$-vertex sample from $\SBM(n;p,Q),$ one performs the following two-step process:
\begin{enumerate}
    \item First, draw $n$ independent labels $\bfx_1,\bfx_2,\ldots,\bfx_n$ from the distribution over $[k]$ specified by $p.$
    \item Produce an $n$-vertex graph $\bfG$ by drawing each edge $(i,j)$ as an independent $\bernoulli(\frac{1+Q_{\bfx_i, \bfx_j}}{2})$ random variable.
\end{enumerate}
\end{definition}
More commonly, the SBM is defined with the probability matrix $M \in [0,1]^{k\times k}$ where $M_{i,j} = (1 + Q_{i,j})/2,$ but the current 
parametrization is more convenient for discussing Fourier coefficients.
We proceed with some examples:
\begin{enumerate}
    \item Suppose that $Q = \mathbf{0}_{k\times k}$ is the zero matrix. Then, regardless of labels, each edge $(i,j)$ appears with probability $1/2.$ Hence, this is $\ergraphhalf.$
    \item Suppose that $k =2,$ $p= (1/n^\alpha,1-1/n^\alpha)$ and $Q= \begin{pmatrix}1 & 0\\ 0 &0\end{pmatrix}.$ Then, $\SBM(n; p,Q)$ is the planted clique distribution with expected clique of size $n^{1-\alpha}.$ Namely, the vertices with label 1 form a clique and every other edge appears with probability $1/2.$ 
\end{enumerate}

\paragraph{Partition Functions.} We now record an explicit formula for $\Fourier_{\SBM(p,Q)}(H).$ 
We will often write $\SBM(p,Q)$ instead of $\SBM(n; p, Q)$ whenever the dependence on $n$ is clear or unimportant. 
We write $\bfx\sim p$ for a variable distributed over $[k]$ with p.m.f. $p.$

\begin{proposition}
\label{prop:explicitFourier}
Consider some $\SBM(p,Q)$ distribution and let $H$ be a graph on $h$ vertices. Then,
\begin{align}
\label{eq:explicitfourier}
\Fourier_{\SBM(p,Q)}(H)
=\expect_{(\bfx_i)_{i = 1}^h\iidsim p}\Bigg[\prod_{(i,j)\in \edges(H)}Q_{\bfx_i\bfx_j}\Bigg] = 
\sum_{x_1, x_2, \ldots, x_h\in [k]}\Big(
\prod_{i = 1}^h p_{x_i}
\times
\prod_{(i,j)\in \edges(H)}
Q_{x_i,x_j}\Big).
\end{align}
\end{proposition}
\begin{proof}
Without loss of generality, assume that the vertices of $H$ are $\{1,2,\ldots,h\}.$
By definition,
\begin{align*}
    \Fourier_{\SBM(p,Q)}(H) & = \expect_{\bfG\sim \SBM(p,Q)}\Bigg[\prod_{(i,j)\in \edges(H)}(2\bfG_{i,j}- 1)\Bigg]\\
    &= \expect
    \Bigg[
    \expect
    \Bigg[
    \prod_{(i,j)\in \edges(H)}(2\bfG_{i,j}- 1)\Bigg|(\bfx_{i})_{1\le i \le h}
    \Bigg]
    \Bigg]\\
    & =
    \expect
    \Bigg[
    \prod_{(i,j)\in \edges(H)}
    \expect
    \Bigg[
    (2\bfG_{i,j}- 1)\Bigg
    |(\bfx_{i})_{1\le i \le h}
    \Bigg]
    \Bigg]\\
    & = 
    \expect
    \Bigg[
    \prod_{(i,j)\in \edges(H)}
    \expect
    \Bigg[
    (2\frac{1+ Q_{\bfx_i,\bfx_j}}{2}- 1)\Bigg
    |(\bfx_{i})_{1\le i \le h}
    \Bigg]
    \Bigg]\\
    & =\expect\Bigg[\prod_{(i,j)\in \edges(H)}Q_{\bfx_i\bfx_j}\Bigg]=
    \sum_{x_1, x_2, \ldots, x_h\in [k]}\Bigg(
    \prod_{i = 1}^h p_{x_i}\times \prod_{(i,j)\in \edges(H)}
Q_{x_i,x_j}\Bigg).\qedhere
\end{align*}
\end{proof}
The right hand-side would be the partition function of a $k$-spin system if the matrix $Q$ had non-negative entries. Denoting $f(s)\coloneqq \log p_s, g(s,t)\coloneqq \log Q_{s,t},$ we obtain the more familiar from spin systems expression 
\begin{equation}
\begin{split}
    \Fourier_{\SBM(n;p,Q)}(H) & = 
\sum_{x_1, x_2, \ldots, x_h\in [k]}\Bigg(
\prod_{i = 1}^h p_{x_i}
\times
\prod_{(i,j)\in \edges(H)}
Q_{x_i,x_j}\Bigg)\\
& =
\sum_{x_1, x_2, \ldots, x_h\in [k]} 
\exp\Bigg(\sum_{(i,j)\in \edges(H)} g(x_i,x_j)+ \sum_{i \in \vertices(H)} f(x_i)\Bigg)\eqqcolon Z_{H}(g,f).
\end{split}
\end{equation}

Hence, (approximately) maximizing $\Psi_{\SBM(p,Q)}(H) = |\Fourier_{\SBM(p,Q)}(H)|^{\frac{1}{|\vertices(H)|}}$ over graphs $H$ is the same as approximately maximizing $|Z_{H}(g,f)|^{\frac{1}{|\vertices(H)|}}$ or, equivalently, $\frac{1}{|\vertices(H)|}\log |Z_{H}(g,f)|.$
This problem has received a lot of attention for spin systems, most notably in the case of the hard-core model \cite{Alon1991IndependentSI,kahn2001hardcore,ZHAO_2009,galvin2011hardcore,davies2017independentsets,Sah2018ARS,Sah_2019}. Yet, our task differs from prior work in at least two notable ways. What makes our problem easier is that we are content with approximate maximization of $|Z_{H}(g,f)|^{\frac{1}{|\vertices(H)|}}$ (up to constant multiplicative factors) as this does not change the asymptotics of testing. What makes our problem harder is that the interaction terms $Q_{i,j}$ might be negative, so prior techniques, for example based on the bipartite swapping trick \cite{ZHAO_2009}, fail.

\subsection{Our Results}
\subsubsection{Results for Restricted Stochastic Block Models}
\begin{theorem}[Main Results on Testing and Partition Function Maximization]
\label{thm:mainintro}
    Consider the $k$-community stochastic block model 
    $\SBM(p,Q).$ 
    Let $D$ be an even integer constant greater than $4.$
    Then, in \cref{problem:maximizingfourriercoefficients},one can take the set $\mathcal{A}_D$ of approximate maximizers, to be:
    \begin{enumerate}
        \item \emph{Diagonal SBMs} (\cref{thm:maximizepartitiondiagonal}): If $Q$ satisfies that $Q_{i,j}= 0$ whenever $i\neq j,$ then one can take $$\mathcal{A}^1_D = \{\text{edge, star on 2 edges}\}.$$
        \item \emph{Non-Negative SBMs} (\cref{thm:maximizinginnonnegative} and \cite{yu2024counting}): If $Q$ satisfies that $Q_{i,j}\ge 0$ for each $i,j,$ then one can take $$\mathcal{A}^2_D = \{\text{stars on at most $D$ edges}\}.$$
        \item \emph{SBMs with Non-Vanishing Community Probabilities} (\cref{thm:maximizingpartitioninnonvanishing}): If $p$ satisfies that $p_i \ge c \; \forall i\in [k]$ for some universal constant $c,$
        then one can take 
        $$\mathcal{A}^3_D = \{\text{edge, star on 2 edges, 4-cycle}\}.$$
        \item \emph{SBMs with Two Communities} (\cref{thm:maximizingpartition2sbm}): If $k = 2$ and $D$ is even, then 
        one can take 
        $$\mathcal{A}^4_D =\{\text{stars on at most $D$ edges}\}\cup \{\text{4-cycles\}}.$$
    \end{enumerate}
    Furthermore, in each of the cases above, if there is a constant degree polynomial test distinguishing the respective SBM model from $\ergraphhalf$ with high probability, then one can also distinguish with high probability using the signed count of some graph in the respective $\mathcal{A}_D.$
\end{theorem}

We remark that allowing for small increase in the vertex size, the same theorem holds for degrees as large as $D = o(\log n/\log\log n).$ We formalize this in \cref{prop:lognoverloglognversion}. We choose to write all the proofs for constant degree polynomials, $D = O(1)$ for simplicity of exposition, but the extensions to $D = o(\log n /\log\log n)$ are only a matter of careful bookkeeping. We do not pursue this direction since we are not aware of any concrete advantages of degree $O(\log n/\log \log n)$ polynomials over constant-degree polynomials. Different would have been the case if our techniques also captured degree $O(\log n)$ polynomials which encompass many spectral methods. This, however, seems challenging at present. 

\subsubsection{One-to-one comparisons of Fourier coefficients: Results and Barriers}
\label{sec:onetoonefourier}
\paragraph{Barrier to one-to-one comparisons.}
One natural approach to identifying such optimal sets $\mathcal{A}_D$ is via \emph{one-to-one comparisons} of Fourier coefficients. For example in the case of diagonal SBMs, prove that for any $H,$ there exists some $\mathsf{K}\in \{\text{edge, wedge}\}$ such that 
$
\Psi_{\SBM(p,Q)}(H)\le 
\Psi_{\SBM(p,Q)}(K)
$ for any diagonal SBM. This indeed works for diagonal SBMs. 

This simple approach, however, fails in the case of SBMs with non-vanishing community fractions and 2-SBMs (as demonstrated in \cref{examplethm:onetoneinsufficient}). Instead, we rely on \emph{many-to-one comparisons}. For example in the case of 2-SBMs,
we prove that for any $H$ and any $\SBM(p, Q),$ there exists some $\mathsf{K}_{\SBM(p,Q)} \in \{\text{stars on at most $D$ edges}\}\cup \{\text{4-cycles}\}$ such that 
$$
\Psi_{\SBM(p,Q)}(H)\lesssim
\Psi_{\SBM(p,Q)}(\mathsf{K}_{\SBM(p,Q)}).
$$
The only difference from one-to-one comparison is in the order of quantifiers -- here $\mathsf{K}$ can depend on the specific stochastic block model. 
This flexibility turns out to be necessary for results 3 and 4 in our \cref{thm:mainintro} as demonstrated by \cref{examplethm:onetoneinsufficient}.

\paragraph{Some one-to-one comparisons.}
It is still useful to consider what one-to-one comparisons we can obtain. 
We summarize below.

\begin{theorem}[One-to-one comparison of Fourier Coefficients]
\label{thm:onetonne}
    Consider any  
    $\SBM(p,Q).$ Then:
    \begin{enumerate}
        \item \cref{thm:largercycles}: For any $t \ge 5,$ $\Psi_{\SBM(p,Q)}(\cycle_t)\le \Psi_{\SBM(p,Q)}(\cycle_4).$ 
        \item \cref{thm:chisquaredapproach}: If $H$ has a degree $d$ vertex, then
        $|\Fourier_{\SBM(p,Q)}(H)|\le |\Fourier_{\SBM(p,Q)}(\complete_{2,d})|^{1/2}.$
        \item \cref{thm:4cyclewithextra}: Let $\complete_4^-$ be the graph on 4 vertices with 5 edges. Then, $\Psi_{\SBM(p,Q)}(\complete_4^-)\le \Psi_{\SBM(p,Q)}(\cycle_4).$ 
    \end{enumerate}
\end{theorem}
The first statement explains why the only signed cycle counts for 
distinguishing stochastic block models and $\ergraphhalf$ that appear in the literature are triangles (for example, planted coloring in \cite{kothari2023planted}) and 4-cycles (for example, the quiet planting in \cite{kothari2023planted}), but no larger cycle counts. The proof is a simple spectral argument, at a high-level, similar to the way spectrum (P3) and cycle counts (P2) in \cref{thm:classicquasirandomnessnew} are related in classical quasirandomness \cite[F4 and F5]{chung87}.

The second inequality also follows the proof of equivalence used in \cite[F12]{chung87}, namely used to show that (P6) implies  (P1) in \cref{thm:classicquasirandomnessnew} (and, implicitly, appears in \cite{Liu2021APV} for our setting of signed subgraph counts). And while the quantitative dependence in the second inequality is strong enough to imply \cref{thm:classicquasirandomnessnew}, it is too weak for the purposes of hypothesis testing. We would like to show the much stronger inequality 
$|\Fourier_{\SBM(p,Q)}(H)|\le |\Fourier_{\SBM(p,Q)}(\complete_{2,d})|^{|\vertices(H)|/|\vertices(\complete_{2,d})|}$ (note that 
$|\Fourier_{\SBM(p,Q)}(\complete_{2,d})|\le 1$ and $|\vertices(H)|\ge d + 1\ge |\vertices(\complete_{2,d})|/2 = (d+2)/2$).

Nevertheless, it turns out that in certain cases, one can improve the argument in Item 2 above. For example, Item 2 implies that  $|\Fourier_{\SBM(p,Q)}(\complete_4^-)|\le |\Fourier_{\SBM(p,Q)}(\complete_{2,2})|^{1/2}.$ However, we show that in this specific case, the argument can be appropriately modified to imply the stronger, and sufficient for the hypothesis testing setting, inequality from Item 3:
$|\Fourier_{\SBM(p,Q)}(\complete_4^-)|\le |\Fourier_{\SBM(p,Q)}(\complete_{2,2})|.$ 
This inequality also appears 
implicitly in \cite{Liu2021APV}.
In particular, this result demonstrates why no examples in the literature appear in which one tests via signed $\complete_4^-$-counts. 
\subsubsection{A Library of Examples}
In \cref{appendix:library} we provide example SBM distributions which demonstrate the necessity 
of different subgraph (signed) counts appearing in the respective sets $\mathcal{A}_D$ in \cref{thm:mainintro}. In 
\cref{examplethm:notodddegree}, we show that for certain SBMs even on just 2 balanced communities, 
one needs to go beyond the signed star counts proposed by \cite{yu2024counting} to distinguish the model from \ER{} with a constant-degree polynomial. In \cref{examplethm:onetoneinsufficient}, we illustrate the aforementioned necessity of many-to-one comparisons. Finally, with \cref{examplethm:fourcyclenostarnotriangle,examplethm:trianglenostarnotfourcycle,examplethm:1starsnotrianglenofourcycle,examplethm:2starsnotrianglenofourcycle,examplethm:largestarsnotrianglenofourcycle} we show that each of the triangle, 4-cycle, edge, wedge, and a large star necessarily belongs to the set in \cref{conj:optimality}.

\subsection{Prior Techniques and Barriers}
Our two main results are parts 3 and 4 of \cref{thm:mainintro}. We begin with an overview of previous techniques appearing in 
the recent work on planted graph models \cite{yu2024counting} and
the classical quasirandomness theory \cite{chung87} and outline two barriers which prevent those techniques from working in our setting. Then, we explain our two new key ideas towards overcoming these barriers.

\subsubsection{Signed Stars in Planted-Subgraph Models}
\label{sec:priortechniquesSTARS}
We begin with the recent work \cite{yu2024counting}. Adapted to our setting of stochastic block models, one part of their argument shows that if $Q_{i,j}\ge 0$ for all $i,j$, then for any graph $H$ on at most $D$ edges without isolated vertices, there exists some star on $t\le D$ edges such that 
$
\Psi_{\SBM(p,Q)}(H)\le
\Psi_{\SBM(p,Q)}(\Star_t)
.
$ For completeness, we give a full proof of this fact in \cref{thm:maximizinginnonnegative}, but here is a sketch of the proof. Recall \eqref{eq:explicitfourier} -- for any graph $H$ without isolated vertices,
$$
\Fourier_{\SBM(p,Q)}(H)
=
\sum_{x_1, x_2, \ldots, x_h\in [k]}\Big(
\prod_{i = 1}^h p_{x_i}
\times
\prod_{(i,j)\in \edges(H)}
Q_{x_i,x_j}\Big).
$$
Since all values of $Q_{x_i, x_j}$ \emph{are non-negative} and $[0,1]$-valued, removing an edge from $H$ cannot decrease the quantity on the right hand-side. Hence, one can \emph{iteratively remove edges} from $H.$ The only condition one needs to ensure is that there are no isolated vertices. Therefore, the process will terminate with a graph $K$ in which every connected component is a star. Stars are the only connected graphs from which one cannot remove an edge without creating isolated vertices. Altogether, this means that 
$\Fourier_{\SBM(p,Q)}(H)\le 
\Fourier_{\SBM(p,Q)}(K)
$ where $K$ has $|\vertices(H)|$ vertices and every connected component is a star. An elementary factorization of Fourier coefficients of SBMs over connected components (\cref{prop:vertexdisjointgrahs}) implies that 
$\Psi_{\SBM(p,Q)}(H)\le 
\Psi_{\SBM(p,Q)}(\Star_t)
$ for some star graph on at most $D$ edges (one of the components of $K$).

The challenge of extending this argument to general $\SBM(p,Q)$ distributions when $Q$ has negative entries is that removing an edge can decrease the Fourier coefficient of a graph. This is illustrated in our \cref{examplethm:notodddegree}.
Namely, take $k = 2, p = (1/2, 1/2)$
and $Q = \begin{pmatrix}1 & -1\\-1 & 1\end{pmatrix}.$ Let $H = \complete_{4,4},$ the complete bipartite graph with two parts of size 4. One can show that $\Fourier_{\SBM(p,Q)}(H) = 1$ (as $H$ has only even-degree vertices). However, for any edge $e,$ if $H\backslash e$ is the graph $H$ with edge $e$ deleted, 
$\Fourier_{\SBM(p,Q)}(H\backslash e) = 0$ (as $H$ has an odd-degree vertex). Hence, removing an edge can dramatically decrease a Fourier coefficient -- from being 1 and sufficient to test against $\ergraphhalf$ to being 0 and, thus, being completely useless towards distinguishing from $\ergraphhalf.$

The fact that edge-removals fail when $Q$ has negative entries is a serious obstacle towards proving our results 3 and 4 in \cref{thm:mainintro}. The reason is that we always want to compare the Fourier coefficient of a graph $H$ to the Fourier coefficient of very sparse graphs -- stars or 4-cycle. 

\begin{challenge}
\label{challenge:iterativetechniquefails}
The natural approach of comparing the Fourier coefficient of a graph $H$ to a sparser graph by iteratively removing edges sometimes fails when $Q$ has negative entries, even on two communities each appearing with probability $1/2.$
\end{challenge}

\noindent
\textbf{Insight from \cref{challenge:iterativetechniquefails}:}
We need \emph{one-shot} comparisons instead of \emph{iterative comparisons}.

\subsubsection{Classical Quasirandomness}
\label{sec:priortechniquesQUASI}
The celebrated \cref{thm:classicquasirandomnessnew} of \cite{chung87} gives several equivalent conditions for when the subgraph counts of graphs resemble that of \ER. One way to rewrite their results to make them more similar to our setting is as follows. Let $G$ be any graph on $n$ vertices and let $\bfG$ be $G$ after a uniformly random vertex permutation. Then, the condition $|\unsignedcount_H(G)| =(1+o(1))n^{|\vertices(H)|}2^{-|\edges(H)|}$ can be equivalently rewritten as 
\begin{align}
\label{eq:reformulatepseudorandom}
\Big|\expect\Big[\prod_{(i,j)\in \edges(H)}\bfG_{i,j}\Big] - 2^{-|\edges(H)|}\Big| = o(1).
\end{align}
There are several ways in which \eqref{eq:reformulatepseudorandom} differs with our result:
\begin{enumerate}
    \item \emph{Distributional:} The distribution of $\bfG$ is not a stochastic block model. However, the stochastic block models is a vertex-symmetric distribution. In particular, \cref{thm:classicquasirandomnessnew} applies to (a high-probability sample of) the SBM distributions. 
    \item \emph{Signed versus Unsigned Counts:} \cref{thm:classicquasirandomnessnew} is for unsigned counts $|\expect \prod_{(i,j)\in \edges(H)}\bfG_{i,j} - 2^{-|\edges(H)|}| $ instead of signed counts $|\expect\prod_{(i,j)\in \edges(H)}(2\bfG_{i,j} -1)|.$ This difference is again minor. Expanding the product in $\prod_{(i,j)\in \edges(H)}(2\bfG_{i,j} -1)$ and applying triangle inequality shows that for constant size graphs $H,$ \eqref{eq:reformulatepseudorandom} is equivalent to
\begin{align}
\label{eq:reformulatepseudorandomfourier}
\Big|\expect \prod_{(i,j)\in \edges(H)}(2\bfG_{i,j} - 1)\Big| = o(1).
\end{align}
\item \emph{Scaling:} The real problem with the approach of \cite{chung87} is the scaling. In place of \eqref{eq:reformulatepseudorandomfourier}, we need the much more fine-grained inequality \eqref{eq:statisticalscaling}:
$|\expect[ \prod_{(i,j)\in \edges(H)}(2\bfG_{i,j} - 1)]|^{\frac{1}{|\vertices(H)|}} = o(n^{-1/2}).$ It turns out that improving \eqref{eq:reformulatepseudorandomfourier} is not simply a matter of carefully keeping track of the $o(1)$-dependence in \cite{chung87} as their methods are fundamentally too weak for our purposes. For example, an intermediate step in their argument (Fact 12), rewritten in the setting of Fourier coefficients of stochastic block models (see \cref{thm:onetonne}) gives that for any $\SBM(p,Q)$ distribution, and graph $H$ with a degree $d$ vertex,
\begin{align}
\label{eq:comparetoK2dintro}
\Big|\expect \prod_{(i,j)\in \edges(H)}(2\bfG_{i,j} - 1)\Big|\le
\Big|\expect \prod_{(i,j)\in E(\complete_{2,d})}(2\bfG_{i,j} - 1))\Big|^{1/2}.
\end{align}
So, take for example $H = \complete_{4,4}.$ This inequality implies that if $\Big|\expect \prod_{(i,j)\in E(\complete_{2,4})}(2\bfG_{i,j} - 1))\Big| = o(1),$ then $\Big|\expect \prod_{(i,j)\in \edges(\complete_{4,4})}(2\bfG_{i,j} - 1)\Big| = o(1).$ However, 
if $\Big|\expect \prod_{(i,j)\in E(\complete_{2,4})}(2\bfG_{i,j} - 1))\Big|^{1/6} = o(n^{-1/2}),$ as needed in \eqref{eq:statisticalscaling}, it implies that 
$\Big|\expect \prod_{(i,j)\in \edges(\complete_{4,4})}(2\bfG_{i,j} - 1)\Big|^{1/8} = o(n^{-3/16})$ which is too weak for our purposes.

\end{enumerate}
We outline the scaling of Fourier coefficients to the power of $1/|\vertices(H)|$ as a second main challenge.

\begin{challenge}
\label{challenge:statisticalscaling}
Developing techniques that not only compare Fourier coefficients, but do so with the appropriate $1/|\vertices(H)|$
scaling appearing in \eqref{eq:statisticalscaling}. 
\end{challenge}

We note that this challenge is directly related to the one-to-one versus many-to-one comparisons discussed in \cref{sec:onetoonefourier} as apparent from the discussion above in Item 3.

\medskip

\noindent
\textbf{Insight from \cref{challenge:statisticalscaling}:} We need \emph{many-to-one} comparison inequalities.

\subsection{Our Key Ideas}
We now describe our two main ideas used to overcome \cref{challenge:iterativetechniquefails} and \cref{challenge:statisticalscaling}.
\subsubsection{Main Idea 1: Leaf-Isolation Technique} This idea is used to overcome both challenges and appears in the proofs of both part 3 and part 4 of \cref{thm:mainintro}. For simplicity, we illustrate  with part 3. That is, $\SBM(p,Q)$ satisfies that each community label $i\in [k]$ appears with probability $p_i$ at least $c$ for some absolute constant $c>0$ (which implies that $k \le 1/c = O(1).$) An elementary calculation (\cref{lem:4cycleinnonvanishing}) shows that  
\begin{equation}
\label{eq:intro4cycleinnonvanishing}
\Fourier_{\SBM(p,Q)}(\cycle_4) = \Omega(\max_{u,v}|Q_{u,v}|^4).
\end{equation}    
Hence, for any graph $H$ that satisfies $h = |\vertices(H)|\le |\edges(H)|,$ by \eqref{eq:explicitfourier}
\begin{equation}
\label{eq:intro4cycle}
\begin{split}
     |\Fourier_{\SBM(p,Q)}(H)|
&=
\Big|\sum_{x_1, x_2, \ldots, x_h\in [k]}\Big(
\prod_{i = 1}^h p_{x_i}
\times
\prod_{(i,j)\in \edges(H)}
Q_{x_i,x_j}\Big)\Big|\\
& \le 
\sum_{x_1, x_2, \ldots, x_h\in [k]}\Big(
\prod_{i = 1}^h p_{x_i}
\times
\prod_{(i,j)\in \edges(H)}
|Q_{x_i,x_j}|\Big)\\
&\lesssim 
\max_{u,v}|Q_{u,v}|^{|\edges(H)|}\le 
\max_{u,v}|Q_{u,v}|^{|\vertices(H)|}\lesssim 
|\Fourier_{\SBM(p,Q)}(\cycle_4)|^{\frac{|\vertices(H)|}{|\vertices(\cycle_4)|}},
\end{split}
\end{equation}
which is enough. Note that so far we have overcome \cref{challenge:iterativetechniquefails} by directly comparing to a 4-cycle. What does remain a difficulty is \cref{challenge:statisticalscaling}. If we use
\eqref{eq:intro4cycle} for a tree $T,$ the \emph{scaling is too weak} as we obtain
$|\Fourier_{\SBM(p,Q)}(T)|^{\frac{1}{|\edges(T)|}}\le |\Fourier_{\SBM(p,Q)}(\cycle_4)|^{\frac{1}{|\edges(\cycle_4)|}}$ and $|\edges(T)| = |\vertices(T)| - 1.$

\begin{figure}[!htb]
\label{fig:leafisolationtriangleineq}
\begin{center}
\begin{tikzpicture}


\filldraw[thick, red] (-6,0) circle (.05cm);
\node[align=center] at (-6,-.3) {\textcolor{red}{$6$}}; 
\filldraw[thick, black] (-6,1) circle (.05cm);
\node[align=center] at (-6,1.3) {\textcolor{black}{$4$}};
\filldraw[thick, black] (-5,0) circle (.05cm);
\node[align=center] at (-5,-.3) {\textcolor{black}{$3$}};
\filldraw[thick, black] (-4,0) circle (.05cm);
\node[align=center] at (-4,-.3) {\textcolor{black}{$2$}};
\filldraw[thick, red] (-4,1) circle (.05cm);
\node[align=center] at (-4,1.3) {\textcolor{red}{$5$}};
\filldraw[thick, black] (-3,0) circle (.05cm);
\node[align=center] at (-3,-.3) {\textcolor{black}{$1$}};

\draw[thick,red] (-6,0) -- (-5,0);
\draw[thick,black] (-6,1) -- (-5,0);
\draw[thick,black] (-5,0) -- (-4,0);
\draw[thick,red] (-4,0) -- (-4,1);
\draw[thick,black] (-4,0) -- (-3,0);

\node[align=center] at (-4.5,-1) {\textcolor{black}{Graph $H$}};

\node[align=center] at (-.5,.5) {
\text{Leaf-Isolation}\\
\text{Triangle Inequality}\\
\text{$\Longrightarrow$}
};

\filldraw[thick, black] (2,1) circle (.05cm);
\node[align=center] at (2,1.3) {\textcolor{black}{$4$}};
\filldraw[thick, black] (3,0) circle (.05cm);
\node[align=center] at (3,-.3) {\textcolor{black}{$3$}};
\filldraw[thick, black] (4,0) circle (.05cm);
\node[align=center] at (4,-.3) {\textcolor{black}{$2$}};
\filldraw[thick, black] (5,0) circle (.05cm);
\node[align=center] at (5,-.3) {\textcolor{black}{$1$}};

\draw[thick,black] (2,1) -- (3,0);
\draw[thick,black] (3,0) -- (4,0);
\draw[thick,black] (4,0) -- (5,0);

\node[align=center] at (6,.5) {\textcolor{black}{$\bigsqcup$}};

\filldraw[thick, red] (7,0) circle (.05cm);
\node[align=center] at (7,-0.3) {\textcolor{red}{$6$}};
\filldraw[thick, blue] (8,0) circle (.05cm);
\node[align=center] at (8,-0.3) {\textcolor{blue}{$7$}};
\filldraw[thick, red] (8,1) circle (.05cm);
\node[align=center] at (8,1.3) {\textcolor{red}{$5$}};

\draw[thick,red] (7,0) -- (8,0);
\draw[thick,red] (8,0) -- (8,1);

\node[align=center] at (4,-1) {\textcolor{black}{Graph $H'$}};

\node[align=center] at (7,-1) {\textcolor{black}{$\Star_2$}};


\end{tikzpicture}
\caption{Illustration of the leaf-isolation inequality over the graph $H$ with leaves $5,6.$ It effectively compares $|\Fourier_{\SBM(p,Q)}(H)|$ to 
$|\Fourier_{\SBM(p,Q)}(H'\sqcup \Star_2)| = 
|\Fourier_{\SBM(p,Q)}(H')|\times |\Fourier_{\SBM(p,Q)}(\Star_2)|. 
$ In this comparison, a new vertex $7$ is created, which resolves the issue that $|\edges(H)|<|\vertices(H)|.$}
\end{center}
\end{figure}
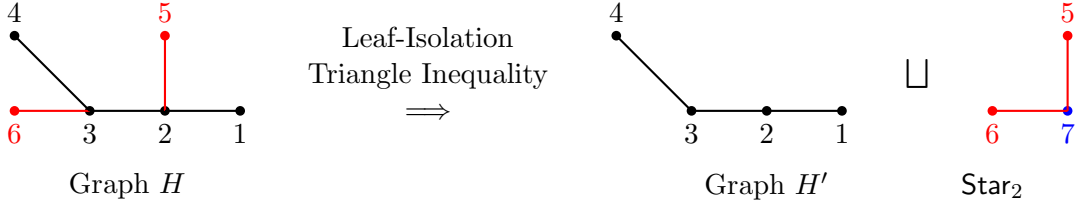

\paragraph{Leaf-Isolation Technique.} When $T$ is a tree on at least 3 vertices, $T$ has at least two leaves. Let these be $h-1,h$ with parents $\parent(h-1), \parent(h).$ Instead of the crude triangle inequality applied as a first step in \eqref{eq:intro4cycle}, we use a more fine-grained ``leaf-isolation'' triangle inequality which allows us to artificially create a star count in \eqref{eq:intro4cycle} as in the figure. On a more technical level, we perform the summation over all vertices but the two leaves first and then over the two leaves:

\begin{equation}
    \label{intro:leafisolation}
\begin{split}
 \Big|\Fourier_{\SBM(p,Q)}(T)\Big| &= 
\Big|
\sum_{x_1, x_2, \ldots, x_h\in [k]}\Big(
\prod_{i = 1}^h p_{x_i}
\times
\prod_{(i,j)\in \edges(T)}
Q_{x_i,x_j}\Big)\Big|\\
& = 
\Bigg|
\sum_{x_{1},x_2, \ldots, x_{h-2}}\Bigg(
p_{x_{1}}p_{x_2}\cdots p_{x_{h-2}}
\prod_{(ij)\in \edges(T)\backslash \{(\parent(h-1),h-1), (\parent(h),h)\}}Q_{x_i,x_j}\times\\
& \quad\quad\quad\quad
\big(
\sum_{x_{h-1}}
p_{x_{h-1}}
Q_{x_{\parent(h-1)}, x_{h-1}}
\big)\times
\big(
\sum_{x_{h}}
p_{x_{h}}
Q_{x_{\parent(h)} x_{h}}
\big)\Bigg)
\Bigg|\\
& \le
\sum_{x_{1},x_2, \ldots, x_{h-2}}\Bigg(
p_{x_{1}}p_{x_2}\cdots p_{x_{h-2}}
\prod_{(ij)\in \edges(T)\backslash \{(\parent(h-1),h-1), (\parent(h),h)\}}|Q_{x_i,x_j}|\times\\
& \quad\quad\quad\quad\times
\big|
\sum_{x_{h-1}}
p_{x_{h-1}}
Q_{x_{\parent(h-1)}, x_{h-1}}
\big|\times
\big|
\sum_{x_{h}}
p_{x_{h}}
Q_{x_{\parent(h)}, x_{h}}
\big|\Bigg)\\
& \lesssim
\sum_{x_{1},x_2, \ldots, x_{h-2}}\Bigg(
\prod_{(ij)\in \edges(T)\backslash \{(\parent(h-1),h-1), (\parent(h),h)\}}|Q_{x_i,x_j}|\times\\
&\quad\quad\quad\quad\times
\bigg(\big(
\sum_{x_{h-1}}
p_{x_{h-1}}
Q_{x_{\parent(h-1)}, x_{h-1}}
\big)^2 +
\big(
\sum_{x_{h}}
p_{x_{h}}
Q_{x_{\parent(h)}, x_{h}}
\big)^2\bigg)\Bigg),
\end{split}
\end{equation}
where we used $|a|\times |b|\le a^2+ b^2$ in the last line.
Now, the product of $|\edges(T)|-2 = |\vertices(T)|-3$ values $Q_{x_i,x_j}$ can be bounded as $|\Fourier_{\SBM(p,Q)}(\cycle_4)|^{\frac{|\vertices(T)|-3}{|\vertices(\cycle_4)|}},$ which is comparable to the signed count of a graph on $|\vertices(T)|-3$ vertices. On the other hand, $\big(
\sum_{x_{h-1}}
p_{x_{h-1}}
Q_{x_{\parent(h-1)}, x_{h-1}}
\big)^2 +
\big(
\sum_{x_{h}}
p_{x_{h}}
Q_{x_{\parent(h)} x_{h}}
\big)^2$ can be easily shown to be bounded by\linebreak  $\Fourier_{\SBM(p,Q)}(\Star_2),$ a graph on $3$ vertices. Altogether, we have managed to show that 
$$
\Big|\Fourier_{\SBM(p,Q)}(T)\Big|\lesssim 
|\Fourier_{\SBM(p,Q)}(\cycle_4)|^{\frac{|\vertices(T)|-3}{|\vertices(\cycle_4)|}}\times 
|\Fourier_{\SBM(p,Q)}(\Star_2)|.
$$
This immediately implies 
$$
\Big|\Fourier_{\SBM(p,Q)}(T)\Big|^{\frac{1}{|\vertices(T)|}}\lesssim \max\Big(
|\Fourier_{\SBM(p,Q)}(\cycle_4)|^{{\frac{1}{|\vertices(\cycle_4)|}}}, 
|\Fourier_{\SBM(p,Q)}(\Star_2)|^{\frac{1}{|\vertices(\Star_2)|}}\Big).
$$

\subsubsection{Main Idea 2: Comparison With a Non-Negative Model}
\label{sec:introcmparewithnonnegative}
The reason why the leaf-isolation technique alone is not sufficient for general SBMs on two communities is that \eqref{eq:intro4cycleinnonvanishing} fails. Namely, one can show (\cref{claim:4cycle2sbm}) that 
\begin{equation}
\label{eq:intro4cycleinnonvanishing2SBM}
\Fourier_{\SBM(p,Q)}(\cycle_4) = \Theta\bigg(\max\Big(p_1^4|Q_{1,1}|^4, p^2_1p_2^2|Q_{1,2}|, p_2^4|Q_{2,2}|^4\Big)\bigg).
\end{equation}  
So, if $p_1$ is very small, say $p_1 = n^{-1/2},$ then the inequality used in  
\eqref{eq:intro4cycle} for non-tree graphs may fail. Indeed, in \cref{examplethm:onetoneinsufficient} we show that there do exist two-community SBMs and connected $H$ which are not trees such that $|\Fourier_{\SBM(p,Q)}(H)|^{\frac{1}{|\vertices(H)|}} = \omega(|\Fourier_{\SBM(p,Q)}(\cycle_4)|^{\frac{1}{|\vertices(\cycle_4)|}}).$

Thus, if we want to utilize 
\eqref{eq:intro4cycleinnonvanishing2SBM} towards the leaf-isolation technique,
we need to understand how the probability vector $(p_1, p_2)$ relates to the entries of $Q.$

\paragraph{Comparison with Non-Negative SBM.} To do so, we compare with a non-negative SBM and utilize what we know from \cite{yu2024counting} (spelled out in \cref{sec:priortechniquesSTARS}). Namely, let $|Q|$ be the matrix $Q$ in which we take entry-wise absolute values. Then, triangle inequality applied to \eqref{eq:explicitfourier} implies that
$$
|\Fourier_{\SBM(p,Q)}(H)|\le 
|\Fourier_{\SBM(p,|Q|)}(H)|.
$$
On the other hand, as explained in \eqref{sec:priortechniquesSTARS}, 
$$ 
|\Fourier_{\SBM(p,|Q|)}(H)|^{\frac{1}{|\vertices(H)|}}\le 
|\Fourier_{\SBM(p,|Q|)}(\Star_\tspecial)|^{\frac{1}{|\vertices(\Star_\tspecial)|}}
$$
for some $\tspecial\in [1, |\edges(H)|].$ Combining the last two inequalities, we conclude that 
$$|\Fourier_{\SBM(p,Q)}(H)|^{\frac{1}{|\vertices(H)|}}\le 
|\Fourier_{\SBM(p,|Q|)}(\Star_\tspecial)|^{\frac{1}{|\vertices(\Star_\tspecial)|}}.
$$
Hence, the desired conclusion for 2-SBMs in \cref{thm:mainintro} would be implied by 
\begin{align*}
    & |\Fourier_{\SBM(p,|Q|)}(\Star_\tspecial)|
     \lesssim
    |\Fourier_{\SBM(p,Q)}(\Star_\tspecial)| \Longleftrightarrow\\
    & p_1(p_1|Q_{1,1}| + p_2|Q_{1,2}|)^\tspecial + 
    p_2(p_1|Q_{1,2}| + p_2|Q_{2,2}|)^\tspecial\lesssim 
    \Big|
    p_1(p_1Q_{1,1} + p_2Q_{1,2})^\tspecial + 
    p_2(p_1Q_{1,2} + p_2Q_{2,2})^\tspecial
    \Big|.
\end{align*}
Unfortunately, such an inequality is wrong as demonstrated, for example, by \cref{examplethm:notodddegree}. Yet, whenever this inequality is violated, there must be certain cancellations in $\Fourier_{\SBM(p,|Q|)}(\Star_\tspecial) = 
 p_1(p_1Q_{1,1} + p_2Q_{1,2})^\tspecial + 
    p_2(p_1Q_{1,2} + p_2Q_{2,2})^\tspecial.
$
The key question, used to address \cref{challenge:iterativetechniquefails}, is: \emph{What causes the cancellations?}
It turns out that the following
\emph{dichotomy} occurs. Exactly one of the following two types of cancellations occurs:
\begin{enumerate}
    \item \emph{Within-community cancellations:} when there is a cancellation within community 1, corresponding to $|p_1Q_{1,1} + p_2Q_{1,2}| = o(p_1|Q_{1,1}| + p_2|Q_{1,2}|)$ or, similarly in community 2, $|p_1Q_{1,2} + p_2Q_{2,2}| = o(p_1|Q_{1,2}| + p_2|Q_{2,2}|)$ .
    \item \emph{Between-community cancellations:} when there are no in-community cancellations, but 
    there is a cancellation between $ p_1(p_1Q_{1,1} + p_2Q_{1,2})^\tspecial$ and
    $p_2(p_1Q_{1,2} + p_2Q_{2,2})^\tspecial.$
\end{enumerate}
In either case, the cancellations imply strong relationships between $Q$ and $p,$ which help us boost \eqref{eq:intro4cycleinnonvanishing2SBM} and then apply an appropriate version of the leaf-isolation technique.
For example, if $p_1\le p_2$ (so, also $p_2 \ge 1/2$) in the case of within-community cancellations we can show that $p_2|Q_{1,2}| = O(p_1|Q_{1,1}|),$ which turns out to be sufficient for making the leaf-isolation technique to work. The reason such an inequality is useful is the following. As $p_2\ge 1/2,$ it implies that $|Q_{1,2}| =O(p_1|Q_{1,1}|).$ This means that when applying the triangle-inequality approach \eqref{eq:intro4cycle} (or the more sophisticated leaf-isolation variation of it, \eqref{intro:leafisolation}), we can replace some of the $Q_{1,2}$ occurrences with $p_{1}|Q_{1,1}|,$ thus, adding more $p_1$ appearances in the expression. As \eqref{eq:intro4cycleinnonvanishing2SBM} reads 
$\Fourier_{\SBM(p,Q)}(\cycle_4) = \Theta\bigg(\max\Big(p_1^4|Q_{1,1}|^4, p^2_1|Q_{1,2}|, |Q_{2,2}|^4\Big)\bigg)$ when $p_2\ge 1/2,$ the extra occurrences of $p_1$ are key to a comparison with $\Fourier_{\SBM(p,Q)}(\cycle_4).$

\section{Preliminaries}
\subsection{Notation}
\paragraph{Graph Notation.} All graphs in the current paper are undirected and have no loops or multiple edges. For a graph $H,$ its vertex set and edge sets are denoted by $\vertices(H)$ and $\edges(H).$ Whenever $A,B\subseteq\vertices(H)$ for some graph $H,$ we denote by $\edges_{H}(A,B)$ the subset of $\edges(H)$ with one endpoint in $A$ and one endpoint in $B.$ For $S\subseteq \vertices(H),$ we denote by $H|_S$ the graph on vertex set $S$ and edge-set $\edges_H(S,S).$

We denote by $\cycle_t$ the cycle on $t$ vertices, by $\complete_t$ the complete graph on $t$ vertices, by $\complete_{t,s}$ the complete bipartite graph on parts with $t$ and $s$ vertices, by $\Star_t$ the star graph on $t+1$ vertices (so $\complete_{1,t}= \Star_t$). In $\Star_t,$ there is one central vertex of degree $t,$ which is adjacent to $t$ leaves. Note that $\Star_1$ is simply an edge and $\Star_2$ is the path of length 2 (also called wedge).

For two graphs $H,K,$ we denote by $H\sqcup K$ their disjoint union. If $H$ and $K$ are labeled, denote by $H\otimes K$ the graph in which $(i,j)$ is an edge if and only if it is an edge in exactly one of $H$ and $K,$and furthemrore all isolated vertices are removed. 

For two graphs $H$ and $K,$ we denote by $H\sim K$ the fact that they are isomorphic.

Denote by
$\graphs_{\le D}$ the set of graphs without isolated vertices and at most $D$ edges and by
$\connectedgraphs_{\le D}$ those of them that are furthermore connected.

\paragraph{Asymptotic Notation.} For two quantities $x(n),y(n)\in \mathbb{R}$ depending on $n\in \mathbb{N},$ we denote $x\gtrsim y$ if $x\ge Cy$ for some absolute constant $C>0.$ If $x\gtrsim y$ and $y\gtrsim x,$ we denote $x\asymp y.$
We similarly denote $x\gtrsim_D y$ if $x\ge C(D)y$ where the constant $C(D)$ can be an arbitrary function of $D$ (but nothing else) and likewise $x\asymp_D y$. We also denote by $y = O(x)$ and 
$x = \Omega(y)$
the fact $x \gtrsim y.$ Finally, $y = o_n(x)$ and $x = \omega_n(y)$ denote 
the fact that $\lim_{n\longrightarrow + \infty} y/x = 0.$ In most places, we omit the dependence on $n.$

\subsection{Testing via Low Degree Polynomials}
\label{sec:ldpprelim}
Condition \eqref{eq:meanseparationforsignedcounts} is a special case of testing between two distributions via low-degree polynomial tests. Since the indicator of edges of a random graph are random variables, one naturally defines low-degree polynomial tests as follows, again exploiting Chebyshov's ineqality.
\begin{definition}[\cite{hopkins2017bayesian,hopkins18} specialized to graphs] 
\label{def:lowdegreesuccess}
Consider the families of graph distributions $(\mathbb{P}_n)_{n\in \mathbb{N}}$ and $\ergraphhalf_{n \in \mathbb{N}}.$ We say that a polynomial $f:\{0,1\}^{\binom{n}{2}}\longrightarrow \mathbb{R}$ distinguishes the two families of distributions if 
$$
\Big|
\expect_{\bfG\sim \mathbb{P}_n}f(\bfG) - 
\expect_{\bfG\sim \ergraphhalf}f(\bfG)
\Big| = \omega\Bigg( \max\Big(
\var_{\bfG\sim \ergraphhalf}[f(\bfG)]^{1/2},
\var_{\bfG\sim \mathbb{P}_n}[f(\bfG)]^{1/2}\Big)\Bigg).
$$
Conversely, if a polynomial $f$ of degree at most $D$ satisfies
$$
\Big|
\expect_{\bfG\sim \mathbb{P}_n}f(\bfG) - 
\expect_{\bfG\sim \ergraphhalf}f(\bfG)
\Big| = o\Bigg( \max\Big(
\var_{\bfG\sim \ergraphhalf}[f(\bfG)]^{1/2},
\var_{\bfG\sim \mathbb{P}_n}[f(\bfG)]^{1/2}\Big)\Bigg),
$$
we say that $f$ fails to distinguish the two distributions. If all polynomials of degree at most $D$ fail,  we say that
$\ergraphhalf$ and $\mathbb{P}_n$ are degree-$D$ indistinguishable.
\end{definition}

Polynomials tests of degree $D = O(\log n)$ capture a wide range of computationally efficient tests such as constant-sized (signed) subgraph counts, certain
spectral methods \cite{kunisky2019notes}, 
statistical-query algorithms \cite{brennan2021statisticalquery}, and approximate message passing \cite{montanari2022equivalence} algorithms among others. Due to the variety of methods captured by low-degree polynomials, hardness against degree $D$ polynomials for some $D = \omega(\log n)$ is viewed as a strong (but, of course, still limited) heuristic for the computational hardness of a problem \cite{hopkins2017efficient,hopkins18}.

When testing against $\ergraphhalf,$ it is convenient to express polynomials in terms of the Fourier characters $\prod_{(ji)\in E(H)}(2G_{ji}-1).$ The key simple fact from Boolean Fourier analysis is the following.

\begin{theorem}[Folklore]
\label{thm:orthofourier}
The set of polynomials $\Bigg\{\prod_{(ji)\in E(H)}(2G_{ji}-1)\Bigg\}_{H}$ where $H$ ranges over all graphs without isolated vertices in $\complete_n$ forms an orthonormal basis of all edge-functions over $n$-vertex graphs with respect to the $\ergraphhalf$ distribution. 
\end{theorem}

\subsection{Simple Facts about Fourier Coefficients of Stochastic Block Models}
\begin{proposition}
\label{prop:vertexdisjointgrahs}
Consider any stochastic block model $\SBM(p, Q)$ and let $H, K$ be any two graphs. Then, 
$$
\Fourier_{\SBM(p,Q)}(H\sqcup K) = 
\Fourier_{\SBM(p,Q)}(H)\times 
\Fourier_{\SBM(p,Q)}(K).
$$
\end{proposition}
\begin{proof}
Since there are no shared vertices between the copies of $H,K$ in $H\sqcup K$ and labels are independent,
\begin{align*}
    & \Fourier_{\SBM(p,Q)}(H\sqcup K) = 
    \expect\Bigg[\prod_{(i,j)\in \edges(H\cup K)}Q_{\bfx_i\bfx_j}\Bigg]\\
    & =
    \expect\Bigg[
    \prod_{(i,j)\in \edges(H)}Q_{\bfx_i\bfx_j}\times 
    \prod_{(u,v)\in \edges(H)}Q_{\bfx_u\bfx_v}
    \Bigg]\\
    & = 
    \expect\Bigg[
    \prod_{(i,j)\in \edges(H)}Q_{\bfx_i\bfx_j}\Bigg]\times 
    \expect\Bigg[
    \prod_{(u,v)\in \edges(H)}Q_{\bfx_u\bfx_v}
    \Bigg] = \Fourier_{\SBM(p,Q)}(H)\times 
\Fourier_{\SBM(p,Q)}(K).\qedhere
\end{align*}
\end{proof}

\begin{corollary}
\label{cor:vertexdisjointcomparison}
Consider any $\SBM(p, Q)$ distribution and let $H, K$ be any two graphs. Then, 
$$
\Psi_{\SBM(p,Q)}(H\sqcup K)
\le 
\max\Big(
\Psi_{\SBM(p,Q)}(H),
\Psi_{\SBM(p,Q)}(K)
\Big).
$$ 
\end{corollary}
\begin{proof} Recalling the definition of $\Psi,$
    \begin{align*}
       & |\Fourier_{\SBM(p,Q)}(H\sqcup K)|^{\frac{1}{|\vertices(H\sqcup K)|}}\\
       & = 
       \Big(|\Fourier_{\SBM(p,Q)}(H)|\times |\Fourier_{\SBM(p,Q)}( K)|\Big)^{\frac{1}{|\vertices(K)| + |\vertices(H)|}}\\
       & = 
       \Big(|\Fourier_{\SBM(p,Q)}(H)|^{\frac{1}{|\vertices(H)|}}\Big)^{\frac{|\vertices(H)|}{|\vertices(K)| + |\vertices(H)|}}\times 
       \Big(|\Fourier_{\SBM(p,Q)}(K)|^{\frac{1}{|\vertices(K)|}}\Big)^{\frac{|\vertices(K)|}{|\vertices(K)| + |\vertices(H)|}}\\
       & \le 
       \max\Big(
|\Fourier_{\SBM(p,Q)}(H)|^{\frac{1}{|\vertices(H)|}},
|\Fourier_{\SBM(p,Q)}(K)|^{\frac{1}{|\vertices(K)|}}
\Big)^{\frac{|\vertices(H)|}{|\vertices(K)| + |\vertices(H)|}}\times\\
& \quad\quad\quad\quad\times
\max\Big(
|\Fourier_{\SBM(p,Q)}(H)|^{\frac{1}{|\vertices(H)|}},
|\Fourier_{\SBM(p,Q)}(K)|^{\frac{1}{|\vertices(K)|}}
\Big)^{\frac{|\vertices(K)|}{|\vertices(K)| + |\vertices(H)|}}\\
& = \max\Big(
|\Fourier_{\SBM(p,Q)}(H)|^{\frac{1}{|\vertices(H)|}},
|\Fourier_{\SBM(p,Q)}(K)|^{\frac{1}{|\vertices(K)|}}
\Big).\qedhere
    \end{align*}
\end{proof}

\begin{corollary}[Star Counts]
\label{cor:starcounts}
For any $\SBM(p,Q)$ model on $k$ communities
$$
\Fourier_{\SBM(p,Q)}
(\Star_t) = 
\sum_{x\in [k]}p_x\times \big(\sum_{y \in [k]}Q_{x,y}p_y\big)^t.
$$
\end{corollary}
\begin{proof} Follows from \cref{prop:explicitFourier} by first summing over the central vertex.
\end{proof}

\subsection{Meta Theorems on Signed Subgraph Counts}

\begin{theorem}[Beating Null Variance by A Low Degree Polynomial]
\label{thm:ldpimpliesbsubgraphcount}
Suppose that $(\mathbb{P}_n)_{n\in \mathbb{N}}$ is a family of vertex-symmetric distributions over $n$-vertex graphs and 
there exists some polynomial $f$ of constant degree $D$ such that 
$$
|\expect_{\mathbb{P}_n}[f] - \expect_{\ergraphhalf}[f]| = \omega(\var_{\ergraphhalf}[f]^{1/2}).
$$
Then, there exists some subgraph $H$ on at most $D$ vertices such that 
$\Psi_{\SBM(p,Q)}(H) = \omega(n^{-1/2}).$
\end{theorem}
\begin{proof} 
Recall that $\graphs_{\le D}$ is the set of graphs without isolated vertices on at most $D$ edges.
Consider the polynomial 
$$
f(G) =\sum_{H \in \graphs_{\le D}}\sum_{H_1\subseteq K_n \; : \; H\sim H_1} c_{H_1}\prod_{(ji)\in \edges(H_1)}(2G_{ji} - 1).
$$
Then, by orthonormality under the $\ergraphhalf$ distribution (\cref{thm:orthofourier}),
\begin{align}
    & \var_{\ergraphhalf}[f(\bfG)] = \sum_{H \in \graphs_{\le D}}\sum_{H_1\subseteq K_n \; : \; H\sim H_1} c^2_{H_1}.
\end{align}
Hence, as $D$ is constant and there are constantly many graphs without isolated vertices and at most $D$ edges, it follows that 
\begin{align}
    \var_{\ergraphhalf}[f(\bfG)]^{1/2} &=\Bigg(\sum_{H \in \graphs_{\le D}}\sum_{H_1\subseteq K_n \; : \; H\sim H_1} c^2_{H_1}\Bigg)^{1/2}\\
    & = 
    \Theta\Bigg(
    \sum_{H \in \graphs_{\le D}}\Big(\sum_{H_1\subseteq K_n \; : \; H\sim H_1} c^2_{H_1}\Big)^{1/2}
    \Bigg).
\end{align}
At the same time, 
\begin{align*}
    & \Big|\expect_{\mathbb{P}_n}[f] - \expect_{\ergraphhalf}[f]\Big|\\
    & = \Big|\sum_{H \in \graphs_{\le D}}\sum_{H_1\subseteq K_n \; : \; H\sim H_1} c_{H_1}\prod_{(ji)\in \edges(H_1)}\expect_{\mathbb{P}_n}(2\bfG_{ji} - 1)\Big|\\
    & = 
    \Big|\sum_{H \in \graphs_{\le D}}\sum_{H_1\subseteq K_n \; : \; H\sim H_1} c_{H_1}\Fourier_{\mathbb{P}_n}(H)\Big|\\
    & \le 
    \sum_{H \in \graphs_{\le D}}
    \Big|\Fourier_{\mathbb{P}_n}(H)\Big|\times
    \sum_{H_1\subseteq K_n \; : \; H\sim H_1} |c_{H_1}|\\
        & \le 
    \sum_{H \in \graphs_{\le D}}
    \Big|\Fourier_{\mathbb{P}_n}(H)\Big|\times
    \frac{\Big({\sum_{H_1\subseteq K_n \; : \; H\sim H_1} c_{H_1}^2}\Big)^{1/2}}{\Big({\sum_{H_1\subseteq K_n \; : \; H\sim H_1} 1}\Big)^{1/2}}\quad\quad\quad\quad\quad(\text{Cauchy-Schwartz})\\
    & = \Theta\Bigg(
     \sum_{H \in \graphs_{\le D}}
    \Big|\Fourier_{\mathbb{P}_n}(H)\Big|\times
    \frac{\Big({\sum_{H_1\subseteq K_n \; : \; H\sim H_1} c_{H_1}^2}\Big)^{1/2}}{n^{|\vertices(H)|/2}}
    \Bigg).
\end{align*}
Hence, $
|\expect_{\mathbb{P}_n}[f] - \expect_{\ergraphhalf}[f]| = \omega(\var_{\ergraphhalf}[f]^{1/2})
$ implies that 
\begin{align*}
    & \sum_{H \in \graphs_{\le D}}
    |\Fourier_{\mathbb{P}_n}(H)|\times
    \frac{\Big({\sum_{H_1\subseteq K_n \; : \; H\sim H_1} c_{H_1}^2}\Big)^{1/2}}{n^{|\vertices(H)|/2}} = 
    \omega\Bigg(
    \sum_{H \in \graphs_{\le D}}\Big(\sum_{H_1\subseteq K_n \; : \; H\sim H_1} c^2_{H_1}\Big)^{1/2}
    \Bigg).
\end{align*}
As the sum is over constantly many terms, at most $2^{O(D\log D)}$ (which is an upper bound on the number of graphs without isolated vertices and at most $D$ edges), there exists some $H$ such that $\sum_{H_1\subseteq K_n \; : \; H\sim H_1} c_{H_1}^2\neq 0$ and
$$
|\Fourier_{\mathbb{P}_n}(H)|\times
    \frac{\Big({\sum_{H_1\subseteq K_n \; : \; H\sim H_1} c_{H_1}^2}\Big)^{1/2}}{n^{|\vertices(H)|/2}} = 
    \omega\Bigg(
    \Big(\sum_{H_1\subseteq K_n \; : \; H\sim H_1} c^2_{H_1}\Big)^{1/2}\Bigg).
$$
the claim follows.\end{proof}
Recall that the reason \eqref{eq:statisticalscaling} is not sufficient to conclude that a given signed subgraph count is a good tester is that \eqref{eq:statisticalscaling} does not take into account the variance under the planted distribution. Somewhat surprisingly, it turns out that the variance of the planted model is actually captured by this equation by the maximizer of $K\longrightarrow \Psi_{\SBM}(K).$

\begin{theorem}[Beating Variance of Null Implies Beating Variance of Planted]
\label{thm:beatingvarnullimpliesbeatingvarplanted}
Suppose that $\SBM(n; p, Q)$ is any SBM model and let $D$ be any constant (independent of $n$). Suppose that the connected graph $H$ on at most $D$ vertices satisfies the following two properties:
\begin{enumerate}
    \item $\Psi_{\SBM(p,Q)}(H) = \omega(n^{-1/2}).$
    \item $H$ is an approximate maximizer of $K \longrightarrow \Psi_{\SBM(p,Q)}(K)$ in the following sense. 
$$
\Psi_{\SBM(p,Q)}(H)\gtrsim_D
\Psi_{\SBM(p,Q)}(K)\text{ for any }K \text{on at most $2D$ edges.}
$$
\end{enumerate}
Then, one can test between $\SBM(n;p, Q)$ and $\ergraphhalf$ using the signed $H$ count, namely
$$
\Big|
\expect_{\SBM(n; p, Q)}
[\signedcount_H] - 
\expect_{\ergraphhalf}
[\signedcount_H]
\Big| = 
\omega\Big(\max(
\var_{\SBM(n; p, Q)}
[\signedcount_H]^{1/2},
\var_{\ergraphhalf}
[\signedcount_H]^{1/2})
\Big).
$$
\end{theorem}
\begin{proof}
Due to the $\Theta(n^{|\vertices(H)|})$ isomorphic copies of $H,$
\begin{align}
\label{eq:signedcountsasfourier}
\Big|
\expect_{\SBM(n; p, Q)}
[\signedcount_H] - 
\expect_{\ergraphhalf}
[\signedcount_H]
\Big| = 
\Theta(n^{|\vertices(H)|}\times |\Fourier_{\SBM(p, Q)}(H)|).
\end{align}
Hence, $|\Fourier_{\SBM(p, Q)}(H)|^{\frac{1}{|\vertices(H)|}} = \omega(n^{-1/2})$ immediately implies that 
$$
\Big|
\expect_{\SBM(n; p, Q)}
[\signedcount_H] - 
\expect_{\ergraphhalf}
[\signedcount_H]
\Big| = \omega(n^{|\vertices(H)|/2}) = 
\omega(\var_{\ergraphhalf}
[\signedcount_H]^{1/2}).
$$
Thus, all we need to show is that 
\begin{align*}
    & \Big|
\expect_{\SBM(n; p, Q)}
[\signedcount_H] - 
\expect_{\ergraphhalf}
[\signedcount_H]
\Big| = 
\omega(\var_{\SBM(p,Q)}
[\signedcount_H]^{1/2}).
\end{align*}
For the left hand-side, we use \eqref{eq:signedcountsasfourier}. The right hand-side we expand as follows:
\begin{align*}
    & \var_{\SBM(p,Q)}
[\signedcount_H]\\
& = \var_{\SBM(p,Q)}
\Bigg[\sum_{H_1\subseteq K_n\; : \; H_1\sim H} 
\prod_{(ji)\in \edges(H_1)}(2\bfG_{ji} - 1)\Bigg]\\
& = \sum_{H_1,H_2\subseteq K_n\; : \; H_1\sim H, H_2\sim H}
\cov_{\SBM(p,Q)}\Bigg[\prod_{(ji)\in \edges(H_1)}(2\bfG_{ji} - 1), \prod_{(ji)\in \edges(H_2)}(2\bfG_{ji} - 1)\Bigg].
\end{align*}
Note that if $H_1, H_2$ don't have a common vertex, the covariance is clearly zero.
Hence, the above expression equals 
\begin{align*}
    & \sum_{H_1,H_2\subseteq K_n\; : \; H_1\sim H, H_2\sim H, \vertices(H_1)\cap \vertices(H_2)\neq \emptyset}
\cov_{\SBM(p,Q)}\Bigg[\prod_{(ji)\in \edges(H_1)}(2\bfG_{ji} - 1), \prod_{(ji)\in \edges(H_2)}(2\bfG_{ji} - 1)\Bigg]\\
& = \sum 
\sum_{H_1,H_2\subseteq K_n\; : \; H_1\sim H, H_2\sim H, \vertices(H_1)\cap \vertices(H_2)\neq \emptyset} 
\Bigg|
\expect\Bigg[\prod_{(ji)\in \edges(H_1)}(2\bfG_{ji} - 1)\prod_{(ji)\in \edges(H_2)}(2\bfG_{ji} - 1)\Bigg]\\
&\quad\quad\quad\quad- 
\expect\Bigg[\prod_{(ji)\in \edges(H_1)}(2\bfG_{ji} - 1)\Bigg]\expect\Bigg[\prod_{(ji)\in \edges(H_2)}(2\bfG_{ji} - 1)\Bigg] 
\Bigg|\\
&\le 
\sum_{H_1,H_2\subseteq K_n\; : \; H_1\sim H, H_2\sim H, \vertices(H_1)\cap \vertices(H_2)\neq \emptyset} \Bigg(
\Bigg|
\expect\Bigg[\prod_{(ji)\in \edges(H_1)}(2\bfG_{ji} - 1)\prod_{(ji)\in \edges(H_2)}(2\bfG_{ji} - 1)\Bigg]\Bigg|\\
&\quad\quad\quad\quad +
\expect\Bigg[\prod_{(ji)\in \edges(H)}(2\bfG_{ji} - 1)\Bigg]^2\Bigg)
\\
& = 
\sum_{H_1,H_2\subseteq K_n\; : \; H_1\sim H, H_2\sim H, \vertices(H_1)\cap \vertices(H_2)\neq \emptyset}\Bigg(
\Bigg|
\expect\Bigg[\prod_{(ji)\in \edges(H_1)\triangle \edges(H_2)}(2\bfG_{ji} - 1)\Bigg]\Bigg|\\
&\quad\quad\quad\quad\quad\quad+ 
\expect\Bigg[\prod_{(ji)\in \edges(H)}(2\bfG_{ji} - 1)\Bigg]^2\Bigg).\\
\end{align*}
First, we will bound the term 
$$
\sum_{H_1,H_2\subseteq K_n\; : \; H_1\sim H, H_2\sim H, \vertices(H_1)\cap \vertices(H_2)\neq \emptyset}
\expect\Bigg[\prod_{(ji)\in \edges(H)}(2\bfG_{ji} - 1)\Bigg]^2.
$$
Note that there are $O(n^{2|\vertices(H)| -1})$ terms in the sum. This is the case since $|\vertices(H_1)\cup\vertices(H_2)|\le2|\vertices(H)|-1$ whenever $ \vertices(H_1)\cap \vertices(H_2)\neq \emptyset.$ Furthermore, each of these terms is $|\Fourier_{\SBM(p, Q)}(H)|^2.$ Altogether,
\begin{align*}
& \sum_{H_1,H_2\subseteq K_n\; : \; H_1\sim H, H_2\sim H, \vertices(H_1)\cap \vertices(H_2)\neq \emptyset}
\expect\Bigg[\prod_{(ji)\in \edges(H)}(2\bfG_{ji} - 1)\Bigg]^2\\
& =
O\Big(n^{2|\vertices(H)| -1}|\Fourier_{\SBM(p, Q)}(H)|^2\Big) = 
o\Big(\Big|
\expect_{\SBM(n; p, Q)}
[\signedcount_H] - 
\expect_{\ergraphhalf}
[\signedcount_H]
\Big|^2\Big)
\end{align*}
by \eqref{eq:signedcountsasfourier}.
Next, consider 
$$
\sum_{H_1,H_2\subseteq K_n\; : \; H_1\sim H, H_2\sim H, \vertices(H_1)\cap \vertices(H_2)\neq \emptyset}
\Bigg|
\expect\Bigg[\prod_{(ji)\in \edges(H_1)\triangle \edges(H_2)}(2\bfG_{ji} - 1)\Bigg]\Bigg|.
$$
Note that the graph defined by edges $\edges(H_1)\triangle \edges(H_2)$ is exactly $H_1\otimes H_2.$ Hence, the above sum is 
\begin{align}
\label{eq:overlappingequation}
\sum_{H_1,H_2\subseteq K_n\; : \; H_1\sim H, H_2\sim H, \vertices(H_1)\cap \vertices(H_2)\neq \emptyset}
\big|
\Fourier_{\SBM(p,Q)}(H_1\otimes H_2)
\big|.
\end{align}
Denote the following cardinalities:
\begin{equation}
\begin{split}
\label{eq:overlappingvertexsets}
& s_\emptyset = |(\vertices(H_1)\cup\vertices(H_2))\backslash \vertices(H_1\otimes H_2)|,\\
& s_1= |(\vertices(H_1)\cap \vertices(H_1\otimes H_2))\backslash \vertices(H_2)|,\\
& s_2= |(\vertices(H_2)\cap \vertices(H_1\otimes H_2))\backslash \vertices(H_1)|,\\
&s_{1,2} = |\vertices(H_1)\cap\vertices(H_2)\cap \vertices(H_1\otimes H_2)|
\end{split}
\end{equation}
A simple count shows that
\begin{equation}
\label{eq:setcardnalities}
\begin{split}
   & s_1+ s_2+ s_{1,2} = |\vertices(H_1\otimes H_2)|,\\
   & s_\emptyset + s_1 + s_{1,2} = |\vertices(H_1)| = |\vertices(H_2)|= 
s_\emptyset + s_2 + s_{1,2},\\
& s_\emptyset + s_1 + s_2 + s_{1,2} = |\vertices(H_1)\cup\vertices(H_2)|.
\end{split}
\end{equation}
The number of ways to choose $H_1,H_2\sim H$ so that $|\vertices(H_1)\cup\vertices(H_2)| = s_\emptyset + s_1 + s_2+s_{1,2}$ is $O(n^{s_\emptyset + s_1 + s_2 + s_{1,2}}).$  Hence, the number of terms in 
\eqref{eq:overlappingequation} with the specific overlap pattern of $H_1,H_2$ satisfying \eqref{eq:overlappingvertexsets} is $O(n^{s_\emptyset + s_1 + s_2 + s_{1,2}}).$ As there are constantly many choices of $s_\emptyset,s_1,s_2,s_{1,2}$ (as $H_1,H_2$ are two graphs on constantly many vertices), recalling \eqref{eq:signedcountsasfourier}, it is enough to show that for any choice for $s_\emptyset,s_1, s_2, s_{1,2},$
\begin{align*}
n^{2|\vertices(H)|}\times |\Fourier_{\SBM(p, Q)}(H)|^2= 
\omega\Bigg(
n^{s_\emptyset + s_1 + s_2 + s_{1,2}}
|\Fourier_{\SBM(p,Q)}(H_1\otimes H_2)|
\Bigg).
\end{align*}
Using the fact that $|\edges(H_1\otimes H_2)| = |\edges(H_1)\triangle\edges(H_2)|\le |\edges(H_1)|+|\edges(H_2)|\le 2D$ by condition 2., the approximate optimality of $H,$ it follows that 
$$
|\Fourier_{\SBM(p,Q)}(H_1\otimes H_2)|^{\frac{1}{|\vertices(H_1\otimes H_2)|}}=O(|\Fourier_{\SBM(p,Q)}(H)|^{\frac{1}{|\vertices(H)|}}).
$$
Hence, it is enough to show that 
\begin{align}
\label{eq:onepossibleoverlap}
n^{2|\vertices(H)|}\times |\Fourier_{\SBM(p, Q)}(H)|^2= 
\omega\Bigg(
n^{s_\emptyset + s_1 + s_2 + s_{1,2}}
|\Fourier_{\SBM(p, Q)}(H)|^{\frac{|\vertices(H_1\otimes H_2)|}{|\vertices(H)|}}
\Bigg).
\end{align}
Recalling \eqref{eq:setcardnalities} and the fact that 
$|\vertices(H)| = \frac{1}{2}(|\vertices(H_1)|+ |\vertices(H_2)|),$ we need to show that 
\begin{align*}
& n^{s_1 + s_2 + 2s_\emptyset+ 2s_{1,2}}\times 
 |\Fourier_{\SBM(p, Q)}(H)|^2= 
\omega\Bigg(
n^{s_\emptyset + s_1 + s_2 + s_{1,2}}
|\Fourier_{\SBM(p, Q)}(H)|^{\frac{s_1 + s_2 + s_{1,2}}{(s_1 + s_2 + 2s_\emptyset+ 2s_{1,2})/2}}\Bigg)\quad \Longleftrightarrow\\
& 
|\Fourier_{\SBM(p, Q)}(H)|^{\frac{2(s_1 + s_2 + 2s_\emptyset + 2s_{1,2})- 
2(s_1 + s_2 + s_{1,2})
}{s_1 + s_2 + 2s_\emptyset + 2s_{1,2}}}
= 
\omega(
n^{-s_\emptyset -s_{1,2}}
)\quad \Longleftrightarrow\\
& 
|\Fourier_{\SBM(p, Q)}(H)|^{\frac{4s_\emptyset + 2s_{1,2}}{2|\vertices(H)|}}=
\omega(
n^{-s_\emptyset -s_{1,2}}
).
\end{align*}
This last inequality follows from condition 1. that
$|\Fourier_{\SBM(p, Q)}(H)|{\frac{1}{|\vertices(H)|}} = \omega(n^{-\frac{1}{2}}).$ Concretely,  
$$
|\Fourier_{\SBM(p, Q)}(H)|^{\frac{4s_\emptyset + 2s_{1,2}}{2|\vertices(H)|}}\ge 
|\Fourier_{\SBM(p, Q)}(H)|^{\frac{4s_\emptyset + 4s_{1,2}}{2|\vertices(H)|}} =
\Bigg(|\Fourier_{\SBM(p, Q)}(H)|^{\frac{1}{|\vertices(H)|}}\Bigg)^{2(s_\emptyset + s_{1,2})} = 
\omega(
n^{-s_\emptyset -s_{1,2}}
).
$$
Thus, we have shown \eqref{eq:onepossibleoverlap} for one possible choice of $s_\emptyset, s_1, s_2,s_{1,2}$ in \eqref{eq:overlappingvertexsets}. Enumerating over the constantly many possible choices for $s_\emptyset, s_1, s_2,s_{1,2}$ gives the result. 
\end{proof}

\section{Proofs of Main Results}
\subsection{Diagonal SBMs}
In this section, we study the family of $\SBM(p,Q)$ models where all off-diagonal entries of $Q$ are non-zero. This is equivalent to saying that whenever two vertices have distinct labels, the probability they are adjacent is exactly $1/2.$

\begin{theorem}[Maximizing Partition Functions in Diagonal SBMs]
\label{thm:maximizepartitiondiagonal}
Suppose that $\SBM(p,Q)$ is such that $Q$ is diagonal. Then,
for any connected graph $H,$ 
$$
\Psi_{\SBM(p,Q)}(H)
\le 
\max(
\Psi_{\SBM(p,Q)}(\Star_1),
\Psi_{\SBM(p,Q)}(\Star_2)
).
$$
\end{theorem}
\begin{proof} Take any connected $H.$ Recall \cref{prop:explicitFourier} stating that
$$
\Fourier_{\SBM(p,Q)}(H)= 
\sum_{x_1, x_2, \ldots, x_h\in [k]}\Bigg(
\prod_{i = 1}^h p_{x_i}
\times
\prod_{(i,j)\in \edges(H)}
Q_{x_i,x_j}\Bigg).
$$
Observe that since $Q_{x_i,x_j} = 0$ whenever $x_i\neq x_j,$ the expression becomes 
$$
\Fourier_{\SBM(p,Q)}(H)= 
\sum_{x\in [k]}
(p_x)^{|\vertices(H)|}
\times (Q_{x,x})^{|\edges(H)|}.
$$
In particular, for any graph $H,$ it is clearly the case that 
\begin{equation}
\begin{split}
\label{eq:triangleineqfourier}
     |\Fourier_{\SBM(p,Q)}(H)|
    & =
    \Big|\sum_{x\in [k]}
(p_x)^{|\vertices(H)|}
\times (Q_{x,x})^{|\edges(H)|}\Big|\\
& \le 
\sum_{x\in [k]}
(p_x)^{|\vertices(H)|}
\times 
|Q_{x,x}|^{|\edges(H)|}\\
& \le 
\sum_{x\in [k]}
(p_x)^{|\vertices(H)|}
\times 
|Q_{x,x}|^{|\vertices(H)|-1},
\end{split}
\end{equation}
where in the last line we used the fact that any connected graph on $v$ vertices has at least $v - 1$ edges. 
Now, to finish the proof, we will use the fact that
\begin{align}
\label{eq:decreasinginv}
v\longrightarrow (\sum_{x\in [k]}
(p_x)^{v}
\times 
|Q_{x,x}|^{v-1})^{1/v}
\end{align}
is decreasing on $[1,+\infty)$. This is a simple analytic inequality which we prove in \cref{sec:smpleinequality}.

\eqref{eq:decreasinginv} implies that for any graph on at least three vertices (and there exists a unique connected graph on two vertices, $\Star_1$), it is the case that 
\begin{equation*}
 |\Fourier_{\SBM(p,Q)}(H)|^{\frac{1}{|\vertices(H)|}}\le 
 (\sum_{x\in [k]}
(p_x)^{3}
\times 
|Q_{x,x}|^{2})^{\frac{1}{3}} = 
 |\Fourier_{\SBM(p,Q)}(\Star_2)|^{\frac{1}{|\vertices(\Star_2)|}}.\qedhere
\end{equation*}
\end{proof}
Combining with \cref{thm:beatingvarnullimpliesbeatingvarplanted}, we obtain a result for the optimal constant degree distinguisher.

\begin{theorem}[Testing in Diagonal SBMs]
Suppose that $\SBM(p,Q)$ is a stochastic block model with a diagonal matrix $Q.$ Then, if there exists a constant degree test distinguishing $\ergraphhalf$ and 
$\SBM(n;p,Q)$ with high probability, one can also distinguish the two distributions with high probability using the signed $\Star_1$ (edge) or $\Star_2$ (wedge) count.
\end{theorem}

\subsection{Non-Negative SBMs}
The case of non-negative SBMs is covered by \cite{yu2024counting}. Let us explain how to relate a non-negative SBM and the more general \cref{thm:countingstars} of \cite{yu2024counting}. This relation also appears in the ``planted dense subgraph application'' of the main result of \cite{yu2024counting}. Namely, let $\SBM(p,Q)$ be a non-negative SBM. That, is $Q_{i,j}\ge 0 \quad \forall i,j.$ Then, the probability of an edge between vertices with labels $i$ and $j$ is $(Q_{i,j}+1)/2 = \frac{1}{2} + \frac{Q_{i,j}}{2}.$ The following procedure generates a sample from $\SBM(p,Q):$
\begin{enumerate}
    \item First, draw $n$ iid labels $\bfx_1, \ldots, \bfx_n\iidsim p.$
    \item Then, draw a graph $H'$ on vertex set $[n]$ by drawing an edge between $i,j$ with probability $Q_{\bfx_i, \bfx_j}.$
    \item Let $H'$ be a uniformly random vertex permutation of $H.$
    \item Now, draw $\bfK\sim\ergraphhalf$ and output $\bfG = \bfK\cup H.$
\end{enumerate}
Step 4 shows that the model can be captured via the planted subgraph result of \cref{thm:countingstars}.  We obtain the following corollary, implicit in \cite{yu2024counting}. 
\begin{corollary}
[Implicit in \cite{yu2024counting}]
Let $Q$ be a non-negative matrix.
If there exists a degree-$D$ test distinguishing $\ergraphhalf$ and 
$\SBM(n;p,Q)$ with high probability for some absolute constant $D,$ then one can also distinguish the two distribution with high probability using the signed count of some star on at most $D$ edges.
\end{corollary}

We still describe the corresponding partition function maximization step as it will be useful in the proof of \cref{thm:maximizingpartition2sbm} (our result for 2-SBMs).
\begin{theorem}[Maximizing Partition Functions in Non-Negative SBMs]
\label{thm:maximizinginnonnegative}
Consider any $\SBM(p,Q)$ model in which all entries of the matrix $Q$ are non-negative. Then, for any connected graph $H$ on at most $D$ edges, 
$$
\Psi_{\SBM(p,Q)}(H)\le
\max_{1\le t \le D}
\Psi_{\SBM(p,Q)}(\Star_t).
$$
\end{theorem}
\begin{proof}
Let $H$ be any connected graph on at most $D$ edges. 
Again, we start by recalling \cref{prop:explicitFourier}:
$$
\Fourier_{\SBM(p,Q)}(H)= 
\sum_{x_1, x_2, \ldots, x_h\in [k]}\Bigg(
\prod_{i = 1}^h p_{x_i}
\times
\prod_{(i,j)\in \edges(H)}
Q_{x_i,x_j}\Bigg).
$$ 
Now, we make the following observation. Suppose that $H$ is not a tree. Then, one can  remove some edge $(i,j)\in \edges(H)$ to obtain a graph $H'$ that is still connected. Clearly, 
\begin{equation}
\label{eq:effectofedgeremoval}
\begin{split}
& \Fourier_{\SBM(p,Q)}(H')= 
\sum_{x_1, x_2, \ldots, x_h\in [k]}\Bigg(
\prod_{i = 1}^h p_{x_i}
\times
\prod_{(i,j)\in \edges(H')}
Q_{x_i,x_j}\Bigg)\\
&\ge 
\sum_{x_1, x_2, \ldots, x_h\in [k]}\Bigg(
\prod_{i = 1}^h p_{x_i}
\times
\prod_{(i,j)\in \edges(H)}
Q_{x_i,x_j}\Bigg) = 
\Fourier_{\SBM(p,Q)}(H),
\end{split}
\end{equation}
where we used the simple fact that $Q\in [0,1]$ and $\edges(H')\subseteq \edges(H).$

Repeating the same edge-removal procedure, we are left with some spanning tree $T$ of $H$ such that 
$$
\Fourier_{\SBM(p,Q)}(T)\ge 
\Fourier_{\SBM(p,Q)}(H).
$$
As $T$ is spanning, $|\vertices(T)|= |\vertices(H)|,$ so 
$$
|\Fourier_{\SBM(p,Q)}(T)|^{\frac{1}{|\vertices(T)|}}\ge 
|\Fourier_{\SBM(p,Q)}(H)|^{\frac{1}{|\vertices(H)|}}.
$$

Now, we will further remove edges from the tree. Suppose that $T$ is not a star. Then, it has a subgraph which is a path of length 3. Removing the middle edge partitions $T$ into two trees (both with at least one edge) such that $\vertices(T_1)\cup \vertices(T_2) = \vertices(T)$ and 
$\vertices(T_1)\cap \vertices(T_2) = \emptyset.$ Using the same-argument as in \eqref{eq:effectofedgeremoval}, we conclude that 
$$
\Fourier_{\SBM(p,Q)}(T_1\sqcup T_2)\ge 
\Fourier_{\SBM(p,Q)}(T).
$$
Using \cref{cor:vertexdisjointcomparison}, we conclude that 
$$
|\Fourier_{\SBM(p,Q)}(H)|^{\frac{1}{|\vertices(H)|}} \le
|\Fourier_{\SBM(p,Q)}(T)|^{\frac{1}{|\vertices(T)|}}\le 
\max(
|\Fourier_{\SBM(p,Q)}(T_1)|^{\frac{1}{|\vertices(T_1)|}}, 
|\Fourier_{\SBM(p,Q)}(T_2)|^{\frac{1}{|\vertices(T_2)|}}
).
$$
Repeating this operation while no paths of length at least 3 are left, we conclude that 
$$
|\Fourier_{\SBM(p,Q)}(H)|^{\frac{1}{|\vertices(H)|}} \le 
|\Fourier_{\SBM(p,Q)}(\Star_t)|^{\frac{1}{|\vertices(\Star_t)|}} 
$$
for some star graph on $t\le D$ edges. 
\end{proof}

\subsection{SBMs with Non-Vanishing Community Probabilities}


\label{sec:sbmnonvanishing}
\begin{theorem}[Maximizing Partition Functions with Non-Vanishing Community Probabilities]
\label{thm:maximizingpartitioninnonvanishing}
Suppose that $\SBM(p,Q)$ is a stochastic block model on $k$ communities such that $p_i \ge c\; \forall  i\in [k]$ for some constant $c>0.$ Then, for any connected graph $H$ on at most $D$ edges,  
\begin{align*}
\Psi_{\SBM(p,Q)}(H)
\lesssim_{c,D}
\max\bigg(
\Psi_{\SBM(p,Q)}(\cycle_4),
\Psi_{\SBM(p,Q)}(\Star_1),
\Psi_{\SBM(p,Q)}(\Star_2)
\bigg).
\end{align*}
\end{theorem}
\begin{proof}
Denote $h = |\vertices(H)|.$
We split the proof into two parts depending on whether $H$ is a tree.

\paragraph{1. $H$ is not a tree.} In that case, we can show that 4-cycles ``dominate'' $H.$ We need the following statement on signed 4-cycle counts.
\begin{lemma}
\label{lem:4cycleinnonvanishing}
Suppose that $\SBM(p,Q)$ is a stochastic block model on $k$ communities such that $p_i \ge c\; \forall  i\in [k]$ for some universal constant $c>0.$ Then, 
$$
\Fourier_{\SBM(p,Q)}(\cycle_4) \gtrsim_c\max_{i,j\in [k]}Q^4_{i,j}\text{ and, equivalently, }
\Psi_{\SBM(p,Q)}(\cycle_4) \gtrsim_c\max_{i,j\in [k]}|Q_{i,j}|.
$$
\end{lemma}

Before we present the proof of the lemma, we will show that it immediately implies the following statement. 

\begin{claim} Under the assumptions on $\SBM(p,Q)$ in \cref{thm:maximizingpartitioninnonvanishing}, for any connected graph $H$ which is not a tree, 
$$
\Psi_{\SBM(p,Q)}(H)\lesssim_{c,D}
\Psi_{\SBM(p,Q)}(\cycle_4). 
$$
\end{claim}
\begin{proof} Since $H$ is not a tree, $|\vertices(H)|\le |\edges(H)|.$ Recalling
\cref{prop:explicitFourier},

\begin{align*}
& \Big|\Fourier_{\SBM(p,Q)}(H)\Big|=
\Bigg|
\sum_{x_1, x_2, \ldots, x_h\in [k]}\Bigg(
\prod_{i = 1}^h p_{x_i}
\times
\prod_{(i,j)\in \edges(H)}
Q_{x_i,x_j}\Bigg)\Bigg|\\
& \le 
\sum_{x_1, x_2, \ldots, x_h\in [k]}
\prod_{i = 1}^h p_{x_i}
\times
\prod_{(i,j)\in \edges(H)}(\max_{u,v}
|Q_{u,v}|)\\
& =
\sum_{x_1, x_2, \ldots, x_h\in [k]}
\prod_{i = 1}^h p_{x_i}
(\max_{u,v}
|Q_{u,v}|)^{|\edges(H)|}\\
& \le
\sum_{x_1, x_2, \ldots, x_h\in [k]}
\prod_{i = 1}^h p_{x_i}
(\max_{u,v}
|Q_{u,v}|)^{|\vertices(H)|}\\
& = (\max_{u,v}
|Q_{u,v}|)^{|\vertices(H)|}.
\end{align*}
The claim follows immediately from \cref{lem:4cycleinnonvanishing}.
\end{proof}

\begin{proof}[Proof of \cref{lem:4cycleinnonvanishing}]
We start by analyzing \cref{prop:explicitFourier}:
\begin{align*}
& \Fourier_{\SBM(p,Q)}(\cycle_4)\\
& = 
\sum_{x_1, x_2, x_3, x_4}
p_{x_1}p_{x_2}p_{x_3}p_{x_4}
Q_{x_1, x_2}
Q_{x_2, x_3}
Q_{x_3, x_4}
Q_{x_4, x_1}\\
& = 
\sum_{x_1,x_3}p_{x_1}p_{x_3}
\Bigg(
\sum_{x}
p_x
Q_{x_1, x}
Q_{x, x_3}
\Bigg)^2\\
& \ge 
\sum_{y}p_{y}^2
\Bigg(
\sum_{x}
p_x
Q_{y, x}
Q_{y,x}
\Bigg)^2\\
& = 
\sum_{y}p_{y}^2
\Bigg(
\sum_{x}
p_x
Q_{y, x}^2
\Bigg)^2\\
& \ge
\sum_y p_y^2\sum_x p_x^2Q_{x,y}^4 = 
\sum_{x,y}
p_x^2p_y^2Q_{x,y}^4\ge \max_{x,y}c^4Q_{x,y}^4 \gtrsim_c \max_{x,y}Q_{x,y}^4,
\end{align*}
as desired.
\end{proof}

\paragraph{2. $H$ is a tree.} Suppose that $H$ is a tree. We can assume that $H$ has at least three edges as otherwise $H$ is a star and there is nothing to prove. In particular, this means that $H$ has at least two leaves. Let these be $h, h-1.$ Let their parents be $\parent(h-1)$ and $\parent(h).$ We rewrite \cref{prop:explicitFourier} in a way that allows us to compare to signed 2-stars and 4-cycles. Applying the leaf-isolation technique in \cref{intro:leafisolation}

\begin{align*}
& \Bigg|\Fourier_{\SBM(p,Q)}(H)\Bigg|\\
&\le
\sum_{x_{1},x_2, \ldots, x_{h-2}}\Bigg(
p_{x_{1}}p_{x_2}\cdots p_{x_{h-2}}
\prod_{(ij)\in \edges(H)\backslash \{(\parent(h-1),h-1), (\parent(h),h)\}}|Q_{x_i,x_j}|\times\\
& \quad\quad\quad\quad\times
\bigg(\big(
\sum_{x_{h-1}}
p_{x_{h-1}}
Q_{x_{\parent(h-1)}, x_{h-1}}
\big)^2 +
\big(
\sum_{x_{h}}
p_{x_{h}}
Q_{x_{\parent(h)}, x_{h}}
\big)^2\bigg)\Bigg)
\end{align*}
We used the simple inequality $|a|\times|b|\le a^2 + b^2.$
We now interpret the terms above.

\paragraph{Square Terms and 2-Stars.} From \cref{cor:starcounts} and the fact that $p_i\ge c>0$ for any $i,$
\begin{align*}
& \big(
\sum_{x_{h-1}}
p_{x_{h-1}}
Q_{x_{\parent(h-1)}, x_{h-1}}
\big)^2 +
\big(
\sum_{x_{h}}
p_{x_{h}}
Q_{x_{\parent(h)} x_{h}}
\big)^2\\
& \lesssim_c
p_{\parent(h-1)}
\big(
\sum_{x_{h-1}}
p_{x_{h-1}}
Q_{x_{\parent(h-1)}, x_{h-1}}
\big)^2 +
p_{\parent(h)}
\big(
\sum_{x_{h}}
p_{x_{h}}
Q_{x_{\parent(h)} x_{h}}
\big)^2\\
& \lesssim_c
\Fourier_{\SBM(p,Q)}(\Star_2).
\end{align*}

\paragraph{4-Cycles.}
Now, by \cref{lem:4cycleinnonvanishing},

\begin{align*}
& \sum_{x_{1},x_2, \ldots, x_{h-2}}
p_{x_{1}}p_{x_2}\cdots p_{x_{h-2}}
\prod_{(ij)\in \edges(H)\backslash \{(\parent(h-1),h-1), (\parent(h),h)\}}|Q_{x_i,x_j}|\\
& \le 
\sum_{x_{1},x_2, \ldots, x_{h-2}}
p_{x_{1}}p_{x_2}\ldots p_{x_{h-2}}
(\max_{u,v}|Q_{u,v}|)^{|\edges(H)|-2}\\
&  = 
(\max_{u,v}|Q_{u,v}|)^{|\vertices(H)|-3}\\
& \lesssim_{c,D}
|\Fourier_{\SBM(p,Q)}(\cycle_4)|^{\frac{|\vertices(H)|-3}{4}}.
\end{align*}
Altogether, we obtain that 
\begin{equation*}
    \begin{split}
        &\big|\Fourier_{\SBM(p,Q)}(H)\big|\lesssim_{c,D} 
        \big|\Fourier_{\SBM(p,Q)}(\Star_2)\big|\times 
        |\Fourier_{\SBM(p,Q)}(\cycle_4)|^{\frac{|\vertices(H)|-3}{4}}.
    \end{split}
\end{equation*}
We now proceed similarly to \cref{cor:vertexdisjointcomparison}. Namely, we rewrite 
\begin{equation*}
    \begin{split}
        &\big|\Fourier_{\SBM(p,Q)}(H)\big|^{\frac{1}{|\vertices(H)|}}\lesssim_{c,D} 
        \big|\Fourier_{\SBM(p,Q)}(\Star_2)\big|^{\frac{1}{|\vertices(H)|}}\times 
        |\Fourier_{\SBM(p,Q)}(\cycle_4)|^{\frac{|\vertices(H)|-3}{4}\times \frac{1}{|\vertices(H)|}} \Longleftrightarrow\\
        &
        \big|\Fourier_{\SBM(p,Q)}(H)\big|^{\frac{1}{|\vertices(H)|}}\lesssim_{c,D} 
        \Bigg(\big|\Fourier_{\SBM(p,Q)}(\Star_2)\big|^{\frac{1}{3}}\Bigg)^{\frac{3}{|\vertices(H)|}}\times 
        \Bigg(\big|\Fourier_{\SBM(p,Q)}(\cycle_4)\big|^{\frac{1}{4}}\Bigg)^{\frac{|\vertices(H)|-3}{|\vertices(H)|}}
    \end{split}
\end{equation*}
As $\frac{|\vertices(H)|-3}{|\vertices(H)|} + \frac{3}{|\vertices(H)|} = 1,$ the conclusion follows.
\end{proof}
Again, \cref{thm:beatingvarnullimpliesbeatingvarplanted} applied for any $D\ge 8$ gives the corresponding testing result.

\begin{theorem}[Testing in Non-Vanishing Community Probabilities]
Suppose that $c>0$ is an absolute constant and $\SBM(p,Q)$ is an SBM model on $k$ communities such that $p_i>c\quad \forall i \in [k].$
If there exists a constant degree test distinguishing $\ergraphhalf$ and 
$\SBM(n;p,Q)$ with high probability, one can also distinguish the two distributions with high probability using the signed count of one of the $\Star_1$ (edge), $\Star_2$ (wedge), or $\cycle_4$ (4-cycle) graphs.
\end{theorem}

\subsection{General 2-SBMs}
\label{sec:2sbms}
\begin{theorem}[Maximizing Partition Functions in 2-SBMs]
\label{thm:maximizingpartition2sbm}
Let $D\in \mathbb{N}$ be some fixed even natural number. 
Suppose that $\SBM(p,Q)$ is an arbitrary SBM on $k=2$ communities. Then, for any connected graph $H$ on at most $D$ edges,
$$
\Psi_{\SBM(p,Q)}(H)
\lesssim_D 
\max\Bigg(
\Psi_{\SBM(p,Q)}(\cycle_4),
\max_{1\le t \le D}
\Psi_{\SBM(p,Q)}(\Star_t)
\Bigg).
$$
\end{theorem}
Due to the length and complexity of the proof, we split it into several sections.
Without loss of generality suppose that $p_1\le p_2.$ As $p_1 + p_2 = 1,$ this means that $p_2\ge 1/2.$

We will also use throughout the following simple claim about signed 4-cycles in 2-SBMs.

\begin{claim}
\label{claim:4cycle2sbm}
Whenever $p_2\ge p_1$ in a 2-community stochastic block model $\SBM(p,Q),$ $$\Fourier_{\SBM(p,Q)}(\cycle_4) =\Theta(\max(p_1^4Q_{1,1}^4, p_1^2Q_{1,2}^4, Q_{2,2}^4)).$$
\end{claim}
\begin{proof}
Expanding \cref{prop:explicitFourier},
\begin{align*}
& \Fourier_{\SBM(p,Q)}(\cycle_4) = 
p_1^4Q_{1,1}^4 + 
4p_1^3p_2Q_{1,1}^2Q_{1,2}^2 + 
4p_1^2p_2^2 Q_{1,1}Q_{1,2}^2Q_{2,2}\\
& \quad\quad\quad\quad\quad\quad\quad+ 
2p_1^2p_2^2 Q_{1,2}^4 + 
4p_1p_2^3 Q_{1,2}^2Q_{2,2}^2 + 
p_2^4Q_{2,2}^4.
\end{align*}
Observe that $|4p_1^2p_2^2 Q_{1,1}Q_{1,2}^2Q_{2,2}|\le 
2p_1^3p_2Q_{1,1}^2Q_{1,2}^2 +
2p_1p_2^3Q_{1,2}^2Q_{2,2}^2
$ by AM-GM. Hence the above expression is bounded between 
$p_1^4Q_{1,1}^4 + 
2p_1^3p_2Q_{1,1}^2Q_{1,2}^2 +  
2p_1^2p_2^2 Q_{1,2}^4 + 
2p_1p_2^3 Q_{1,2}^2Q_{2,2}^2 + 
p_2^4Q_{2,2}^4$ and 
$p_1^4Q_{1,1}^4 + 
6p_1^3p_2Q_{1,1}^2Q_{1,2}^2 +  
2p_1^2p_2^2 Q_{1,2}^4 + 
6p_1p_2^3 Q_{1,2}^2Q_{2,2}^2 + 
p_2^4Q_{2,2}^4.$ Similarly, 
observe that $0\le 2p_1^3p_2Q_{1,1}^2Q_{1,2}^2\le p_1^4Q_{1,1}^4 + p_1^2p_2^2Q_{1,2}^2.$ Using also that $p_2\in [1/2, 1)$ gives the desired conclusion.
\end{proof}

\subsubsection{Step 1: Comparison with a non-negative block-model.}
The key idea in the proof is a \emph{comparison between SBM models.} Namely, let $|Q|$ be the matrix formed by taking entry-wise absolute values of $Q.$ 
Consider $\SBM(p,|Q|).$  \cref{prop:explicitFourier} combined with triangle inequality implies that 
\begin{equation}
    \begin{split}
        & |\Fourier_{\SBM(p,Q)}(H)|^{\frac{1}{|\vertices(H)|}}=\Bigg| 
\sum_{x_1, x_2, \ldots, x_h\in [k]}\Bigg(
\prod_{i = 1}^h p_{x_i}
\times
\prod_{(i,j)\in \edges(H)}
Q_{x_i,x_j}\Bigg)\Bigg|^{\frac{1}{|\vertices(H)|}}\\
& \le \sum_{x_1, x_2, \ldots, x_h\in [k]}\Bigg(
\prod_{i = 1}^h p_{x_i}
\times
\prod_{(i,j)\in \edges(H)}
|Q_{x_i,x_j}|\Bigg)^{\frac{1}{|\vertices(H)|}} = 
|\Fourier_{\SBM(p,|Q|)}(H)|^{\frac{1}{|\vertices(H)|}}.
    \end{split}
\end{equation}
On the other hand, from \cref{thm:maximizinginnonnegative}, we know that 
$$
|\Fourier_{\SBM(p,|Q|)}(H)|^{\frac{1}{|\vertices(H)|}}\lesssim_D
|\Fourier_{\SBM(p,|Q|)}(\Star_t)|^{\frac{1}{|\vertices(\Star_t)|}}
$$
for some $t \in \{1,2,\ldots, D\}.$
Altogether, this implies that 
\begin{align}
\label{eq:compareHtoabsvaluemodel}
|\Fourier_{\SBM(p,Q)}(H)|^{\frac{1}{|\vertices(H)|}}\lesssim_D
|\Fourier_{\SBM(p,|Q|)}(\Star_t)|^{\frac{1}{|\vertices(\Star_t)|}}.
\end{align}
\eqref{eq:compareHtoabsvaluemodel} would imply the result if it were the case that taking absolute values of $Q$ does not significantly increase star counts, i.e.
\begin{align}
\label{eq:signedvsunsignedsbmforstars}
|\Fourier_{\SBM(p,|Q|)}(\Star_t)|^{\frac{1}{|\vertices(\Star_t)|}}\lesssim_t
|\Fourier_{\SBM(p,Q)}(\Star_t)|^{\frac{1}{|\vertices(\Star_t)|}}
\end{align}
Unfortunately, \eqref{eq:signedvsunsignedsbmforstars} is certainly incorrect as there are matrices $Q$ for which $\Fourier_{\SBM(p,Q)}(\Star_t)|^{\frac{1}{|\vertices(\Star_t)|}} = 0$ for any $t\ge 1,$ but $\Fourier_{\SBM(p,|Q|)}(\Star_t)|^{\frac{1}{|\vertices(\Star_t)|}}>0$
(see \cref{examplethm:notodddegree}). 

Yet, it turns out that whenever \eqref{eq:signedvsunsignedsbmforstars} does not hold, $\SBM(p,Q)$ needs to have a \emph{very specific structure which leads to cancellations in the Fourier coefficients of stars}. The rest of the proof is devoted to first describing such a structure, and then exploiting it to compare the Fourier coefficient of $H$ with that of stars and 4-cycles. For the rest of the proof, let $\tspecial$ be one value of $t\in \{1,2,\ldots, D\}$ such that \eqref{eq:compareHtoabsvaluemodel} is satisfied.

\subsubsection{Step 2: Identifying structure which leads to cancellations in the Fourier coefficients of stars.} By \cref{cor:starcounts},
$$
\Fourier_{\SBM(p,Q)}(\Star_\tspecial) = 
p_1(p_1 Q_{11} + p_2Q_{1,2})^\tspecial + 
p_2(p_1 Q_{12} + p_2Q_{2,2})^\tspecial.
$$
Now, suppose that \eqref{eq:signedvsunsignedsbmforstars} does not hold. Then, for some large absolute constant $C,$\footnote{For concreteness, $C = 128$ works.} it must be the case that 
\begin{equation}
\label{eq:signedstrtseqfail}
\begin{split}
    & (4C^2)^{\tspecial}\Big|
    p_1(p_1 Q_{11} + p_2Q_{1,2})^\tspecial + 
p_2(p_1 Q_{12} + p_2Q_{2,2})^\tspecial
    \Big|\\
    &\le 
    p_1(p_1 |Q_{11}| + p_2|Q_{1,2}|)^\tspecial + 
p_2(p_1 |Q_{12}| + p_2|Q_{2,2}|)^\tspecial.
\end{split}
\end{equation}
There might be two reasons for this inequality:
\begin{enumerate}
    \item \emph{Within-community cancellation:} 
    \begin{align}
            & p_1\Big(p_1|Q_{1,1}| + p_2|Q_{1,2}|\Big)^\tspecial \ge C^\tspecial \times
            p_1\Big|p_1Q_{1,1} + p_2Q_{1,2}\Big|^\tspecial \tag{31.a}
            \label{eq:inbracketcancelationa}
            \\
            & \quad\quad\quad\quad\quad\quad\quad\quad\quad\quad\quad\quad\quad \text{or}\label{eq:inbracketcancelation}\\
            & p_2\Big(p_1|Q_{1,2}| + p_2|Q_{2,2}|\Big)^\tspecial \ge C^\tspecial \times
            p_2\Big|p_1Q_{1,2} + p_2Q_{2,2}\Big|^\tspecial.
            \tag{31.b}
            \label{eq:inbracketcancelationb}
    \end{align}
    \item \emph{Between-community cancellation:} 
    
    \begin{equation}
    \label{eq:outofbracketcancelation}
        \begin{split}
            & p_1\Big(p_1|Q_{1,1}| + p_2|Q_{1,2}|\Big)^\tspecial \le C^\tspecial\times
            p_1\Big|p_1Q_{1,1} + p_2Q_{1,2}\Big|^\tspecial \\
            & \quad\quad\quad\quad\quad\quad\quad\quad\quad\quad\quad\quad\quad \text{and}\\
            & p_2\Big(p_1|Q_{1,2}| + p_2|Q_{2,2}|\Big)^\tspecial \le C^\tspecial\times 
            p_2\Big|p_1Q_{1,2} + p_2Q_{2,2}\Big|^\tspecial,\\
            & \quad\quad\quad\quad\quad\quad\quad\quad\quad\quad\quad\quad\quad \text{but}\\
            & 
            (4C^2)^\tspecial\times\Big|p_1\Big(p_1Q_{1,1} + p_2Q_{1,2}\Big)^\tspecial + 
            p_2\Big(p_1Q_{1,2} + p_2Q_{2,2}\Big)^\tspecial\Big|
            \\
            & \quad\quad\quad\quad\quad\le 
            p_1\Big|p_1Q_{1,1} + p_2Q_{1,2}\Big|^\tspecial + 
            p_2\Big|pQ_{1,2} + p_2Q_{2,2}\Big|^\tspecial
        \end{split}
    \end{equation}
    
\end{enumerate}
\subsubsection{Step 3: Within-community Cancellations.}
\label{sec:incommunitycancelation}
We begin by analyzing the case of \emph{within-community cancellations} which turns out to be the simpler case.

\begin{lemma}
\label{lem:11dominates12}
Suppose that \eqref{eq:signedstrtseqfail} and \eqref{eq:inbracketcancelation} hold for some $C\ge 4.$ Then, $|Q_{1,2}|\le Cp_1|Q_{1,1}|.$
\end{lemma}
\begin{proof}
\textbf{Case 1)} First, suppose that \eqref{eq:inbracketcancelationa} holds. Then, 
$$
(p_1|Q_{1,1}| + p_2|Q_{1,2}|)\ge C
\Big|
p_1Q_{1,1} + p_2Q_{1,2}\Big|\ge 
C\Big|
p_1|Q_{1,1}| - p_2|Q_{1,2}|
\Big|.
$$ 
When $C\ge 4,$ this immediately implies that $p_2|Q_{1,2}|\le 2p_1|Q_{1,1}|$ which is enough as $p_2\ge 1/2.$

\medskip

\noindent
\textbf{Case 2)} Now, suppose that that \eqref{eq:inbracketcancelationb} holds.
Then, 
$p_1|Q_{1,2}| + p_2|Q_{2,2}|\ge C \Big|p_1Q_{1,2} + p_2Q_{2,2}
\Big|\ge C\Big|
p_1|Q_{1,2}| - p_2|Q_{2,2}|
\Big|.$ As in Case 1),  $p_2|Q_{2,2}|\le 2p_1|Q_{1,2}| .$  Hence, $(p_1|Q_{1,2}| + p_2|Q_{2,2}|)\le 3p_1|Q_{1,2}|,$ so 
\begin{align}
\label{eq:upperboundsecondcommunity}
\Big|p_1Q_{1,2} + p_2Q_{2,2}
\Big|\le \frac{1}{C}(p_1|Q_{1,2}| + p_2|Q_{2,2}|)\le \frac{3}{C}p_1|Q_{1,2}|
\end{align}
There are two cases. Either $|p_1Q_{1,1} + p_2Q_{1,2}|\le  \frac{1}{\sqrt{C}}(p_1|Q_{1,1}| + p_2|Q_{1,2}|),$ which immediately implies  $p_2|Q_{1,2}|\le 2p_1|Q_{1,1}|$ as in {Case 1)} provided $C\ge 4^2.$ Or, 
\begin{equation}
\label{eq:lowerboundfirstcommunity}
\begin{split}
&|p_1Q_{1,1} + p_2Q_{1,2}|>  \frac{1}{\sqrt{C}}(p_1|Q_{1,1}| + p_2|Q_{1,2}|)\ge\frac{1}{\sqrt{C}}p_2|Q_{1,2}|\ge \frac{1}{2\sqrt{C}}|Q_{1,2}|.
\end{split}
\end{equation}
Combining \eqref{eq:upperboundsecondcommunity} and 
\eqref{eq:lowerboundfirstcommunity},
\begin{align*}
    & \Bigg|
    p_1(p_1Q_{1,1} + p_2Q_{1,2})^\tspecial + 
p_2(pQ_{1,2} + p_2Q_{2,2})^\tspecial
    \Bigg|\\
    & \ge
    \Bigg|
    p_1(p_1Q_{1,1} + p_2Q_{1,2})^\tspecial\Bigg| - \Bigg| 
p_2(pQ_{1,2} + p_2Q_{2,2})^\tspecial
    \Bigg|\\
    & \ge p_1 \times \Big(\frac{1}{2\sqrt{C}}|Q_{1,2}|\Big)^\tspecial - 
    p_2\Big(\frac{3}{C}p_1|Q_{1,2}|\Big)^\tspecial\\
    & \ge 
     p_1 \times \Big(\frac{1}{2\sqrt{C}}|Q_{1,2}|\Big)^\tspecial - 
    p_1\Big(\frac{3}{C}|Q_{1,2}|\Big)^\tspecial\\
    & \ge
    \frac{p_1|Q_{1,2}|^\tspecial}{(2C)^\tspecial}
\end{align*}
whenever $C\ge 128.$ 

Now, recall \eqref{eq:signedstrtseqfail}. Then, as $(p_1|Q_{1,2}| + p_2|Q_{2,2}|)\le 3p_1|Q_{1,2}|,$ it must be the case that
\begin{align*}
    & p_1(p_1|Q_{1,1}| + p_2|Q_{1,2}|)^\tspecial + p_2(3p_1|Q_{1,2}|)^\tspecial\\
    & \ge 
    p_1(p_1|Q_{1,1}| + p_2|Q_{1,2}|)^\tspecial +p_2
    (p_1|Q_{1,2}| + p_2|Q_{2,2}|)^\tspecial\\
    & \ge 
    (4C^2)^{\tspecial}\Bigg|
    p_1(p_1Q_{1,1} + p_2Q_{1,2})^\tspecial + 
p_2(p_1Q_{1,2} + p_2Q_{2,2})^\tspecial
    \Bigg|\\
    &\ge 
    (4C^2)^{\tspecial}
    \frac{p_1|Q_{1,2}|^\tspecial}{(2C)^\tspecial}.
\end{align*}
For large enough $C\ge 128,$ this immediately implies that
$$
p_1|Q_{1,1}| + p_2|Q_{1,2}| \ge 
\frac{3C}{2}\times |Q_{1,2}|.
$$
Again, this implies $|Q_{1,2}|\le Cp_1|Q_{1,1}|.$ 
\end{proof}

It turns out that \cref{lem:11dominates12} is enough to imply \cref{thm:maximizingpartition2sbm}. 

\begin{lemma}
\label{lem:when11dominates12}
Suppose that $|Q_{1,2}|\le C p_1|Q_{1,1}|$  for some absolute constant $C.$ Then, 
for any connected graph $H$ on at most $D$ edges, 
\begin{equation*}
    \begin{split}
        & 
\Psi_{\SBM(p,Q)}(H)
\lesssim_D 
\max\Bigg(
\Psi_{\SBM(p,Q)}(\cycle_4),
\max_{1\le t \le D}
\Psi_{\SBM(p,Q)}(\Star_t)
\Bigg).
        \end{split}
\end{equation*}
\end{lemma}
In the rest of the section, we prove \cref{lem:when11dominates12}.
Again, the proof depends on whether $H$ is a tree.

\begin{proof}[Proof of \cref{lem:when11dominates12} when $H$ is not a tree]
Suppose that $H$ is connected and not a tree. Then, $H$ has at least $|\vertices(H)|$ edges. Now, we rewrite \cref{prop:explicitFourier} as follows:
\begin{equation}
\label{eq:maximalbound}
    \begin{split}
        & \Fourier_{\SBM(p,Q)}(H)\\
        & = 
\sum_{x_1, x_2, \ldots, x_h\in \{1,2\}}\Bigg(
\prod_{i = 1}^h p_{x_i}
\times
\prod_{(i,j)\in \edges(H)}
Q_{x_i,x_j}\Bigg)\\
& = \sum_{K\subseteq \vertices(H)}
p_1^{|K|}\times 
p_2^{|\vertices(H)| - |K|}\times 
Q_{1,1}^{|\edges_H(K,K)|}
Q_{1,2}^{|\edges_H(K, \vertices(H)\backslash K)|}
Q_{2,2}^{|\edges_H(\vertices(H)\backslash K, \vertices(H)\backslash K)|},
    \end{split}
\end{equation}
where the equivalence between the second and the third line follows simply by choosing $K$ to be the subset of vertices labeled by $1.$ In particular, 
\begin{equation}
    \begin{split}
        & |\Fourier_{\SBM(p,Q)}(H)|\\
        & \le \sum_{K\subseteq \vertices(H)}\Bigg|
p_1^{|K|}\times 
p_2^{|\vertices(H)| - |K|}\times 
Q_{1,1}^{|\edges_H(K,K)|}
Q_{1,2}^{|\edges_H(K, \vertices(H)\backslash K)|}
Q_{2,2}^{|\edges_H(\vertices(H)\backslash K, \vertices(H)\backslash K)|}\Bigg|\\
& \lesssim_D
\max_{K\subseteq \vertices(H)}
\Bigg|
p_1^{|K|}\times 
Q_{1,1}^{|\edges_H(K,K)|}
Q_{1,2}^{|\edges_H(K, \vertices(H)\backslash K)|}
Q_{2,2}^{|\edges_H(\vertices(H)\backslash K, \vertices(H)\backslash K)|}\Bigg|\\
& = 
\max_{K\subseteq \vertices(H)}
p_1^{|K|}\times 
|Q_{1,1}|^{|\edges_H(K,K)|}
|Q_{1,2}|^{|\edges_H(K, \vertices(H)\backslash K)|}
|Q_{2,2}|^{|\edges_H(\vertices(H)\backslash K, \vertices(H)\backslash K)|}
    \end{split}
\end{equation}

From now on, we will focus on a specific choice of $K.$ Suppose that $K$ is such that $H|_{K}$ has $a$ connected components and $|K| = a+t.$ Hence, $|\edges_H(K,K)|\ge  t.$ Let $H|_{\vertices(H)\backslash K}$ have $b$ connected components and $b+s$ vertices. Hence, 
$|\edges_H(\vertices(H)\backslash K, \vertices(H)\backslash K)|\ge s.$ Finally, observe that $|\edges_H(K,\vertices(H)\backslash K)|\ge a+b-1$ since $H$ is connected.

\begin{figure}

\begin{center}
\begin{tikzpicture}
\draw[thick] (0,0) ellipse (2cm and 1cm);
\node at (0,0.7) {\text{$K$}}; 

\draw[thick, blue] (-1.3,0) circle (0.3cm); 
\node at (-1.3,0) {\text{$K_1$}}; 

\draw[thick, blue] (0,0) circle (0.3cm); 
\node at (0,0) {\text{$K_2$}}; 

\draw[thick, blue] (1.3,0) circle (0.3cm); 
\node at (1.3,0) {\text{$K_a$}}; 

\node at (0.65,0) {\text{$\dots$}}; 

\draw[thick] (-1.1,1) -- (-1.1,1.5); 
\draw[thick] (1.1,1) -- (1.1,1.5); 
\draw[thick] (-1.1,1.5) -- (1.1,1.5); 

\node at (0,1.8) {\text{$a+t$ vertices,{} \textcolor{red}{ at least $t$ edges}}}; 

\draw[thick] (0,-3) ellipse (2cm and 1cm);
\node at (0,-3.7) {\text{$\vertices(H) \setminus K$}}; 

\draw[thick, blue] (-1.3,-3) circle (0.3cm); 
\node at (-1.3,-3) {\text{$T_1$}}; 

\draw[thick, blue] (0,-3) circle (0.3cm); 
\node at (0,-3) {\text{$T_2$}}; 

\draw[thick, blue] (1.3,-3) circle (0.3cm); 
\node at (1.3,-3) {\text{$T_b$}}; 

\node at (0.65,-3) {\text{$\dots$}}; 

\draw[thick] (-1.1,-4) -- (-1.1,-4.5); 
\draw[thick] (1.1,-4) -- (1.1,-4.5); 
\draw[thick] (-1.1,-4.5) -- (1.1,-4.5); 

\node at (0,-5) {\text{$b+s$ vertices, \textcolor{red}{ at least $s$ edges}}}; ; 


\draw[thick, red] (-1.3,-0.3) -- (-1.3,-2.7); 

\draw[thick, red] (-1.3,-0.3) -- (1.3,-2.7); 

\draw[thick, red] (0,-0.3) -- (-1.3,-2.7); 

\draw[thick, red] (0,-0.3) -- (1.3,-2.7); 

\draw[thick, red] (0.65,-0.3) -- (0.65,-2.7); 

\draw[thick, red] (0,-2.7) -- (1.3,-0.3); 

\draw[thick, red, rotate around={180:(2.2,-1.5)}] (2.2,0) -- (2.2,-3); 
\draw[thick, red, rotate around={180:(2.2,-1.5)}] (2.2,-3) -- (2.5,-3); 
\draw[thick, red, rotate around={180:(2.2,-1.5)}] (2.2,0) -- (2.5,0); 

\node[align=center] at (3,-1.5) {\text{at least} \\ \text{$a+b-1$}\\\text{edges}}; 

\end{tikzpicture}
\caption{Decomposition of $H$ into vertices and edges in the case of in-community cancellations when $H$ is not a tree. $|\vertices(H)| = a+b+t+s,$ $|\edges_H(K,K)|\ge t, |\edges_H(K, \vertices(H)\backslash K)|\ge a+b -1,$ $ |\edges_H(\vertices(H)\backslash K, \vertices(H)\backslash K)|\ge s$ and $|\edges(H)|\ge |\vertices(H)| = a+ b + t + s.$ The blue circles represent the connected components in $H|_{K}$ and $H|_{\vertices(H)\backslash K},$ respectively.}
\end{center}
\end{figure}
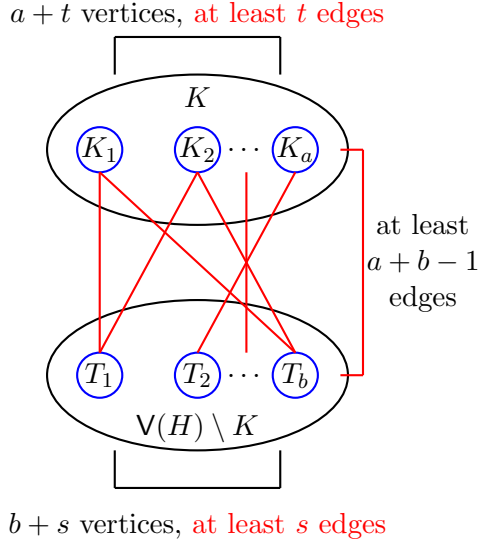

Since $H$ is not a tree, $|\edges(H)|\ge |\vertices(H)| = a+b+t+s.$ Thus, there is at least one more edge besides the $a+b+t+s-1$ identified so far. Suppose that it is of the type $|Q_{i,j}|.$ As in the layout from the end of \cref{sec:introcmparewithnonnegative}, we aim to replace some of the $|Q_{1,2}|$ instances by $p_1|Q_{1,1}|$ so that we can compare to a 4-cycle. Namely:
\begin{equation*}
    \begin{split}
        & p_1^{|K|}
        |Q_{1,1}|^{|\edges_H(K,K)|}
        |Q_{1,2}|^{|\edges_H(K,\vertices(H)\backslash K)|}
        |Q_{2,2}|^{|\edges_H(\vertices(H)\backslash K, \vertices(H)\backslash K)|}\\ 
        & \le p_1^{a+t}\times |Q_{1,1}|^t\times |Q_{1,2}|^{a+b-1}\times 
        |Q_{2,2}|^s\times |Q_{i,j}|\\
        & \lesssim p_1^{a+ t + a + b - 1} \times |Q_{1,1}|^{a + b + t -1}\times |Q_{2,2}|^s \times |Q_{i,j}|,
    \end{split}
\end{equation*}
where we used \cref{lem:11dominates12} in the second inequality.

If $(i,j)\in \{(1,1), (1,2)\},$ then $a\ge 1$ (as label 1 exists) and $|Q_{i,j}| = O(|Q_{1,1}|).$ Thus, the above is expression is bounded by 
\begin{align*}
&p_1^{a + b + t } \times |Q_{1,1}|^{a + b + t}\times |Q_{2,2}|^s\\
&\le (p_1|Q_{1,1}|)^{a + b + t}\times |Q_{2,2}|^s\\
& \lesssim_D |\Fourier_{\SBM(p,Q)}(\cycle_4)|^{\frac{a+ b + t}{4}}\times 
|\Fourier_{\SBM(p,Q)}(\cycle_4)|^{\frac{s}{4}}\\
&= |\Fourier_{\SBM(p,Q)}(\cycle_4)|^{\frac{|\vertices(H)|}{4}},
\end{align*} which completes the proof.

If $(i,j) = (2,2),$ the above expression is bounded by
$$
p_1^{t + a + b - 1} \times |Q_{1,1}|^{a + b + t -1}\times |Q_{2,2}|^{s + 1}.
$$
We argue similarly that this is at most $\Fourier_{\SBM(p,Q)}(\cycle_4)|^{\frac{|\vertices(H)|}{4}}$ up to multiplicative constants depending on $D.$
\end{proof}

If we were to apply the same argument when $H$ were a tree, we would prove that
$$
|\Fourier_{\SBM(p,Q)}(H)| 
\lesssim_D
|\Fourier_{\SBM(p,Q)}(\cycle_4)|^{\frac{|\vertices(H)|-1}{4}},
$$ 
which is not strong enough. Thus, we need to use the technique of isolating leaves as in \cref{thm:maximizingpartitioninnonvanishing}

\begin{proof}[Proof of \cref{lem:when11dominates12} when $H$ is a tree]
Now, suppose that $H$ is a tree. Again, there is nothing to prove if it has less than 3 vertices. If it has at least three vertices, then it has at least two leaves. Let these be $h-1$ and $h.$

We now bound $|\Fourier_{\SBM(p,Q)}(H)|$ by using the same technique as in \cref{thm:maximizingpartitioninnonvanishing}. Namely, 
\begin{equation}
\label{eq:boundwiyhoutleavesinnerbrackets}
    \begin{split}
        & \Fourier_{\SBM(p,Q)}(H)\\
        &=  
        \sum_{K\subseteq \vertices(H)\backslash\{h-1, h\}}
        p^{|K|}        
        Q_{1,1}^{|\edges_H(K,K)|}
        Q_{1,2}^{|\edges_H(K,\vertices(H)\backslash K)|}
        Q_{2,2}^{|E(\vertices(H)\backslash K, \vertices(H)\backslash K)|}\\
        & \quad\quad\quad\quad\quad\quad\times (p_1Q_{1,1} + p_2Q_{1,2})^{P_1(K)}
        (p_1Q_{1,2} + p_2Q_{2,2})^{P_2(K)},
    \end{split}
\end{equation}
where $P_1(K)$ is the number of parents of $h, h-1$ with label $1$ according to $K$ and $P_2(K)$ with label 2 (in case of a common parent, we count it twice). Note that $P_1(K) + P_2(K) = 2.$ By AM-GM,
\begin{align*}
&|(p_1Q_{1,1} + p_2Q_{2,2})^{P_1(K)}
        (p_1Q_{1,2} + p_2Q_{2,2})^{P_2(K)}|\\
&\lesssim 
p_1^{-\indicator[|K|\ge 1]}\Bigg(
p_1
(p_1Q_{1,1} + p_2Q_{1,2})^2 + 
p_2
(p_1Q_{1,2} + p_2Q_{2,2})^2\Bigg)\\
&=
p_1^{-\indicator[|K|\ge 1]}\Fourier_{\SBM(p,Q)}(\Star_2).
\end{align*}
The $\indicator[|K|\ge 1]$ factor comes from the fact that if $|K| = 0,$ then $P_1(K) = 0.$   
Combining with \eqref{eq:boundwiyhoutleavesinnerbrackets},
\begin{equation}
\label{eq:treecaseincommunitystarcomp}
\begin{split}
    &\Big| \Fourier_{\SBM(p,Q)}(H)\Big|\\
    &\lesssim_D \max_{K\subseteq \vertices(H)\backslash\{h,h-1\}}
        p^{|K|}        
        |Q_{1,1}|^{|\edges_H(K,K)|}
        |Q_{1,2}|^{|\edges_H(K,H\backslash K)|}
        |Q_{2,2}|^{|\edges_H(\vertices(H)\backslash K, \vertices(H)\backslash K)|}\\
        & \quad\quad\quad\quad\quad\quad\quad\times p^{-\indicator[|K|\ge 1]}|
        \Fourier_{\SBM(p,Q)}(\Star_2)|
\end{split}
\end{equation}
Again, we denote by $a$ the connected components of $K,$ by $a + t$ its vertices, by $b$ the connected components of $\vertices(H)\backslash K$ and by $b + s$ the number of vertices. 

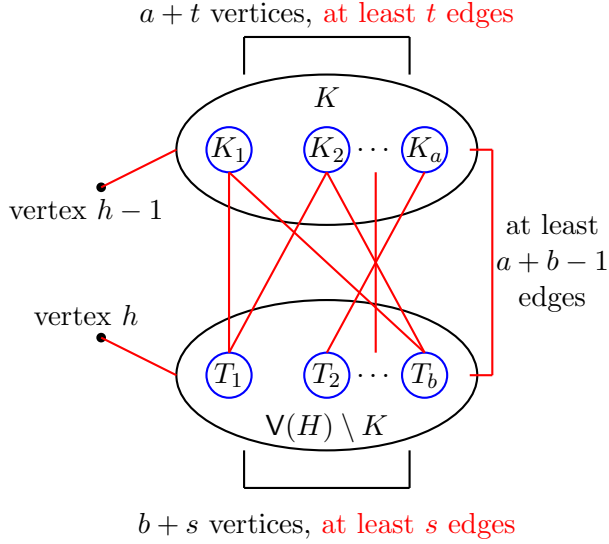
\begin{figure}
\begin{center}
\begin{tikzpicture}

\draw[thick] (0,0) ellipse (2cm and 1cm);
\node at (0,0.7) {\text{$K$}}; 

\filldraw[thick, black] (-3,-.5) circle (.05cm);
\node at (-3.2,-.8) {\text{vertex $h-1$}};
\draw[thick,red] (-3,-.5) -- (-2,0);

\draw[thick, blue] (-1.3,0) circle (0.3cm); 
\node at (-1.3,0) {\text{$K_1$}}; 

\draw[thick, blue] (0,0) circle (0.3cm); 
\node at (0,0) {\text{$K_2$}}; 

\draw[thick, blue] (1.3,0) circle (0.3cm); 
\node at (1.3,0) {\text{$K_a$}}; 

\node at (0.65,0) {\text{$\dots$}}; 

\draw[thick] (-1.1,1) -- (-1.1,1.5); 
\draw[thick] (1.1,1) -- (1.1,1.5); 
\draw[thick] (-1.1,1.5) -- (1.1,1.5); 

\node at (0,1.8) {\text{$a+t$ vertices,{} \textcolor{red}{ at least $t$ edges}}}; 

\draw[thick] (0,-3) ellipse (2cm and 1cm);
\node at (0,-3.7) {\text{$\vertices(H) \setminus K$}}; 

\filldraw[thick, black] (-3,-2.5) circle (.05cm);
\node at (-3.2,-2.2) {\text{vertex $h$}};
\draw[thick,red] (-3,-2.5) -- (-2,-3);

\draw[thick, blue] (-1.3,-3) circle (0.3cm); 
\node at (-1.3,-3) {\text{$T_1$}}; 

\draw[thick, blue] (0,-3) circle (0.3cm); 
\node at (0,-3) {\text{$T_2$}}; 

\draw[thick, blue] (1.3,-3) circle (0.3cm); 
\node at (1.3,-3) {\text{$T_b$}}; 

\node at (0.65,-3) {\text{$\dots$}}; 

\draw[thick] (-1.1,-4) -- (-1.1,-4.5); 
\draw[thick] (1.1,-4) -- (1.1,-4.5); 
\draw[thick] (-1.1,-4.5) -- (1.1,-4.5); 

\node at (0,-5) {\text{$b+s$ vertices, \textcolor{red}{ at least $s$ edges}}}; ; 


\draw[thick, red] (-1.3,-0.3) -- (-1.3,-2.7); 

\draw[thick, red] (-1.3,-0.3) -- (1.3,-2.7); 

\draw[thick, red] (0,-0.3) -- (-1.3,-2.7); 

\draw[thick, red] (0,-0.3) -- (1.3,-2.7); 

\draw[thick, red] (0.65,-0.3) -- (0.65,-2.7); 

\draw[thick, red] (0,-2.7) -- (1.3,-0.3); 

\draw[thick, red, rotate around={180:(2.2,-1.5)}] (2.2,0) -- (2.2,-3); 
\draw[thick, red, rotate around={180:(2.2,-1.5)}] (2.2,-3) -- (2.5,-3); 
\draw[thick, red, rotate around={180:(2.2,-1.5)}] (2.2,0) -- (2.5,0); 

\node[align=center] at (3,-1.5) {\text{at least} \\ \text{$a+b-1$}\\\text{edges}}; 

\end{tikzpicture}
\end{center}
\caption{Decomposition of $H$ into vertices and edges in the case of in-community cancellations when $H$ is a tree. $|\vertices(H)| = a+b+t+s+2,$ $|\edges_H(K,K)|\ge t, |\edges_H(K, \vertices(H)\backslash K)|= a+b -1,$ $ |\edges_H(\vertices(H)\backslash K, \vertices(H)\backslash K)|= s$ and
 $|\edges(H)|= a+ b + t + s+1.$ The blue circles represent the connected components in $H|_{K}$ and $H|_{\vertices(H)\backslash K},$ respectively.
 We have drawn the two leaves to have parents with different labels for the purposes of illustration, but this might not be the case.
 }

\end{figure}

We rewrite the right hand-side of \eqref{eq:treecaseincommunitystarcomp} as
\begin{align*}
& p_1^{a + t}
|Q_{1,1}|^t 
|Q_{1,2}|^{a + b - 1}
|Q_{2,2}|^s
p_1^{-\indicator[a\ge 1]}|\Fourier_{\SBM(p,Q)}(\Star_2)|\\
& \le 
p_1^{a + t -1 + a + b - \indicator[a\ge 1]}
|Q_{1,1}|^{a+ b + t -1}
|Q_{2,2}|^s
|\Fourier_{\SBM(p,Q)}(\Star_2)|,
\end{align*}
where again the inequality follows from \cref{lem:11dominates12}.
Using that $a-\indicator[a\ge 1]\ge 0,$ we continue as follows 
\begin{align*}
&(p_1|Q_{1,1}|)^{a+ b + t -1}
   |Q_{2,2}|^s
|\Fourier_{\SBM(p,Q)}(\Star_2)|\\
& \lesssim_D
|\Fourier_{\SBM(p,Q)}(\cycle_4)|^{\frac{a +b + t -1}{4}}
|\Fourier_{\SBM(p,Q)}(\cycle_4)|^\frac{s}{4}
|\Fourier_{\SBM(p,Q)}(\Star_2)|\\
& = 
|\Fourier_{\SBM(p,Q)}(\cycle_4)|^\frac{|V(H)| - 3}{4}
|\Fourier_{\SBM(p,Q)}(\Star_2)|.
\end{align*}
Exactly as in the proof of \cref{thm:maximizingpartitioninnonvanishing}, this implies that
$$
\Psi_{\SBM(p,Q)}(H)\lesssim_D
\max\big(
\Psi_{\SBM(p,Q)}(\cycle_4),
\Psi_{\SBM(p,Q)}(\Star_2)
\big).\qedhere
$$
\end{proof}
\subsubsection{Step 4: Between-Community Cancellations.}
\label{sec:outofbracketcancelation}
Now, suppose that \eqref{eq:signedstrtseqfail} holds in the case of between-community cancellations, \eqref{eq:outofbracketcancelation}. Additionally, we assume that $|Q_{1,2}|\ge p_1|Q_{1,1}|$ as otherwise we can invoke \cref{lem:when11dominates12} and complete the proof. In particular, 
$p_1|Q_{1,1}| + p_2|Q_{1,2}| \le 2|Q_{1,2}|.$ By \eqref{eq:outofbracketcancelation}, also $|p_1Q_{1,1} + p_2Q_{1,2}|\ge \frac{p_2}{C}|Q_{1,2}|\ge \frac{1}{2C}|Q_{1,2}|.$
We record:
\begin{align}
\label{eq:boundonpbeta11plusbeta22}
\frac{1}{2C}|Q_{1,2}|\le |p_1Q_{1,1} + p_2Q_{1,2}|\le 2|Q_{1,2}|\end{align}
Furthermore, \eqref{eq:outofbracketcancelation} implies that $p_1(p_1Q_{1,1} + p_2Q_{1,2})^\tspecial, 
            p_2(p_1Q_{1,2} + p_2Q_{2,2})^\tspecial$ have different signs. In particular, $\tspecial$ is odd. Hence, $\tspecial\le D-1$ as $D$ is even. 
            
Next, we will show that unless $p_1\ge |Q_{1,2}|,$ $\Star_{\tspecial+1}$ dominates $H.$

\begin{lemma}
\label{lem:whenbeta12gep}
Suppose that \eqref{eq:signedstrtseqfail} holds and $p_1 \le |Q_{1,2}|.$ Then, 
$$
\Psi_{\SBM(p,Q)}(H)\lesssim_D
\Psi_{\SBM(p,Q)}(\Star_{\tspecial+1})
$$    
\end{lemma}
\begin{proof}
    For brevity, denote $\mu = p_1Q_{1,1} + p_2Q_{1,2},\lambda = 
p_1Q_{1,2} + p_2Q_{2,2}.$ When \eqref{eq:outofbracketcancelation} holds, $C|\mu|\ge p_1|Q_{1,1}| + p_2|Q_{1,2}| $ and 
$C|\lambda|\ge p_1|Q_{1,2}| + p_2|Q_{2,2}|.$ However, 
$(4C^2)^\tspecial\times |p_1\mu^\tspecial + p_2\lambda^\tspecial|\le \Big(p_1|\mu|^\tspecial + p_2|\lambda|^\tspecial\Big).$ In particular, this means
$(4C^2)^\tspecial\times \Big|p_1|\mu|^\tspecial - p_2|\lambda|^\tspecial\Big|\le  \Big(p_1|\mu|^\tspecial + p_2|\lambda|^\tspecial\Big),$ so 
$p_1|\mu|^\tspecial\in (1\pm 1/2)p_2|\lambda|^\tspecial$ when $C\ge 4.$  By \eqref{eq:boundonpbeta11plusbeta22}, also 
$|\mu|\ge \frac{1}{2C}|Q_{1,2}|\ge \frac{p_1}{2C}.$

Now, since $\tspecial+1$ is even and $\mu \ge \frac{p_1}{2C},$ $(p_1\mu^{\tspecial+1})^{1/(\tspecial+2)}\gtrsim_D (p_1|\mu|^\tspecial)^{1/(\tspecial+1)}.$ Altogether, 
\begin{align*}
        & |\Fourier_{\SBM(p,Q)}(\Star_{\tspecial+1})|^{\frac{1}{|\vertices(\Star_{\tspecial+1})|}} = |p_1\mu^{\tspecial+1} + p_2\lambda^{\tspecial+1}|^{1/(\tspecial+2)}\ge |p_1\mu^{\tspecial+1}|^{1/(\tspecial+2)}
        \gtrsim_D |p_1\mu^\tspecial|^{1/(\tspecial+1)} \\&\gtrsim_D \big( p_1|\mu|^\tspecial + p_2|\lambda|^\tspecial\big)^{1/(\tspecial+1)}  =
        |\Fourier_{\SBM(p,|Q|)}(\Star_{\tspecial})|^{1/(\tspecial+1)} \gtrsim_D
        |\Fourier_{\SBM(p,Q)}(H)|^{\frac{1}{|\vertices(H)|}}.\qedhere
\end{align*}
\end{proof}
Thus, from now on, we 
\begin{align}
\label{eq:beta12issmall}
\text{ Assume that }|Q_{1,2}|\le p_1.
\end{align}

\begin{lemma}
\label{lem:whenbeta12lep}
Suppose that \eqref{eq:signedstrtseqfail} holds, $|Q_{1,2}|\le p_1,$ and $\tspecial>1.$ Then, 
$$
\Psi_{\SBM(p,Q)}(H)\lesssim_D
\Psi_{\SBM(p,Q)}(\Star_{\tspecial-1})
$$    
\end{lemma}
\begin{proof}[Proof sketch]
The proof is exactly the same as that of \cref{lem:whenbeta12lep}. The only difference is that this time $|\mu|\lesssim_D p_1,$ which implies 
$(p_1\mu^{\tspecial-1})^{1/\tspecial}\gtrsim_D (p_1|\mu|^\tspecial)^{1/(\tspecial+1)}.$ 
\end{proof}

Hence, what is left is the case $\tspecial = 1.$ For convenience, we restate and recall the conditions in the remaining case:
\begin{enumerate}
    \item $\tspecial = 1.$
    \item \eqref{eq:outofbracketcancelation}: $p_1|Q_{1,1}| + p_2|Q_{1,2}|\le C\times|pQ_{1,1} + p_2Q_{1,2}|,$ and 
    $p_1|Q_{1,2}| + p_2|Q_{2,2}|\le C \times|p_1Q_{1,2} + p_2Q_{2,2}|,$ and 
    \begin{align*}
    & (4C^2)\times |p_1(p_1Q_{1,1} + p_2Q_{1,2}) + p_2(p_1Q_{1,2} + p_2Q_{2,2})| \le
    p_1|p_1Q_{1,1} + p_2Q_{1,2}| + p_2|pQ_{1,2} + p_2Q_{2,2}|.
    \end{align*}
    \item $|Q_{1,2}|\ge p_1|Q_{1,1}|$ as otherwise \cref{lem:when11dominates12} gives \cref{thm:maximizingpartition2sbm}.
    \item $|Q_{1,2}|\le p_1$ as otherwise \cref{lem:whenbeta12gep} implies the result.
\end{enumerate}
Observe that under the above assumptions, $|p_1(p_1Q_{1,1} + p_2Q_{1,2})|\le 
p_1 |p_1Q_{1,1}| + p_1|p_2Q_{1,2})|\le 
2p_1|Q_{1,2}|.$ Furthermore,
\begin{align*}
    & 
    p_1|p_1Q_{1,1} + p_2Q_{1,2}| + p_2|pQ_{1,2} + p_2Q_{2,2}|\\
    & \ge C\times |p_1(p_1Q_{1,1} + p_2Q_{1,2}) + p_2(pQ_{1,2} + p_2Q_{2,2})|\\
    & \ge C\times\Big|
    p_1|p_1Q_{1,1} + p_2Q_{1,2}| - 
    p_2|p_1Q_{1,2} + p_2Q_{2,2}|
    \Big|.
\end{align*}
For large enough $C,$ this implies
\begin{align}
\label{eq:insecondcommunity12dominates}
p_2|p_1Q_{1,2} + p_2Q_{2,2}|\le 2p_1|p_1Q_{1,1} + p_2Q_{1,2}| \le 4p_1|Q_{1,2}|.
\end{align}
As $p_2|p_1Q_{1,2} + p_2Q_{2,2}|\ge \frac{1}{C}p_1|Q_{1,2}| + \frac{1}{C}p_2|Q_{2,2}|,$ this further implies
\begin{enumerate}
    \item [4.] $|Q_{2,2}|\le 8cp_1|Q_{1,2}|$ for large enough $C.$
\end{enumerate}
Finally, this immediately implies that $|p_1Q_{1,2} + p_2Q_{2,2}| = O(p_1|Q_{1,2}|).$ As we are in the case of between-community cancellations, the reverse inequality is also true and $|p_1Q_{1,2} + p_2Q_{2,2}| \asymp p_1|Q_{1,2}|.$
\begin{enumerate}
    \item [5.] For any even $t\in [2,D],$
    \begin{equation}
    \label{eq:starexpectationsint1outbrackets}
    \begin{split}
     &\Fourier_{\SBM(p,Q)}(\Star_t)\\
     & = 
	p_1(p_1Q_{1,1} + p_2Q_{1,2})^t  + 
	p_2(p_1Q_{1,2} + p_2Q_{2,2})^t\\
	& \asymp_D 
	p_1(p_1|Q_{1,1}| + p_2|Q_{1,2}|)^t + 
	p_2(p_1|Q_{1,2}| + p_2|Q_{2,2}|)^t\\
	& \asymp_D 
	p_1 |Q_{1,2}|^t + p_1^{t}|Q_{1,2}|^t\asymp_D 
	p_1 |Q_{1,2}|^t.
    \end{split}
    \end{equation}
\end{enumerate}
Now, we will carry out a similar analysis as in \cref{sec:incommunitycancelation} but we will need an even more careful analysis of the leaves. Namely, suppose that $\leaves(H)\subseteq \vertices(H)$ is the set of leaves and for $K\subseteq V(H)\backslash \leaves(H),$ the
set of leaves with parents of label $1$ is $\leaves_1(K)$ and with parents of label $2$ is $\leaves_2(K).$ Hence, 
\begin{equation}
\label{eq:trianglineqcrosscommunity}
\begin{split}
    &|\Fourier_{\SBM(p,Q)}(H)|\\
    & = \Bigg|\sum_{K\subseteq \vertices(H)\backslash \leaves(H)}
    p_1^{|K|}p_2^{|\vertices(H)| - |\leaves(H)| - |K|}\times\\
    &\quad\quad\quad\quad\quad\times
    Q_{1,1}^{|\edges_H(K,K)|}
    Q_{1,2}^{|\edges_H(K,\vertices(H)\backslash (K\cup \leaves(H)))|}
    Q_{2,2}^{|\edges_H(\vertices(H)\backslash (K\cup \leaves(H)),  \vertices(H)\backslash (K\cup \leaves(H)))|}\\
    &\quad\quad\quad\quad\quad\times
    (p_1Q_{1,1} + p_2Q_{1,2})^{|\leaves_1(K)|}
    (p_1Q_{1,2} + p_2Q_{2,2})^{|\leaves_2(K)|}
    \Bigg|\\
    & \lesssim_D \max_{K\subseteq \vertices(H)\backslash \leaves(H)} p_1^{|K|}\\
     & \quad\quad\quad\quad\quad \times
    |Q_{1,1}|^{|\edges_H(K,K)|}
    |Q_{1,2}|^{|\edges_H(K,\vertices(H)\backslash (K\cup \leaves(H)))|}|Q_{2,2}|^{|\edges_H(\vertices(H)\backslash (K\cup \leaves(H)),  \vertices(H)\backslash (K\cup \leaves(H)))|}\\
    & \quad\quad\quad\quad\quad \times
    |p_1Q_{1,1} + p_2Q_{1,2}|^{|\leaves_1(K)|}
    |p_1Q_{1,2} + p_2Q_{2,2}|^{|\leaves_2(K)|}.
    \end{split}
\end{equation}
Now on, let $K$ be the respective maximizer in \eqref{eq:trianglineqcrosscommunity}.
Again, let $K$ have $a$ connected components and $a+t$ vertices and $H|_{\vertices(H)\backslash (K\cup \leaves(H))}$ have $b$ connected components and $b+s$ vertices. In particular, 
$$
|\vertices(H)| = a + b + t + s + |\leaves_1(K)| + |\leaves_2(K)|.
$$

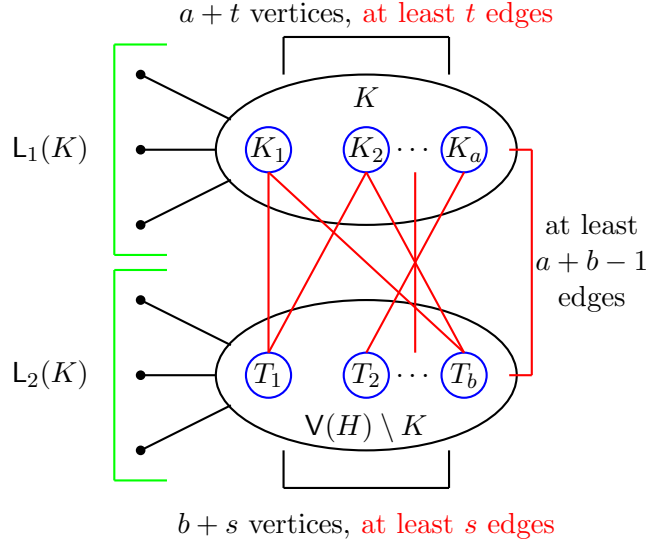
\begin{figure}
\begin{center}
\begin{tikzpicture}

\draw[thick] (0,0) ellipse (2cm and 1cm);
\node at (0,0.7) {\text{$K$}}; 

\filldraw[thick, black] (-3,-1) circle (.05cm);
\draw[thick] (-3,-1) -- (-1.8,-.4);

\filldraw[thick, black] (-3,0) circle (.05cm);
\draw[thick] (-3,0) -- (-2,0);

\filldraw[thick, black] (-3,1) circle (.05cm);
\draw[thick] (-3,1) -- (-1.8,.4);

\draw[thick, green] (-3.35,1.4) -- (-2.65,1.4); 

\draw[thick, green] (-3.35,1.4) -- (-3.35,-1.4); 

\draw[thick, green] (-3.35,-1.4) -- (-2.65,-1.4);

\node at (-4.2,0) {\text{$\leaves_1(K)$}};

\filldraw[thick, black] (-3,-4) circle (.05cm);
\draw[thick] (-3,-4) -- (-1.8,-3.4);

\filldraw[thick, black] (-3,-3) circle (.05cm);
\draw[thick] (-3,-3) -- (-2,-3);

\filldraw[thick, black] (-3,-2) circle (.05cm);
\draw[thick] (-3,-2) -- (-1.8,-2.6);

\draw[thick, green] (-3.35,-1.6) -- (-2.65,-1.6); 

\draw[thick, green] (-3.35,-1.6) -- (-3.35,-4.4); 

\draw[thick, green] (-3.35,-4.4) -- (-2.65,-4.4);

\node at (-4.2,-3) {\text{$\leaves_2(K)$}};

\draw[thick, blue] (-1.3,0) circle (0.3cm); 
\node at (-1.3,0) {\text{$K_1$}}; 

\draw[thick, blue] (0,0) circle (0.3cm); 
\node at (0,0) {\text{$K_2$}}; 

\draw[thick, blue] (1.3,0) circle (0.3cm); 
\node at (1.3,0) {\text{$K_a$}}; 

\node at (0.65,0) {\text{$\dots$}}; 

\draw[thick] (-1.1,1) -- (-1.1,1.5); 
\draw[thick] (1.1,1) -- (1.1,1.5); 
\draw[thick] (-1.1,1.5) -- (1.1,1.5); 

\node at (0,1.8) {\text{$a+t$ vertices,{} \textcolor{red}{ at least $t$ edges}}}; 

\draw[thick] (0,-3) ellipse (2cm and 1cm);
\node at (0,-3.7) {\text{$\vertices(H) \setminus K$}}; 

\draw[thick, blue] (-1.3,-3) circle (0.3cm); 
\node at (-1.3,-3) {\text{$T_1$}}; 

\draw[thick, blue] (0,-3) circle (0.3cm); 
\node at (0,-3) {\text{$T_2$}}; 

\draw[thick, blue] (1.3,-3) circle (0.3cm); 
\node at (1.3,-3) {\text{$T_b$}}; 

\node at (0.65,-3) {\text{$\dots$}}; 

\draw[thick] (-1.1,-4) -- (-1.1,-4.5); 
\draw[thick] (1.1,-4) -- (1.1,-4.5); 
\draw[thick] (-1.1,-4.5) -- (1.1,-4.5); 

\node at (0,-5) {\text{$b+s$ vertices, \textcolor{red}{ at least $s$ edges}}}; ; 


\draw[thick, red] (-1.3,-0.3) -- (-1.3,-2.7); 

\draw[thick, red] (-1.3,-0.3) -- (1.3,-2.7); 

\draw[thick, red] (0,-0.3) -- (-1.3,-2.7); 

\draw[thick, red] (0,-0.3) -- (1.3,-2.7); 

\draw[thick, red] (0.65,-0.3) -- (0.65,-2.7); 

\draw[thick, red] (0,-2.7) -- (1.3,-0.3); 

\draw[thick, red, rotate around={180:(2.2,-1.5)}] (2.2,0) -- (2.2,-3); 
\draw[thick, red, rotate around={180:(2.2,-1.5)}] (2.2,-3) -- (2.5,-3); 
\draw[thick, red, rotate around={180:(2.2,-1.5)}] (2.2,0) -- (2.5,0); 

\node[align=center] at (3,-1.5) {\text{at least} \\ \text{$a+b-1$}\\\text{edges}}; 

\end{tikzpicture}
\caption{Decomposition of $H$ into vertices and edges in the case of between-community cancellations. $|\vertices(H)| = a+b+t+s + |\leaves_1(K)| + |\leaves_2(K)|,$ $|\edges_H(K,K)|\ge t, |\edges_H(K, \vertices(H)\backslash K)|\ge a+b -1,$ 
$ |\edges_H(\vertices(H)\backslash K,\vertices(H)\backslash K)|\ge s.$ The blue circles represent the connected components in $H|_{K}$ and $H|_{\vertices(H)\backslash K}.$}
\end{center}
\end{figure}
For a fixed $K,$ the last expression in 
\eqref{eq:trianglineqcrosscommunity}
is at most 
$$
p_1^{a+t}|Q_{1,1}|^t |Q_{1,2}|^{a + b - 1}
|Q_{2,2}|^s 
|p_1Q_{1,1} + Q_{1,2}|^{|\leaves_1(K)|}
    |p_1Q_{1,2} + p_2Q_{2,2}|^{|\leaves_2(K)|}.
$$
Using that $|Q_{2,2}|\le p_1|Q_{1,2}|,$ the above expression is bounded by 
\begin{align}
\label{eq:comparisoninoutofbracketscase}
    p_1^{a+t +s + |\leaves_2(K)|}|Q_{1,1}|^t |Q_{1,2}|^{a + b - 1 + s + |\leaves_2(K)|}
|p_1Q_{1,1} + Q_{1,2}|^{|\leaves_1(K)|}.
\end{align}
We now analyze two cases based on which of the two quantities $|Q_{1,1}|\sqrt{p_1}$ and  $|Q_{1,2}|$ is larger. Again, the intuition is as in \cref{sec:introcmparewithnonnegative} -- this will allow us to replace instances of $|Q_{1,2}|$ with $|Q_{1,1}|\sqrt{p_1}$ and compare witha four-cycle.
\paragraph{ Case 1) $|Q_{1,1}|\sqrt{p_1}\le |Q_{1,2}|.$}
Using $|Q_{1,1}|\sqrt{p_1}\le |Q_{1,2}|$, we bound \eqref{eq:comparisoninoutofbracketscase}
by
\begin{align}
\label{eq:whensqrtismsaller}
p_1^{a+t/2 +s + |\leaves_2(K)|} |Q_{1,2}|^{a + b - 1 + s + |\leaves_2(K)| + t}
|p_1Q_{1,1} + Q_{1,2}|^{|\leaves_1(K)|}.
\end{align}
\paragraph{Case 1.1) $|\leaves_1(K)|\le 2.$}
Using that $p_1|Q_{1,1}| \le|Q_{1,2}|,$ the expression in \eqref{eq:whensqrtismsaller} is bounded by 
\begin{align*}
 & p_1^{a+t/2 +s + |\leaves_2(K)|}|\times  |Q_{1,2}|^{a + b - 1 + s + |\leaves_2(K)| +t + |\leaves_1(K)|}\\
 & \lesssim 
  p_1^{a+t/2 +s + |\leaves_2(K)|-1}|\times  |Q_{1,2}|^{a + b - 3 + s + |\leaves_2(K)| +t+ |\leaves_1(K)|}\times 
  \Fourier_{\SBM(p,Q)}(\Star_2)
\end{align*}
where we used \eqref{eq:starexpectationsint1outbrackets}.
\paragraph{Case 1.1.1)}
 Now, suppose that 
 $$
 a+t/2 +s + |\leaves_2(K)|-1\ge  \frac{1}{2}(a + b - 3 + s +t+ |\leaves_2(K)| + |\leaves_1(K)|) = \frac{1}{2}(|\vertices(H)|-3).
 $$
 Then, using \cref{claim:4cycle2sbm}, the above expression is bounded by 
\begin{align*}
 & (\sqrt{p_1}|Q_{1,2}|)^{a + b - 3 + s + |\leaves_2(K)| +t+ |\leaves_1(K)|}\times 
  \Fourier_{\SBM(p,Q)}(\Star_2)\\
  & \lesssim 
  |\Fourier_{\SBM(p,Q)}(\cycle_4)|^{\frac{|\vertices(H)|-3}{4}}\times 
  |\Fourier_{\SBM(p,Q)}(\Star_2)|.
\end{align*}
As in \cref{thm:maximizinginnonnegative}, this expression is at most $$\max(|\Fourier_{\SBM(p,Q)}(\cycle_4)|^{1/4}, 
  |\Fourier_{\SBM(p,Q)}(\Star_2)|^{1/3})^{|V(H)|}.$$

\paragraph{Case 1.1.2)}
The remaining case is 
\begin{align*}
&a+t/2 +s + |\leaves_2(K)|-1< \frac{1}{2}(a + b - 3 + s + |\leaves_2(K)| + t+|\leaves_1(K)|)\quad\quad\quad\Longleftrightarrow\\
&
a + s + |\leaves_2(K)| + 1< b + |\leaves_1(K)|\quad\quad\quad\Longleftrightarrow\\
& a + s + |\leaves_2(K)| + 2\le b + |\leaves_1(K)|.
\end{align*}
Since $|\leaves_1(K)|\le 2$ in the current case, it follows that $b = a + s + |\leaves_2(K)| + r$ for some $r\ge 0.$  Under these assumptions, we prove the following fact.
\begin{claim}
 $H|_{\vertices(H)\backslash (K\cup \leaves_1(K))}$ has at least $R \ge b - s-|\leaves_2(K)| = a +r$ isolated vertices. 
 \end{claim}
 \begin{proof}
 Indeed, note that $H|_{\vertices(H)\backslash (K\cup \leaves(H))}$ has $b$ connected components and $b+ s$ vertices. Hence, at least $b-s$ of the vertices in $H|_{\vertices(H)\backslash (K\cup \leaves(H))}$ are isolated. These same vertices are not isolated in $H|_{\vertices(H)\backslash (K\cup \leaves_1(K))}$ if and only if they have a neighbor in $\leaves_2(K).$ As each vertex in $\leaves_2(K)$ is a leaf, there are at most $|\leaves_2(K)|$ such of them. Hence, at least $b - s - |\leaves_2(K)| = a  + r$ of the vertices in $\vertices(H)\backslash (K\cup \leaves(H))$
are isolated in 
$H|_{\vertices(H)\backslash (K\cup |\leaves_1(K)|)}.$ \end{proof}

Denote these isolated vertices by $\mathcal{S},$ so $|\mathcal{S}|\ge a +r.$
Now consider 
$
H|_{V(H)\backslash \leaves(H)}.
$ This is a connected graph since $H$ is connected and $\leaves(H)$ is the set of leaves. Hence, each connected component in $H|_{V(H)\backslash (K\cup \leaves(H))}$ has at least one edge to a vertex in $K.$ If, furthermore, this connected component is one of the vertices in $\mathcal{S},$ it must have at least $2$ edges to $K.$ At least one edge due to connectivity and at least one more as the vertices of $\mathcal{S}$ are not leaves. Indeed, note that $\mathcal{S}\subseteq \vertices(H)\backslash (K\cup \leaves_1(K)$) and they don't have any edges to other vertices in $\vertices(H)\backslash (K\cup \leaves_1(K))$ due to being isolated, and in $\leaves_1(K)$ due to the fact that the parents of leaves in $\leaves_1(K)$ are all in $K.$ Altogether, this means that the number of edges between $K$ and $\vertices(H)\backslash (\leaves(H)\cup K)$ is at least 
$$
2|\mathcal{S}| + (b - |\mathcal{S}|)\ge 
b + |\mathcal{S}|\ge 
b + a + r .
$$ Note that all of these edges are of type $(1,2).$
Thus, the bound in \eqref{eq:trianglineqcrosscommunity} becomes
\begin{equation*}
\begin{split}
    &|\Fourier_{\SBM(p,Q)}(H)|\\
    & \lesssim_D 
    p_1^{a + t}|Q_{1,1}|^t\times |Q_{1,2}|^{b + a + r }\times 
    |Q_{2,2}|^{s}\times 
    |p_1Q_{1,1} + p_2Q_{1,2}|^{|\leaves_1(K)|}
    |p_1Q_{1,2} + p_2Q_{2,2}|^{|\leaves_2(K)|}\\
    & \quad\quad\quad\quad\text{(Using that $\sqrt{p_1}|Q_{1,1}|\le |Q_{1,2}|, p_1|Q_{1,2}|\gtrsim |Q_{2,2}|$ and \cref{eq:boundonpbeta11plusbeta22,eq:insecondcommunity12dominates})}\\
    & \lesssim_D 
    p_1^{a + t}|Q_{1,2}|^tp^{-t/2}
    |Q_{1,2}|^{b + a + r }\times 
    |Q_{1,2}|^{s}p^s
    |Q_{1,2}|^{|\leaves_1(K)|}
    |Q_{1,2}|^{|\leaves_2(K)|}
    p_1^{|\leaves_2(K)|}\\
    & \lesssim_D 
    p_1^{a + t/2 + s + |\leaves_2(K)|}
    |Q_{1,2}|^{t + b + a + r +s+|\leaves_1(K)| + |\leaves_2(K)|}\\
    & \lesssim 
    p_1^{a + t/2 + s + |\leaves_2(K)| -1}\times 
    |Q_{1,2}|^{t + b + a + s+|\leaves_1(K)| + |\leaves_2(K)| - 3}\times p_1|Q_{1,2}|^{2}\times 
    |Q_{1,2}|^{1 + r}\\
    & \quad\quad\quad\quad\text{(Using that $|Q_{1,2}|\le p_1.$)}\\
    & \lesssim 
    p_1^{a + t/2 + s + |\leaves_2(K)| + r}
    \times 
    |Q_{1,2}|^{t + b + a + s+|\leaves_1(K)| + |\leaves_2(K)| - 3}\times p_1|Q_{1,2}|^{2},
\end{split}
\end{equation*}
 Now, observe that 
$$
a + t/2 + s + |\leaves_2(K)| + r\ge \frac{1}{2}(t + b + a + s+|\leaves_1(K)| + |\leaves_2(K)| - 3)
$$
since this is equivalent to 
$$
a + s + |\leaves_2(K)| + 2r + 3\ge b + |\leaves_1(K)|,
$$
but $a + s + |\leaves_2(K)| + r = b,r + 3\ge 3>|\leaves_1(K)|.$ Altogether, 
the above expression is bounded by the familiar
$$
(\sqrt{p_1}|Q_{1,2}|)^{t + b + a + s+|\leaves_1(K)| + |\leaves_2(K)| - 3}\times p_1|Q_{1,2}|^{2}\lesssim 
|\Fourier_{\SBM(p,Q)}(\cycle_4)|^{\frac{|\vertices(H)|-3}{4}}
\times 
|\Fourier_{\SBM(p,Q)}(\Star_2)|,
$$
where the inequality follows from \cref{claim:4cycle2sbm,eq:starexpectationsint1outbrackets}.

\paragraph{Case 1.2) $|\leaves_1(K)|\ge 3.$} Let $\xi\in \{0,1\}$ be such that $|\leaves_1(K)| -\xi$ is even. We bound
\eqref{eq:trianglineqcrosscommunity} by 
\begin{align*}
& p_1^{a+t +s + |\leaves_2(K)|}|Q_{1,1}|^t |Q_{1,2}|^{a + b - 1 + s}
|p_1Q_{1,1} + p_2Q_{1,2}|^{|\leaves_1(K)|}\times
|p_1Q_{1,2} + p_2Q_{2,2}|^{|\leaves_2(K)|}
\\
& \quad\quad\quad\quad\text{(Using that $\sqrt{p_1}|Q_{1,1}|\le |Q_{1,2}|$ and \eqref{eq:boundonpbeta11plusbeta22} and \eqref{eq:insecondcommunity12dominates})}\\
& \lesssim_D 
p_1^{a+t +s + |\leaves_2(K)|}|Q_{1,2}|^tp^{-t/2} |Q_{1,2}|^{a + b - 1 + s + |\leaves_2(K)|}
|p_1Q_{1,1} + p_2Q_{1,2}|^{|\leaves_1(K)|-\xi}|Q_{1,2}|^\xi\\
& \quad\quad\quad\quad\text{(Using \eqref{eq:starexpectationsint1outbrackets})}\\
& \lesssim_D 
p_1^{a+t/2 +s + |\leaves_2(K)|}
|Q_{1,2}|^{a + b - 1 + s + |\leaves_2(K)| + t + \xi}
|\Fourier_{\SBM(p,Q)}(\Star_{|\leaves_1(K)|-\xi})|.
\end{align*}
Again, we consider two cases. 

\paragraph{Case 1.2.1)} If 
\begin{align*}
& a+t/2 +s + |\leaves_2(K)|\ge 
\frac{1}{2}(a + b - 1 + s + |\leaves_2(K)| + t + \xi)\quad\quad\quad\Longleftrightarrow\\
& a + s + |\leaves_2(K)|\ge b + \xi - 1,
\end{align*}
the last expression is bounded by 
\begin{align*}
&(\sqrt{p_1}|Q_{1,2}|)^{a + b - 1 + s + |\leaves_2(K)| + t + \xi}
|\Fourier_{\SBM(p,Q)}(\Star_{|\leaves_1(K)|-\xi})|\\
& \quad\quad\quad\quad\text{(Using \cref{claim:4cycle2sbm})}\\
&\lesssim |\Fourier_{\SBM(p,Q)}(\cycle_4)|^{\frac{|V(H)| - |\leaves_1(K)| -1 + \xi}{4}}
|\Fourier_{\SBM(p,Q)}(\Star_{|\leaves_1(K)|-\xi})| \\
&\le  \max(|\Fourier_{\SBM(p,Q)}(\cycle_4)|^{1/4}, 
  |\Fourier_{\SBM(p,Q)}(\Star_{|\leaves_1(K)|-\xi})|^{1/(|\leaves_1(K)|-\xi + 1)})^{|\vertices(H)|}.
\end{align*}
The last inequality follows in the same way as for 2-stars and 4-cycles as in \cref{thm:maximizinginnonnegative}.

\paragraph{Case 1.2.2)} Otherwise, 
\begin{align*}
& a+t/2 +s + |\leaves_2(K)|< 
\frac{1}{2}(a + b - 1 + s + |\leaves_2(K)| + t + \xi)\quad\quad\quad\Longleftrightarrow\\
& a + s + |\leaves_2(K)|< b + \xi - 1,
\end{align*}
so $b + \xi \ge a + s + |\leaves_2(K)| + 2.$ Again, let $b+ \xi = a  + s+ |\leaves_2(K)| + r$ for some $r\ge 2.$ We use the same argument as in Case  1) to argue that there are at least 
$b + a + r - \xi$ edges between $K$ and $\vertices(H)\backslash (K\cup \leaves(H)).$ We bound \eqref{eq:trianglineqcrosscommunity} by
\begin{equation*}
\begin{split}
    &|\Fourier_{\SBM(p,Q)}(H)|\\
    & \lesssim_D 
    p_1^{a + t}|Q_{1,1}|^t\times |Q_{1,2}|^{b + a + r - \xi }\times 
    |Q_{2,2}|^{s}\times 
    |p_1Q_{1,1} + p_2Q_{1,2}|^{|\leaves_1(K)|}
    |p_1Q_{1,2} + p_2Q_{2,2}|^{|\leaves_2(K)|}\\
    & \quad\quad\quad\quad\text{(Using $|Q_{1,1}|\sqrt{p_1}\le |Q_{1,2}|, |Q_{2,2}|\le p_1|Q_{1,2}|$ and \eqref{eq:starexpectationsint1outbrackets})}\\
    & \lesssim 
    p_1^{a + t}|Q_{1,2}|^tp_1^{-t/2}
    |Q_{1,2}|^{b + a + r -\xi}\times 
    |Q_{1,2}|^{s}p^s\\
    & \quad\quad\quad\quad\quad\quad\times
    p^{-1}|\Fourier_{\SBM(p,Q)}(\Star_{|\leaves_1(K)|-\xi})|\times |Q_{1,2}|^\xi
    |Q_{1,2}|^{|\leaves_2(K)|}
    p_1^{|\leaves_2(K)|}\\
    & \lesssim 
    p_1^{a + t/2 + s + |\leaves_2(K)|-1}
    |Q_{1,2}|^{t + b + a + r +s+|\leaves_2(K)|}\times 
    |\Fourier_{\SBM(p,Q)}(\Star_{|\leaves_1(K)|-\xi})|
    \\
     & \quad\quad\quad\quad\text{(Using $|Q_{1,2}|\le p_1$)}\\
    & \lesssim 
    p_1^{a + t/2 + s + |\leaves_2(K)| + r - \xi}
    |Q_{1,2}|^{t + b + a  +s+|\leaves_2(K)| + \xi - 1}\times 
    |\Fourier_{\SBM(p,Q)}(\Star_{|\leaves_1(K)|-\xi})|.
\end{split}
\end{equation*}
Now, observe that 
\begin{align*}
    & a + t/2 + s + |\leaves_2(K)| + r - \xi\ge 
    \frac{1}{2}(t + b + a  +s+|\leaves_2(K)| + \xi - 1)\quad\quad\quad\Longleftrightarrow\\
    & a + s + |\leaves_2(K)| + 2r +1\ge b + 3\xi,
\end{align*}
but $b+ \xi = a + s + |\leaves_2(K)| + r$ and $r + 1\ge 3\ge 2\xi.$ Altogether, this means that the above expression is at most 
\begin{align*}
& (\sqrt{p_1}|Q_{1,2}|)^{t + b + a  +s+|\leaves_2(K)| + \xi - 1}\times 
    |\Fourier_{\SBM(p,Q)}(\Star_{|\leaves_1(K)|-\xi})|\\
&\le
|\Fourier_{\SBM(p,Q)}(\cycle_4)|^{\frac{|V(H)| - |\leaves_1(K)| -1 + \xi}{4}}\times
|\Fourier_{\SBM(p,Q)}(\Star_{|\leaves_1(K)|-\xi})|.
\end{align*}
Again, this is enough.

\paragraph{Case 2) $|Q_{1,1}|\sqrt{p_1}> |Q_{1,2}|$}
We consider the same two cases based on $|\leaves_1(K)|.$

\paragraph{Case 2.1) $|\leaves_1(K)|\le 2.$}
We bound the expression in  \eqref{eq:trianglineqcrosscommunity} by 
\begin{align*}
    & p_1^{a+t +s + |\leaves_2(K)|}|Q_{1,1}|^t |Q_{1,2}|^{a + b - 1 + s + |\leaves_2(K)|}
|p_1Q_{1,1} + p_2Q_{1,2}|^{|\leaves_1(K)|}\\
& \quad\quad\quad\quad\text{(Using \eqref{eq:boundonpbeta11plusbeta22} and \eqref{eq:insecondcommunity12dominates})}\\
& \lesssim_D p_1^{a+t +s + |\leaves_2(K)|}|Q_{1,1}|^t \times 
|Q_{1,2}|^{a + b - 1 + s + |\leaves_2(K)|}\times 
|Q_{1,2}|^{|\leaves_1(K)|}\\
& \quad\quad\quad\quad\text{(Using that ${p_1}\ge|Q_{1,2}|$)}\\
& \lesssim_D 
p_1^{a+t +s + |\leaves_2(K)| -1}\times 
|Q_{1,1}|^t \times 
|Q_{1,2}|^{a + b + s + |\leaves_2(K)| + |\leaves_1(K)| - 3}\times 
p_1|Q_{1,2}|^2\\
& \quad\quad\quad\quad\text{(Using that $\sqrt{p_1}|Q_{1,1}|> |Q_{1,2}|$)}\\
& \lesssim_D
p_1^{a+t +s + |\leaves_2(K)| -1}\times 
|Q_{1,1}|^t \times 
|Q_{1,1}|^{a + b + s + |\leaves_2(K)|+ |\leaves_1(K)| - 3}p^{(a + b + s + |\leaves_2(K)| - 3)/2}\times 
p_1|Q_{1,2}|^2\\
& \quad\quad\quad\quad\text{(Using \eqref{eq:starexpectationsint1outbrackets})}\\
& \lesssim_D
p_1^{a+t +s + |\leaves_2(K)| -1 + \frac{1}{2}(a + b + s + |\leaves_2(K)|+ |\leaves_1(K)| - 3)}|Q_{1,1}|^{t + a + b + s + |\leaves_2(K)|+ |\leaves_1(K)| - 3}\times \Fourier_{\SBM(p,Q)}(\Star_{2}).
\end{align*}
\paragraph{Case 2.1.1)}
If 
\begin{align*}
    & a+t +s + |\leaves_2(K)| -1 + \frac{1}{2}(a + b + s + |\leaves_2(K)|+ |\leaves_1(K)| - 3)\ge\\
    & \quad\quad\quad\quad \ge
    t + a + b + s + |\leaves_2(K)|+ |\leaves_1(K)| - 3\Longleftrightarrow\\
    & a + s + \leaves_2 + 1\ge b + \leaves_1,
\end{align*}
then going back we have 
\begin{align*}
    & p^{a+t +s + |\leaves_2(K)| -1 + \frac{1}{2}(a + b + s + |\leaves_2(K)|+ |\leaves_1(K)| - 3)}|Q_{1,1}|^{t + a + b + s + |\leaves_2(K)|+ |\leaves_1(K)| - 3}\times \Fourier_{\SBM(p,Q)}(\Star_{2})\\
    & \le 
    (p_1|Q_{1,1}|)^{|V(H)|-3}
    \Fourier_{\SBM(p,Q)}(\Star_{2})\\
    & \le 
    \Fourier_{\SBM(p,Q)}(\cycle_4)^{\frac{|V(H)|-3}{4}}\times
    \Fourier_{\SBM(p,Q)}(\Star_{2})\\
    & \le 
    \max(|\Fourier_{\SBM(p,Q)}(\cycle_4)|^{1/4}, 
  |\Fourier_{\SBM(p,Q)}(\Star_{2})|^{1/3})^{|V(H)|}.
\end{align*}
\paragraph{Case 2.1.2)}
Otherwise, 
$a + s + |\leaves_2(K)| + 1< b + |\leaves_1(K)|\le b + 2,$ so 
$b  = a + s + \leaves_2(K) + r,$ where $r\ge 0.$ Again, using the argument of counting isolated vertices, we lower-bound the number of edges between $K$ and $\vertices(H)\backslash(K\cup \leaves(H))$ by $a+b+r,$ so

\begin{equation}
\label{eq:crosscase212}
\begin{split}
    &|\Fourier_{\SBM(p,Q)}(H)|\\
    & \quad\quad\quad\quad\text{(By \eqref{eq:trianglineqcrosscommunity})}\\
    & \lesssim_D
    p_1^{a + t}|Q_{1,1}|^t\times |Q_{1,2}|^{b + a + r}\times 
    |Q_{2,2}|^{s}\times 
    |p_1Q_{1,1} + p_2Q_{1,2}|^{|\leaves_1(K)|}
    |p_1Q_{1,2} + p_2Q_{2,2}|^{|\leaves_2(K)|}\\
    & \quad\quad\quad\quad\text{(Using that $\sqrt{p_1}|Q_{1,1}|> |Q_{1,2}|$)}\\
    & \lesssim 
    p_1^{a + t}|Q_{1,1}|^t\times |Q_{1,2}|^{b + a + r}\times 
    |Q_{1,2}|^{s}p^s\times 
    |Q_{1,2}|^{|\leaves_1(K)|}
    |Q_{1,2}|^{|\leaves_2(K)|}p_1^{|\leaves_2(K)|}\\
    & \lesssim 
    p_1^{a + t + s  + |\leaves_2(K)|}\times |Q_{1,1}|^t\times 
    |Q_{1,2}|^{b +a + r + s + |\leaves_1(K)| + |\leaves_2(K)|}\\
    & \quad\quad\quad\quad\text{(Using that $p_1>|Q_{1,2}|$)}\\
    & \lesssim p_1^{a + t + s  + |\leaves_2(K)| + r}
    |Q_{1,1}|^t\times 
    |Q_{1,2}|^{b +a + s + |\leaves_1(K)| + |\leaves_2(K)|}\\
    & \quad\quad\quad\quad\text{(Using that $p_1>|Q_{1,2}|$)}\\
    & \lesssim p_1^{a + t + s  + |\leaves_2(K)| + r}
    |Q_{1,1}|^t\times 
    |Q_{1,2}|^{b +a + s + |\leaves_1(K)| + |\leaves_2(K)| - 3}p_1|Q_{1,2}|^2\\
    & \quad\quad\quad\quad\text{(Using \eqref{eq:starexpectationsint1outbrackets})}\\
    & \lesssim p_1^{a + t + s  + |\leaves_2(K)| + r + \frac{1}{2}(b +a + s + |\leaves_1(K)| + |\leaves_2(K)|-3)}
    |Q_{1,1}|^{t + b +a + s + |\leaves_1(K)| + |\leaves_2(K)|-3}p_1|Q_{1,2}|^2.\\
\end{split}
\end{equation}
Now, observe that 
\begin{align*}
    & a + t + s  + |\leaves_2(K)| + r + \frac{1}{2}(b +a + s + |\leaves_1(K)| + |\leaves_2(K)|-3)\ge\\
    &\quad\quad\quad\quad\ge t + b +a + s + |\leaves_1(K)| + |\leaves_2(K)|-3\quad\quad\quad\Longleftrightarrow\\
    & 3 + 2 r + a + s + |\leaves_2(K)|\ge b + \leaves_1.
\end{align*}
This inequality holds as $r + a + s + |\leaves_2(K)|\ge b, 3 + r \ge 3\ge |\leaves_2(K)|.$ Hence, the last expression in \eqref{eq:crosscase212} is bounded by 
$$
(p_1|Q_{1,1}|)^{t + b +a + s + |\leaves_1(K)| + |\leaves_2(K)|-3}p_1|Q_{1,2}|^2\lesssim 
|\Fourier_{\SBM(p,Q)}(\cycle_4)|^{\frac{|V(H)|-3}{4}}\times 
|\Fourier_{\SBM(p,Q)}(\Star_2)|.
$$

\paragraph{Case 2.2) $|\leaves_1(K)|> 2.$}
Let $\xi \in \{0,1\}$ be such that $|\leaves_1(K)|-\xi$ is even. Then, we bound the expression in  \eqref{eq:trianglineqcrosscommunity} by 
\begin{align*}
&p_1^{a+t +s }|Q_{1,1}|^t |Q_{1,2}|^{a + b - 1 + s}|p_1Q_{1,2} + p_2Q_{2,2}|^{|\leaves_2(K)|}
|p_1Q_{1,1} + p_2Q_{1,2}|^{|\leaves_1(K)|}\\
& \quad\quad\quad\quad\text{(Using \cref{eq:boundonpbeta11plusbeta22,eq:insecondcommunity12dominates})}\\
& \lesssim_D p_1^{a+t +s + |\leaves_2(K)|}|Q_{1,1}|^t |Q_{1,2}|^{a + b - 1 + s + |\leaves_2(K)|}
|p_1Q_{1,1} + p_2Q_{1,2}|^{|\leaves_1(K)|}\\
& \lesssim_D
p_1^{a+t +s + |\leaves_2(K)|-1}|Q_{1,1}|^t \times |Q_{1,2}|^{a + b - 1 + s + |\leaves_2(K)|}\times 
p_1 |p_1Q_{1,1} + p_2Q_{1,2}|^{|\leaves_1(K)|-\xi}|\times |Q_{1,2}|^\xi\\
    & \quad\quad\quad\quad\text{(Using \eqref{eq:starexpectationsint1outbrackets})}\\
& \lesssim
p_1^{a+t +s + |\leaves_2(K)|-1}|Q_{1,1}|^t \times |Q_{1,2}|^{a + b - 1 + s + |\leaves_2(K)|+\xi}\times 
|\Fourier_{\SBM(p,Q)}(\Star_{|\leaves_1(K)| - \xi})|\\
& \quad\quad\quad\quad\text{(Using $\sqrt{p_1}|Q_{1,1}|>|Q_{1,2}|$)}\\
& \lesssim p_1^{a+t +s + |\leaves_2(K)|-1}|Q_{1,1}|^t \times |Q_{1,1}|^{a + b - 1 + s + |\leaves_2(K)|+\xi}p_1^{\frac{1}{2}(a + b - 1 + s + |\leaves_2(K)|+\xi)}\times 
|\Fourier_{\SBM(p,Q)}(\Star_{|\leaves_1(K)| - \xi})|\\
& \lesssim 
p_1^{a+t +s + |\leaves_2(K)|-1 + \frac{1}{2}(a + b - 1 + s + (K)+\xi)}|Q_{1,1}|^{t +a + b - 1 + s + |\leaves_2(K)|+\xi }\times 
|\Fourier_{\SBM(p,Q)}(\Star_{|\leaves_1(K)| - \xi})|.
\end{align*}
\paragraph{Case 2.2.1)}
If 
\begin{align*}
    & a+t +s + |\leaves_2(K)|-1 + \frac{1}{2}(a + b - 1 + s + |\leaves_2(K)|+\xi)\\
    &\quad\quad\quad\quad\quad\ge 
    t +a + b - 1 + s + |\leaves_2(K)|+\xi \quad\quad\quad\Longleftrightarrow\\
    & a - 1 + s + |\leaves_2(K)|\ge b + \xi, 
\end{align*}
the above expression is at most 
\begin{align*}
    & (p_1 |Q_{1,1}|)^{t +a + b - 1 + s + |\leaves_2(K)|+\xi}\times |\Fourier_{\SBM(p,Q)}(\Star_{|\leaves_1(K)| - \xi})|\\
    & \quad\quad\quad\quad\text{(Using \cref{claim:4cycle2sbm})}\\
    & \lesssim 
    |\Fourier_{\SBM(p,Q)}(\cycle_4)|^{\frac{|V(H)| - |\leaves_1(K)| - 1 + \xi}{4}}\times |\Fourier_{\SBM(p,Q)}(\Star_{|\leaves_1(K)| - \xi})|\le \\
    & \le \max(|\Fourier_{\SBM(p,Q)}(\cycle_4)|^{1/4}, 
  |\Fourier_{\SBM(p,Q)}(\Star_{|\leaves_1(K)|-\xi})|^{1/(|\leaves_1(K)|-\xi + 1)})^{|V(H)|}.
\end{align*}
\paragraph{Case 2.2.2)}
Otherwise, $a - 1 + s + |\leaves_2(K)|< b + \xi,$ so $b + \xi = a + s + |\leaves_2(K)| + r,$ where $r\ge 0.$ Using the isolated-vertices argument, there are at least $b + a + r - \xi$ edges between $K$ and 
$\vertices(H)\backslash(K\cup \leaves(H)).$ We bound \eqref{eq:trianglineqcrosscommunity}
as

\begin{equation}
\label{eq:longlastequation}
\begin{split}
    &|\Fourier_{\SBM(p,Q)}(H)|\\
    & \lesssim_D 
    p_1^{a + t}|Q_{1,1}|^t\times |Q_{1,2}|^{b + a + r - \xi }\times 
    |Q_{2,2}|^{s}\times 
    |p_1Q_{1,1} + p_2Q_{1,2}|^{|\leaves_1(K)|}
    |p_1Q_{1,2} + p_2Q_{2,2}|^{|\leaves_2(K)|}\\
    & \quad\quad\quad\quad\text{(Using ${p_1}|Q_{1,2}|\gtrsim|Q_{2,2}|$) and \eqref{eq:boundonpbeta11plusbeta22} and \eqref{eq:insecondcommunity12dominates}}\\
    & \lesssim 
    p_1^{a + t}|Q_{1,1}|^t
    |Q_{1,2}|^{b + a + r -\xi}\times 
    |Q_{1,2}|^{s}p_1^s
    p_1^{-1}
    \times\\
    &\quad\quad\quad\quad\times
    p_1
    |p_1Q_{1,1} + p_2Q_{1,2}|^{|\leaves_1(K)|-\xi}\times  
    |p_1Q_{1,1} + p_2Q_{1,2}|^\xi
    |Q_{1,2}|^{|\leaves_2(K)|}
    p_1^{|\leaves_2(K)|}\\
    & \quad\quad\quad\quad\text{(Using \eqref{eq:starexpectationsint1outbrackets})}\\
    & \lesssim 
    p_1^{a + t}|Q_{1,1}|^t
    |Q_{1,2}|^{b + a + r -\xi}\times 
    |Q_{1,2}|^{s}p_1^s
    p_1^{-1}
    |\Fourier_{\SBM(p,Q)}(\Star_{|\leaves_1(K)|-\xi})|\times
    |Q_{1,2}|^{|\leaves_2(K)|+\xi}
    p_1^{|\leaves_2(K)|}\\
    & \lesssim 
    p_1^{a + t + s + |\leaves_2(K)|-1}\times 
    |Q_{1,1}|^{t}\times 
    |Q_{1,2}|^{b + a + r+ s +|\leaves_2(K)|}\times 
    |\Fourier_{\SBM(p,Q)}(\Star_{|\leaves_1(K)|-\xi})|\\
    & \quad\quad\quad\quad\text{(Using $p_1\ge |Q_{1,2}|$)}\\
    & \lesssim 
    p_1^{a + t + s + |\leaves_2(K)|-1}\times 
    |Q_{1,1}|^{t}\times 
    |Q_{1,2}|^{b + a + r+ s +|\leaves_2(K)|-3}\times p_1Q_{1,2}^2\times 
    |\Fourier_{\SBM(p,Q)}(\Star_{|\leaves_1(K)|-\xi})|\\
     & \quad\quad\quad\quad\text{(Using \eqref{eq:starexpectationsint1outbrackets})}\\
    & \lesssim 
    p_1^{a + t + s + |\leaves_2(K)|-1}\times 
    |Q_{1,1}|^{t}\times 
    |Q_{1,2}|^{b + a + r+ s +|\leaves_2(K)|-3}\times |\Fourier_{\SBM(p,Q)}(\Star_2)|\times 
    |\Fourier_{\SBM(p,Q)}(\Star_{|\leaves_1(K)|-\xi})|\\
    & \quad\quad\quad\quad\text{(Using $p_1\ge |Q_{1,2}|$)}\\
    & \lesssim 
    p_1^{a + t + s + |\leaves_2(K)|-1 + 1 + r-\xi}\times 
    |Q_{1,1}|^{t}\times 
    |Q_{1,2}|^{b + a + s +|\leaves_2(K)|-4 + \xi}\times\\
    & \quad\quad\quad\quad\quad\quad\quad\quad|\Fourier_{\SBM(p,Q)}(\Star_2)|\times 
    |\Fourier_{\SBM(p,Q)}(\Star_{|\leaves_1(K)|-\xi})|\\
    & \quad\quad\quad\quad\text{(Using $\sqrt{p_1}|Q_{1,1}|\ge |Q_{1,2}|$)}\\
    & \lesssim 
    p_1^{a + t + s + |\leaves_2(K)|+ r-\xi}\times 
    |Q_{1,1}|^{t + b + a + s +|\leaves_2(K)|-4 + \xi}\times 
    p_1^{\frac{1}{2}(b + a + s +|\leaves_2(K)|-4 + \xi)}\times\\
    & \quad\quad\quad\quad\quad\quad\quad\quad|\Fourier_{\SBM(p,Q)}(\Star_2)|\times 
    |\Fourier_{\SBM(p,Q)}(\Star_{|\leaves_1(K)|-\xi})|\\
    & \lesssim 
    p_1^{a + t + s + |\leaves_2(K)|+ r-\xi + \frac{1}{2}(b + a + s +|\leaves_2(K)|-4 + \xi)}\times 
    |Q_{1,1}|^{t + b + a + s +|\leaves_2(K)|-4 + \xi} 
    \times\\
    & \quad\quad\quad\quad\quad\quad\quad\quad|\Fourier_{\SBM(p,Q)}(\Star_2)|\times 
    |\Fourier_{\SBM(p,Q)}(\Star_{|\leaves_1(K)|-\xi})|.\\
\end{split}
\end{equation}
Now, observe that 
\begin{align*}
    & a + t + s + |\leaves_2(K)|+ r-\xi + \frac{1}{2}(b + a + s +|\leaves_2(K)|-4 + \xi)\ge\\
    &\quad\quad\quad\quad\ge
    t + b + a + s +|\leaves_2(K)|-4 + \xi\quad\quad\quad\Longleftrightarrow\\
    & 2r + a + s + |\leaves_2(K)| + 4 \ge 3\xi + b,
\end{align*}
which is true since $b + \xi = r + a + s + |\leaves_2(K)|, 2\xi \le 4.$ Furthermore, 
$$
t + b + a + s +|\leaves_2(K)|-4 + \xi = 
|\vertices(H)\backslash\leaves(K)| + 
|\leaves_2(K)|-4 + \xi\ge 
1 + 3 - 4\ge 0,
$$
since $|\leaves_2(K)|\ge 3$ in case 2.2) and
$|\vertices(H)\backslash\leaves(K)|\ge 1$ as $H$ is connected and has at least 3 leaves and any such graph has a non-leaf vertex, and $\xi\ge 0$ by definition.

Altogether, \eqref{eq:longlastequation} is bounded by  
\begin{align*}
    & (p_1|Q_{1,1}|)^{t + b + a + s +|\leaves_2(K)|-4 + \xi}
    \times 
    |\Fourier_{\SBM(p,Q)}(\Star_2)|\times 
    |\Fourier_{\SBM(p,Q)}(\Star_{|\leaves_1(K)|-\xi})|\\
    & \quad\quad\quad\quad\text{(Using \cref{claim:4cycle2sbm})}\\
    & \lesssim_D |\Fourier_{\SBM(p,Q)}(\cycle_4)|^{\frac{|V(H)| + \xi - |\leaves_1(K)|-4}{4}}\times 
    |\Fourier_{\SBM(p,Q)}(\Star_2)|\times 
    |\Fourier_{\SBM(p,Q)}(\Star_{|\leaves_1(K)|-\xi})|\\
    & = 
    \Bigg(|\Fourier_{\SBM(p,Q)}(\cycle_4)|^{1/4}\Bigg)^{|V(H)| + \xi - |\leaves_1(K)|-4}\times 
    \Bigg(|\Fourier_{\SBM(p,Q)}(\Star_2)|^{1/3}\Bigg)^3\\
    & \quad\quad\quad\times 
    \bigg(|\Fourier_{\SBM(p,Q)}(\Star_{|\leaves_1(K)|-\xi})|^{1/(|\leaves_1(K)|-\xi + 1)}\Bigg)^{|\leaves_1(K)|-\xi + 1}\\
    & \le \max\Big(
    \Psi_{\SBM(p,Q)}(\cycle_4),
    \Psi_{\SBM(p,Q)}(\Star_2),
    \Psi_{\SBM(p,Q)}(\Star_{|\leaves_1(K)|-\xi})\Big)^{|\vertices(H)|}.
\end{align*}
With this, the proof of \cref{thm:maximizingpartition2sbm} is complete.\qed

\subsubsection{Testing in 2-SBMs}
\begin{theorem}[Testing in 2-SBMs] Suppose that $\SBM(n;p,Q)$ is a stochastic block-model on two communities such that $\min(p_1,p_2) = \omega(1/n).$ Suppose that there exists a polynomial test of degree at most $D,$ for some even absolute constant $D\ge 4,$  distinguishing $\SBM(n;p,Q)$ and 
$\ergraphhalf$ with high probability. Then, one can also distinguish the two graph distribution with high probability via either the signed 4-cycle count or the signed count of a star on at most $D$ edges. 
\end{theorem}

\begin{remark} The condition $\min(p_1,p_2) = \omega(1/n)$ is very minimal.
Suppose that $p_1\le p_2.$
If $\min(p_1,p_2) = o(1/n),$ so $p_1 = o(1/n),$ then with high probability, there is no vertex of label 1, so $\SBM(n;p,Q)$ is (in total variation) close to 
$\mathbb{G}(n, (Q_{2,2}+1)/2).$ Trivially, the signed edge count is an optimal distinguisher between $\mathbb{G}(n,q)$ and $\ergraphhalf$ for any $q.$ If $p_1 = \Theta(1/n),$ then we can prove a slightly weaker result by blowing up the ``vertex-sample complexity'' as in \cref{def:vertexsamplecomplexity}. Namely, the result we will prove is that if there exists a polynomial test of degree at most $D$ for some even $D\ge 4$ distinguishing $\SBM(n;p,Q)$ and 
$\ergraphhalf,$ then one can distinguish 
$\SBM(N(n);p,Q)$ and $\mathbb{G}(N(n),1/2)$ via either the signed 4-cycle count or the signed count of a star on at most $D$ edges, where $N(n)$ is any function that satisfies $N(n) = \omega(n).$ The idea is that if $N(n) = \omega(n)$ and 
$p_1 = \Theta(1/n),$ then $p_1 = \omega(1/N(n)).$
\end{remark}
\begin{proof}
By \cref{thm:maximizingpartition2sbm}, we know that for any graph on at most $D$ edges without isolated vertices, $$
\Psi_{\SBM(p,Q)}(H)
\lesssim_D 
\max\Bigg(
\Psi_{\SBM(p,Q)}(\cycle_4),
\max_{1\le t \le D}
\Psi_{\SBM(p,Q)}(\Star_t)
\Bigg).
$$
If an approximate maximum on the right hand-side above is achieved by a 4-cycle or a star on at most $D/2$ edges, the conclusion follows by \cref{thm:beatingvarnullimpliesbeatingvarplanted}. Suppose instead that the maximum is achieved by some $\Star_t$ such that $t \in \{D/2 + 1,\ldots, D\}.$ 

\paragraph{Step 1: Identifying relationships between $p,Q.$} As in the proof of \cref{thm:maximizingpartition2sbm}, we first identify several relationships between $p$ and $Q.$

First, from \cref{thm:ldpimpliesbsubgraphcount}, 
\begin{align}
\label{eq:2sbmtestingtstarisenough}
\Psi_{\SBM(p,Q)}(\Star_t) = \omega(n^{-1/2}).
\end{align}
Second, it implies that $\Star_2$ is not an approximate maximizer, so 

\begin{align}
\label{eq:tstarbeats2star2sbm}
\Psi_{\SBM(p,Q)}(\Star_3) = o(
\Psi_{\SBM(p,Q)}(\Star_t)
).
\end{align}
Without loss of generality, let $p_1\le p_2.$ Let also $\lambda_i\coloneqq p_1Q_{1,i} + p_2Q_{2,i},$ so 
$\Fourier_{\SBM(p,Q)}(\Star_r) = p_1\lambda_1^r + p_2\lambda_2^r.$ Now, \eqref{eq:tstarbeats2star2sbm} implies that 
\begin{align*}
    & (p_1\lambda_1^2 + p_2\lambda_2^2)^{1/3} = 
    o\Big(
    (p_1\lambda_1^t + p_2\lambda_2^t)^{1/(t+1)}
    \Big) \Longrightarrow\\
    & (p_1\lambda_1^2 + p_2\lambda_2^2)^{(t+1)} = 
    o\Big(
    (p_1|\lambda_1|^t + p_2|\lambda_2|^t)^{3}
    \Big) \Longrightarrow\\
    & p_1^{t+1}|\lambda_1|^{2(t+1)} + 
    p_2^{t+1}|\lambda_2|^{2(t+1)}  = 
    o\Big(
p_1^{3}|\lambda_1|^{3t} + 
    p_2^{3}|\lambda_2|^{3t}
    \Big).
\end{align*}
Now, one of the following two systems of inequalities needs to be satisfied
\begin{enumerate}
    \item $p_2^{t+1}|\lambda_2|^{2(t+1)} = o( p_2^{3}|\lambda_2|^{3t})$ and 
    $ p_2^{3}|\lambda_2|^{3t} \ge  p_1^{3}|\lambda_1|^{3t}.$ Then, as $t = O(1),$
    the first inequality implies that $p_2^{t-2} = o(|\lambda_2|^{t-2}).$ Thus, 
    $$
    p_2 = o(|\lambda_2|) = o(|p_1Q_{1,2} + p_2Q_{2,2}|) = 
    o(p_1|Q_{1,2}| + p_2|Q_{2,2}|) = o(2p_2),
    $$
    contradiction. We used that $p_1\le p_2$ and $|Q_{i,j}|\le 1.$
    \item $p_1^{t+1}|\lambda_1|^{2(t+1)} = o( p_1^{3}|\lambda_1|^{3t})$ and 
    $ p_1^{3}|\lambda_1|^{3t} \ge  p_2^{3}|\lambda_2|^{3t}.$ As in the previous case, 
    the first inequality implies that $p_1 = o(|\lambda_1|).$ The second implies that $|\lambda_1| \ge  |\lambda_2|$ as $p_1\le p_2.$ 
\end{enumerate}
Altogether, we have learned that:
\begin{align}
\label{eq:2sbmtestingrelationships}
p_1\le p_2,  p_1 = \omega(1/n), p_1 = o(|\lambda_1|), \text{ and }
|\lambda_1|\ge |\lambda_2|.
\end{align}
This further implies the following two inequalities:
\begin{align}
\label{eq:1dominates}
    & p_1|\lambda_1|^s\ge p_2|\lambda_2|^s \text{ for any }s\ge t.
\end{align}
This is true since $\frac{ p_1|\lambda_1|^s}{p_2|\lambda_2|^s} = 
\frac{ p_1|\lambda_1|^t}{p_2|\lambda_2|^t}\times 
\Big(\frac{|\lambda_1|}{|\lambda_2|}\Big)^{s-t},
$ which is a product of terms at least equal to 1.

\begin{align}
    \label{eq:powerDdominates}
    |(p_1\lambda_1^s + p_2\lambda_2^s)|^{\frac{1}{s+1}} \gtrsim_D 
    |(p_1\lambda_1^t + p_2\lambda_2^t)|^{\frac{1}{t+1}} \text{ for any even $s\in [t, D].$}
\end{align}
 Indeed, from \eqref{eq:1dominates}, it is enough to show that 
$$
(p_1|\lambda_1|^s)^{\frac{1}{s+1}}\ge 
(p_1|\lambda_1|^t)^{\frac{1}{t+1}} \Longleftrightarrow
|\lambda_1|^{s-t}\ge 
p_1^{s-t}
$$
which follows from $s\ge t, p_1 = o(|\lambda_1|).$ 

Furthermore, \eqref{eq:2sbmtestingtstarisenough} implies that 
\begin{align}
\label{eq:lowerboundlambda1}
    |\lambda_1| = \omega\Big((p_1|\lambda_1|^t)^{/t+1}\Big) = 
    \omega\Big(|\Fourier_{\SBM(p,Q)}(\Star_t)|^{1/(t+1)}) = \omega(n^{-1/2}).
\end{align}
Finally, from \eqref{eq:powerDdominates}, we assume that $t = D.$

\paragraph{Step 2: Analysis of Variance of SBM.} We need to show that 
$$\Big|\expect_{\bfG\sim\SBM(p,Q)}\big[\signedcount_{\Star_D}(\bfG)\big]\Big| = 
\omega\Big(
\Var_{\bfG\sim\SBM(p,Q)}\big[\signedcount_{\Star_D}(\bfG)]^{1/2}
\Big).
$$
As in the proof of \cref{thm:beatingvarnullimpliesbeatingvarplanted}, 
take any graph $H$ which is isomorphic to $S_1\otimes  S_2$ where $S_1, S_2$ are both isomorphic to $\Star_D$ and share at least one vertex. Suppose that there are $M_H$ copies of the signed count of $H$ after expanding $\Var_{\bfG\sim\SBM(p,Q)}\big[\signedcount_{\Star_D}(\bfG)].$ Then, we have to show that 
$$
n^{2(D+1)}( p_1\lambda_1^D + p_2\lambda_2^D)^2= 
\omega\Big(
M_H \times \Fourier_{\SBM(p,Q)}(H)
\Big).
$$
We consider three possible cases for $H = S_1\otimes S_2:$
\paragraph{Case 1) $S_1$ and $S_2$ share their central vertex.} Then, $H$ is a star on $\ell\le 2D$ leaves. This means that $S_1$ and $S_2$ also share 
    $(2D-\ell)/2$ leaves. Altogether, $|\vertices(S_1)\cup \vertices(S_2)|= 1 + \ell +
    (2D-\ell)/2 = 1 + D + \ell/2.
    $ Hence, $M_H = \Theta(n^{1+D+\ell/2}).$ All we need to prove is that 
    \begin{equation}
    \label{eq:ineqinfirstcase}
    \begin{split}
        & n^{2(D+1)}( p_1\lambda_1^D + p_2\lambda_2^D)^2 = \omega(
        n^{1+D+\ell/2} 
        ( p_1\lambda_1^\ell + p_2\lambda_2^\ell)
        ) \Longleftrightarrow\\
        & n^{1 + D -\ell/2}\max(p^2_1\lambda_1^{2D}, p^2_2\lambda_2^{2D}) = 
        \omega( p_1\lambda_1^\ell + p_2\lambda_2^\ell).
    \end{split}
    \end{equation}
    Again, there are two cases. If $\ell \ge D,$ then by \eqref{eq:1dominates} the inequality is equivalent to 
    $$
    n^{1 + D -\ell/2}p^2_1\lambda_1^{2D} = \omega(
       p_1\lambda_1^{\ell}) \Longleftrightarrow
    n^{1 + D-\ell/2} 
    p_1\lambda_1^{2D - \ell}
    = \omega(1).
    $$
    Note that $p_1 = \omega(1/n)$ and  $|\lambda_1|=\omega(n^{-1/2})$ by \eqref{eq:lowerboundlambda1}, which is enough.

    If $\ell<D,$ then $|p_1\lambda_1^\ell + p_2\lambda_2^\ell| = |\Fourier_{\SBM(p,Q)}(\Star_\ell)| \lesssim_D |\Fourier_{\SBM(p,Q)}(\Star_D)|^{\frac{\ell + 1}{D+1}},$ so the inequality reduces to 
    $$
    n^{1 + D -\ell/2} |\Fourier_{\SBM(p,Q)}(\Star_D)|^{2 - \frac{\ell+1}{D+1}} = \omega(1),
    $$
    which follows immediately from \eqref{eq:2sbmtestingtstarisenough}.
\paragraph{Case 2) $S_1$ and $S_2$ do not have a common central vertex and do not have common edges.} Hence, $S_1\otimes S_2$ has at most $2(D+1)-1$ vertices (at $S_1, S_2$ need to share at least one vertex) and $2D$ edges. Hence, $M_H = O(n^{2D+1}).$ By \cref{thm:maximizingpartition2sbm} applied for $2D,$ $|\Fourier_{\SBM(p,Q)}(S_1\otimes S_2)|\lesssim_D 
    |\Fourier_{\SBM(p,Q)}(\Star_{2D})|.
    $ Hence, all we need to show is that 
    $$
    n^{2(D+1)}( p_1\lambda_1^D + p_2\lambda_2^D)^2 = \omega(
        n^{1+2D} 
        ( p_1\lambda_1^{2D} + p_2\lambda_2^{2D}).
    $$
    This is a special case of \eqref{eq:ineqinfirstcase} proved in Case 1) when $\ell = 2D.$
\paragraph{Case 3) $S_1$ and $S_2$ do not have a common central vertex but do have common edges.} Note that $S_1, S_2$ are stars and do not share their central vertex, they can have at most one common edge. Hence, $S_1\otimes S_2$ has at most $2D-2$ edges, so $|\Fourier_{\SBM(p,Q)}(S_1\otimes S_2)|\lesssim_D 
    |\Fourier_{\SBM(p,Q)}(\Star_{2D-2})|
    $ by \cref{thm:maximizingpartition2sbm} applied for $2D.$ As $S_1, S_2$ share an edge, they have at least two common vertices, so $|\vertices(S_1)\cup \vertices(S_2)|\le 2(D+1)-2\le 2D.$ Thus, $M_H = O(n^{2D}).$ Altogether, we need to prove that 
    $$
    n^{2(D+1)}( p_1\lambda_1^D + p_2\lambda_2^D)^2 = \omega(
        n^{2D} 
        ( p_1\lambda_1^{2D-2} + p_2\lambda_2^{2D-2}).
    $$
    This is a special case of \eqref{eq:ineqinfirstcase} proved in Case 1) when $\ell = 2D-2.$
\end{proof}

\section{Comparison Inequalities}
We note that all arguments in the current section apply more generally to any graphon instead of stochastic block model, provided no measurability issues occur.
\subsection{Cycle Comparisons: A Spectral Approach}
We prove the following theorem, which explains why signed triangles and 4-cycles are used for detecting stochastic block models, but larger cycles are not.
\begin{theorem}
\label{thm:largercycles}
For any $\SBM(p,Q)$ distribution and $t\ge 5,$
$$
\Psi_{\SBM(p,Q)}(\cycle_t)\le 
\Psi_{\SBM(p,Q)}(\cycle_4).
$$
\end{theorem}
\begin{proof} Consider the expression for a signed $t$-cycle count. It is given by 
\begin{equation*}
    \begin{split}
         \Fourier_{\SBM(p,Q)}(\cycle_t)
        &= 
        \sum_{x_1, x_2, \ldots, x_t}
        p_{x_1}p_{x_2}\cdots p_{x_t} Q_{x_1x_2}Q_{x_2x_3}\cdots Q_{x_tx_1}\\
        & = 
        \sum_{x_1, x_2, \ldots, x_t} 
        (\sqrt{p_{x_1}}Q_{x_1,x_2}\sqrt{p_{x_2}})
        (\sqrt{p_{x_2}}Q_{x_2,x_3}\sqrt{p_{x_3}})
        \cdots
        (\sqrt{p_{x_t}}Q_{x_t,x_1}\sqrt{p_{x_1}})\\
        & = \trace((\sqrt{P}Q\sqrt{P})^t),
    \end{split}
\end{equation*}
where $\sqrt{P}$ is the diagonal matrix with entries $(\sqrt{p_1}, \sqrt{p_2}, \ldots, \sqrt{p_k}).$ Let $\lambda_1, \lambda_2, \ldots, \lambda_k$ be the k eigenvalues of $\sqrt{P}Q\sqrt{P}.$ Then, 
$$
\Fourier_{\SBM(p,Q)}(\cycle_t) = \sum_{i = 1}^k \lambda_i^t.
$$
Now, for any $t\ge 5,$
$$
|\Fourier_{\SBM(p,Q)}(\cycle_t)|^{1/t} = 
\Big|\sum_{i = 1}^k \lambda_i^t\Big|^{1/t}\le 
(\sum_{i = 1}^k |\lambda_i|^t)^{1/t}\le
(\sum_{i = 1}^k |\lambda_i|^4)^{1/4} = 
|\Fourier_{\SBM(p,Q)}(\cycle_4)|^{1/4}.
$$
The first inequality is triangle inequality and the second monotonicity of $t\longrightarrow\|(\lambda_1, \ldots, \lambda_k)\|_t.$
 \end{proof}
\begin{remark}[Triangles vs 4-Cycles] The above proof also sheds light on when triangles are more informative than 4-cycles and vice versa. Namely, observe that if the positive (or negative) eigenvalues of $\sqrt{P}Q\sqrt{P}$ dominate, then 
$$
\Big|\sum_{i= 1}^k \lambda_i^3\Big|^{1/3}\approx
\Big|\sum_{i= 1}^k |\lambda_i|^3\Big|^{1/3}\ge 
\Big|\sum_{i= 1}^k \lambda_i^4\Big|^{1/4},
$$
so 
$|\Fourier_{\SBM(p,Q)}(\cycle_3)|^{1/3}\ge
c
|\Fourier_{\SBM(p,Q)}(\cycle_4)|^{1/4}.$
Whenever this is not the case (as in the quiet planted coloring distribution in \cite{kothari2023planted}, the $L_\infty$ random geometric graphs in \cite{bangachev2024detection}, certain non-PSD Gaussian random geometric graphs in \cite{bangachev2023random}), the mass on positive and negative eigenvalues should be relatively balanced. Heuristically, this is caused by certain bipartiteness/hyperbolic behavior of $\SBM(p,Q).$ 

A more detailed analysis of the performance of signed triangles and 4-cycles for testing \emph{against sparser \ER{}} as well is carried out in 
\cite{Jin2019OptimalAO}.
\end{remark}

\subsection{Exploiting Symmetries: A Sum-of-Squares Approach}
We prove the following theorem, which explains why signed $\complete_4^-$ counts are not used for detecting latent space structure, where $\complete_4^-$ is the graph on 4 vertices and 5 edges (say with vertex set$\{1,2,3,4\}$ and all edges but $(2,4)$). 
One can certainly generalize this approach to other graphs with enough symmetry. The argument appears implicitly in \cite[Lemma 4.10]{Liu2021APV}.
\begin{theorem}
\label{thm:4cyclewithextra}
For any $\SBM(p,Q)$ distribution,
$$
|\Fourier_{\SBM(p,Q)}(K_4^-)|\le 
|\Fourier_{\SBM(p,Q)}(\cycle_4)|.
$$
\end{theorem}
\begin{proof} Consider the $\cycle_4$ with edges $(12), (23), (34), (14)$ and the $K_4^-$ on these edges with the extra edge $(13).$ 
We will first rewrite the expression for $\Fourier_{\SBM(p,Q)}(\cycle_4).$ Note that 
\begin{equation}
    \begin{split}
        \Fourier_{\SBM(p,Q)}(\cycle_4) & =
        \expect_{\bfx_1, \bfx_2, \bfx_3,\bfx_4}[
        Q_{\bfx_1,\bfx_2}
        Q_{\bfx_2,\bfx_3}
        Q_{\bfx_3,\bfx_4}
        Q_{\bfx_4,\bfx_1}
        ]\\
        & = \expect_{\bfx_1,\bfx_3}\Big[
        \expect_{\bfx_2,\bfx_4}[Q_{\bfx_1,\bfx_2}
        Q_{\bfx_2,\bfx_3}
        Q_{\bfx_3,\bfx_4}
        Q_{\bfx_4,\bfx_1}|\bfx_1,\bfx_3]
        \Big].
    \end{split}
\end{equation}
Note that $\bfx_2, \bfx_4$ are independent even conditioned on $\bfx_1, \bfx_3.$ Hence, 
\begin{align*}
& \expect[Q_{\bfx_1,\bfx_2}
        Q_{\bfx_2,\bfx_3}
        Q_{\bfx_3,\bfx_4}
        Q_{\bfx_4,\bfx_1}|\bfx_1,\bfx_3]\\
&= 
   \expect[Q_{\bfx_1,\bfx_2}
        Q_{\bfx_2,\bfx_3}|\bfx_1,\bfx_3]\times 
        \expect[Q_{\bfx_1,\bfx_4}
        Q_{\bfx_4,\bfx_3}|\bfx_1,\bfx_3]\\
&= 
    \expect[Q_{\bfx_1,\bfx_2}
        Q_{\bfx_2,\bfx_3}|\bfx_1,\bfx_3]^2,  
\end{align*}
where we used the fact that the conditional expectation are identically distributed. Hence, we obtain that 
$$
\Fourier_{\SBM(p,Q)}(\cycle_4)  =
\expect\Bigg[
\expect[Q_{\bfx_1,\bfx_2}
        Q_{\bfx_2,\bfx_3}|\bfx_1, \bfx_3]^2
\Bigg].
$$
In the exact same way, we conclude that 
$$
\Fourier_{\SBM(p,Q)}(K^-_4)  =
\expect\Bigg[
\expect[Q_{\bfx_1,\bfx_2}
        Q_{\bfx_2,\bfx_3}|\bfx_1, \bfx_3]^2Q_{\bfx_1,\bfx_3}
\Bigg].
$$
Altogether, 
\begin{align*}
    & \Fourier_{\SBM(p,Q)}(\cycle_4) - \Fourier_{\SBM(p,Q)}(K^-_4) = 
    \expect\Bigg[
\expect[Q_{\bfx_1,\bfx_2}
        Q_{\bfx_2,\bfx_3}|\bfx_1, \bfx_3]^2(1-Q_{\bfx_1,\bfx_3})
\Bigg]\ge 0,\\
    & \Fourier_{\SBM(p,Q)}(\cycle_4) + \Fourier_{\SBM(p,Q)}(K^-_4) = 
    \expect\Bigg[
\expect[Q_{\bfx_1,\bfx_2}
        Q_{\bfx_2,\bfx_3}|\bfx_1, \bfx_3]^2(1+Q_{\bfx_1,\bfx_3})
\Bigg]\ge 0,
\end{align*}
where we used that $-1\le Q_{\bfx_1,\bfx_3}\le 1.$ 
Together, the two inequalities give the desired result.
\end{proof}

\subsection{Ghost Vertices: A Second Moment Approach}
The key idea in the proof of \cref{thm:4cyclewithextra} was the symmetry around $(1,3).$
One can ``artificially'' create such symmetry via a second moment argument, but this unfortunately yields comparison inequalities too weak for the purposes of \eqref{eq:statisticalscaling}. Here, we present one possible result. The proof follows implicitly \cite[Fact 12]{chung87} and \cite[(4) in Section 3]{Liu2021APV}.
\begin{theorem}
\label{thm:chisquaredapproach}
Consider any graph $H$ and suppose that it has a vertex of degree $d.$ Then, for any $\SBM(p,Q)$ distribution,
$$
|\Fourier_{\SBM(p,Q)}(H)|\le
|\Fourier_{\SBM(p,Q)}(\complete_{2,d})|^{1/2}.
$$
\end{theorem}
\begin{proof} Let the vertex set of $H$ be $\{1,2,\ldots, h\}$ such that $h$ has degree $d.$ Then,
\begin{equation*}
\begin{split}
    & |\Fourier_{\SBM(p,Q)}(H)|
     = \Bigg|\expect\Bigg[\prod_{(i,j)\in \edges(H)}Q_{\bfx_i\bfx_j}\Bigg]\Bigg|\\
     & 
    = \Bigg|\expect\Bigg[ 
    \prod_{i,j<h\;:\; (i,j)\in \edges(H)}
    Q_{\bfx_i,\bfx_j}
    \times
    \expect\Big[
    \prod_{k\; : \; (k,h)\in \edges(H)}Q_{\bfx_k\bfx_h}\Big|x_1, x_2, \ldots, x_{h-1}\Big]
    \Bigg]\Bigg|\\
    & \le
    \expect\Bigg[\Bigg| 
    \prod_{i,j<h\;:\; (i,j)\in \edges(H)}
    Q_{\bfx_i,\bfx_j}
    \times
    \expect\Big[
    \prod_{k\; : \; (k,h)\in \edges(H)}Q_{\bfx_k\bfx_h}\Big|x_1, x_2, \ldots, x_{h-1}\Big]
    \Bigg|\Bigg]\\
    & \le
    \expect\Bigg[\Bigg| 
    \expect\Big[
    \prod_{k\; : \; (k,h)\in \edges(H)}Q_{\bfx_k\bfx_h}\Big|x_1, x_2, \ldots, x_{h-1}\Big]\Bigg|\Bigg]\\
    & \le
    \expect\Bigg[
    \expect\Big[
    \prod_{k\; : \; (k,h)\in \edges(H)}Q_{\bfx_k\bfx_h}\Big|x_1, x_2, \ldots, x_{h-1}\Big]^2\Bigg]^{1/2}\\
    & =
    \expect_{\bfx_1,\bfx_2,\ldots,\bfx_{h-1}, \bfx_{h'}, \bfx_{h''}}\Bigg[
    \prod_{k\; : \; (k,h)\in \edges(H)}Q_{\bfx_k\bfx_{h'}}Q_{\bfx_k\bfx_{h''}}
    \Bigg]^{1/2} = |\Fourier_{\SBM(p,Q)}(\complete_{2,d})|^{1/2}.\qedhere
\end{split}
\end{equation*}
\end{proof}

\section{Future Directions}
\label{sec:discussion}
Beyond the natural direction of proving \cref{conj:optimality} for all stochastic block models (and, more generally, graphons), we outline several different directions of interest. They modify the conditions in the current work in different ways.

\paragraph{1. Beyond dense graphs: testing against a sparse \ER.} One natural direction is to develop a quasirandomness theory for testing against $\ergraph$ when $q$ depends on $n,$ say $q = n^{-\beta}$ for some absolute constant $\beta\in (0,1].$ The
$q$-biased Fourier coefficients are
$$
\Fourier^q_{\SBM(p,Q)}(H)\coloneqq \expect\Bigg[
\prod_{(i,j)\in \edges(H)}
\frac{(\bfG_{ij} - q)}{\sqrt{q(1-q)}}\Bigg]\asymp
n^{\frac{\beta \times |\edges(H)|}{2}}\times 
\expect\Bigg[
\prod_{(i,j)\in \edges(H)}
{(\bfG_{ij} - q)}\Bigg],
$$ where the asymptotic equivalence holds for constant sized graphs $H.$ 
The equivalent condition to \eqref{eq:statisticalscaling} is 
$|\Fourier^q_{\SBM(p,Q)}(H)|^{\frac{1}{|\vertices(H)|}} = \omega(n^{-1/2}).$ What makes this problem different is the $n^{\frac{\beta \times |\edges(H)|}{2|\vertices(H)|}}$ term in $ |\Fourier^q_{\SBM(p,Q)}(H)|^{\frac{1}{|\vertices(H)|}}.$
This term poses several challenges, including the argument on the variance of the planted model in \cref{thm:beatingvarnullimpliesbeatingvarplanted}. One viable approach seems the consideration of \emph{balanced} graphs which \cite{dhawan2023detection} use exactly in the setting of planting a dense subgraph in $\ergraph.$ In a balanced graph $H,$ the key quantity
$\frac{|\edges(H)|}{|\vertices(H)|}$ is larger than the same quantity for any subgraph $H'$ of  $H.$

\paragraph{2. Beyond SBMs: testing against vertex-transitive distributions.} The question of developing quasirandomness theory for testing against $\ergraphhalf$ is, of course, not restricted to stochastic block models as in our work and planted subgraph models as in \cite{yu2024counting}. In fact, one can even ask the question for fixed graphs as in the original quasrirandomness work of \cite{chung87}.

One way to phrase \cref{problem:maximizingfourriercoefficients} for a fixed graph is as follows. Take any fixed graph $G$ and let $\Pi_G$ be the distribution formed by applying a uniformly random vertex permutation to $G.$ Find the possible approximate maximizers of 
$
H\longrightarrow \Psi_{\Pi_G}(H)
$
where $H$ ranges over constant-sized graphs without isolated vertices.

\paragraph{3. Beyond constant degree tests: towards computational hardness.} Our results as well as the results of \cite{yu2024counting} apply only to indistinguishability against constant-degree polynomial tests. However, it would be useful to have results for polynomials of higher-degree. Especially desirable would be hardness against degree $D = \omega(\log n)$ polynomial tests as this is frequently viewed as strong (even though by no means perfect) evidence for computational hardness.

Our current framework does not allow for any super-constant $D$ since one can check that the implicit constants in $\gtrsim_D$ are on the order of $2^{\Theta(D\log D)}$ (when enumerating over all graphs on at most $D$ edges in \cref{thm:ldpimpliesbsubgraphcount} and taking max over $K$ in \cref{thm:maximizinginnonnegative,thm:maximizingpartition2sbm}). Explicitly, this means that when $D = \omega(1),$ the different $\gtrsim_D$ factors no longer indicate inequalities up to absolute constants. Yet, if we allow for an $n^{o(1)}$ blow-up in the sample complexity, the same techniques with a more careful bookkeeping of the $\gtrsim_D$ dependence still apply. To illustrate, useful is the following sample-complexity formulation of \eqref{eq:hypothesistest}.

\begin{definition}[``Vertex''-Sample-Complexity Perspective of Testing against \ER]
\label{def:vertexsamplecomplexity}
Consider a sequence of stochastic block-models $\bigg\{\SBM(p^k, Q^k)\bigg\}_{k \in \mathbb{N}}$ such that each coordinate of $p^k$ is positive and $Q^k$ is not the zero matrix. Find the minimal number of vertices $n(k)$ such that one can test between 
$\SBM(n(k);p^k, Q^k)$ and $\mathbb{G}(n(k), 1/2)$ 
via a degree-$D$ polynomial test with success probability $1 - o_k(1).$
\end{definition}

With a small blow-up in the sample-complexity, we can show the following theorem. The proofs are identical, observing that the hidden factors are on the order of $2^{O(D\log D)}$ everywhere. Hence, if $D = o(\log n/\log\log n),$ the hidden factors are of order $n^{o(1)}.$
\begin{proposition}
\label{prop:lognoverloglognversion}
Suppose that there is a degree $D = o(\log n/\log\log n)$ polynomial test that succeeds in distinguishing $\SBM(n(k); p^k, Q^k)$ and  
$\mathbb{G}(n(k), 1/2)$ with success probability $1-o_k(1).$ Suppose furthermore that the family $\SBM(n(k);p^k, Q^k)$ belongs to one of the four cases $i\in \{1,2,3,4\}$ described in \cref{thm:mainintro}. Then, for some appropriate $\epsilon_k = o_k(1),$ there exists some signed subgraph count of $H\in \mathcal{A}^i_D$ that distinguishes 
$\SBM(n(k)^{1+\epsilon_k}; p^k, Q^k)$ and  
$\mathbb{G}(n(k)^{1+\epsilon_k}, 1/2)$ with success probability $1-o_k(1).$
\end{proposition}

The reason why we write all the proofs for constant degree instead of $o(\log n/\log\log n)$ is that while we are not aware of any advantage of the $o(\log n/\log\log n)$ tests, there is significant notational simplicity when hiding factors depending on the degree. Of course, if one manages to push the result to degree $\Theta(\log n),$ this would have the significant advantage of capturing spectral methods. However, this seems to be not merely a matter of carefully keeping track of the constants in our proofs. If any form of \cref{conj:optimality} holds true for degree $D = \omega(\log n)$ polynomial tests, it seems to require essentially new techniques.

\begin{remark}[Completeness of \cref{def:vertexsamplecomplexity}]
A-priori, it is not even clear that $n(k)$ in \cref{def:vertexsamplecomplexity} exists. Yet, a simple argument mimicking \cref{claim:4cycle2sbm} and \cref{lem:4cycleinnonvanishing} shows that 
$$
\Fourier_{\SBM(p,Q)}(\cycle_4) = 
\Omega\Big(\max_{i,j}
p_i^2p_j^2 |Q_{i,j}|^4
\Big)
$$
for any $\SBM(p,Q).$
As the variance under planted scales as $O_k(n^7) = o_k(n^{2|\vertices(\cycle_4)|}),$ one can use this to show that for large enough $n(k),$ one can always test via the signed 4-cycle count.
\end{remark}

\paragraph{4. Beyond complete graphs: ``edge''-sample-complexity of testing.} A different sample complexity perspective of graph hypothesis testing was introduced in \cite{mardia24ldpquery} and further analyzed in \cite{bangachev24fourier}. The goal is to capture the query-complexity of low-degree polynomial tests. We take the following formulation as described in \cite{bangachev24fourier}. 

For a mask $\mathcal{M}\in \{\textsf{view, hide}\}^{N\times N}$ and (adjacency) matrix  $A\in \{0,1\}^{N\times N},$ denote by $A\odot\mathcal{M}$ the $N\times N$ array in which 
$(A\odot \mathcal{M})_{ji}= A_{ji}$ whenever $\mathcal{M}_{ji} = \textsf{view}$ and $(A\odot \mathcal{M})_{ji}= \,\,?$ whenever $\mathcal{M}_{ji} = \textsf{hide}.$ Testing between graph distributions with masks corresponds to a non-adaptive edge query model. Instead of viewing a full graph, one can choose to observe a smaller more structured set
$\mathcal{M}$
of edges in order to obtain a more data-efficient algorithm. 
The number of $\textsf{view}$ entries $|\mathcal{M}|$ of $\mathcal{M}$ is a natural proxy for ``sample complexity'' in the case of low-degree polynomials as the input variables of low-degree polynomials are edges rather than vertices.
Instead of asking for $n(k)$ in \cref{def:vertexsamplecomplexity}, one can ask for the size of the optimal mask $\mathcal{M}(k).$

\begin{definition}[``Edge''-Sample-Complexity Perspective of Testing against \ER]
\label{def:edgesamplecomplexity}
Consider a sequence of stochastic block-models $\bigg\{\SBM(p^k, Q^k)\bigg\}_{k \in \mathbb{N}}$ such that each coordinate of $p^k$ is positive and $Q^k$ is not the zero matrix. Find the minimal number of edges $M_k$ such that there exists some $N(k)\in \mathbb{N}$ and a 
mask $\mathcal{M}_k$ of size $|\mathcal{M}_k| = M_k$ on $N(k)$ vertices with the following property. 
One can test between 
$\mathcal{M}_k\odot\SBM(N(k);p^k, Q^k)$ and $\mathcal{M}_k\odot\mathbb{G}(N(k), 1/2)$ 
via a degree-$D$ polynomial test with success probability $1 - o_k(1).$
\end{definition}

Again, one can ask for a quasirandomness criterion in the case of edge-query-complexity. A theorem due to Alon \cite{Alon1981OnTN} shows that the maximal number of graphs isomorphic to $H$ in a graph on $M$ edges is $\Theta_H(M^{\frac{|\vertices(H)| + \delta(H)}{2}}),$ where $$\delta(H)\coloneqq \max_{S\subseteq \vertices(H)}|S| - |\{j \in \vertices(H)\; : \; \exists i \in S \text{ s.t. }(j,i)\in \edges(H)\}|.$$ 
Reasoning as in \cref{sec:introformalizingsubgraphsbmpartition}, instead of finding the approximate maximizers of\linebreak
$H \longrightarrow |\Fourier_{\SBM(p,Q)}(H)|^\frac{1}{|\vertices(H)|},$ one should look for the approximate maximizers of 
$$H \longrightarrow |\Fourier_{\SBM(p,Q)}(H)|^\frac{1}{|\vertices(H)| + \delta(H)}.$$ Needless to say, different techniques than ours are required.

\paragraph{5. The Sign of Fourier Coefficients of Stochastic Block Models.} The argument of \cref{lem:4cycleinnonvanishing} shows that for any $\SBM(p,Q)$ such that $Q$ is not the zero matrix, 
$\Fourier_{\SBM(p,Q)}(\cycle_4)>0.$ \emph{What other graphs $H$ besides $\cycle_4$ have this property?} The argument in 
\cref{lem:4cycleinnonvanishing} applies verbatim for any other graph $\complete_{2,d}$ when $d$ is even. The examples in \cref{appendix:library} show that $H$ does not necessarily satisfy this property if $H$ has an odd degree vertex (\cref{examplethm:notodddegree}), $H$ has an odd number of edges (\cref{examplethm:fourcyclenostarnotriangle}), $H$ is not 2-connected (\cref{examplethm:trianglenostarnotfourcycle}), or $H$ is not bipartite (\cref{examplethm:largestarsnotrianglenofourcycle}). 

What if we relax the inequality: \emph{what graphs $H$ have the property that $\Fourier_{\SBM(p,Q)}(H)\ge 0$ for any $\SBM(p,Q)$?} The argument in \cref{thm:largercycles} shows that all cycles of even length, in addition to the graphs $\complete_{2,d}$ satisfy this property.

\section*{Acknowledgments}
We thank Hannah Munkhbat for participating in early stages of this work. 

\bibliography{ref}
\bibliographystyle{alpha}

\appendix

\section{Library of Examples}
\label{appendix:library}
We will repeatedly refer to the following condition.
\begin{proposition}
\label{prop:fullyunbiased}
We call an $\SBM(p,Q)$ distribution on $k$ communities \emph{fully unbiased} if for any $x\in [k],$ it holds that $\sum_y p_y Q_{x,y} = 0.$
For any fully unbiased $\SBM(p,Q)$ and graph $H$ with a leaf (in particular, any tree), $\Fourier_{\SBM(p,Q)}(H) = 0.$ 
\end{proposition}
\begin{proof} Suppose that $H$ has $h$ vertices and vertex $h$ is a leaf with parent $h-1.$ By \cref{prop:explicitFourier}, 
\begin{align*}
&
\Fourier_{\SBM(p,Q)}(H)
 = 
\sum_{x_1, x_2, \ldots, x_h\in [k]}\Bigg(
\prod_{i = 1}^h p_{x_i}
\times
\prod_{(i,j)\in \edges(H)}
Q_{x_i,x_j}\Bigg)\\
& = 
\sum_{x_1, x_2, \ldots, x_{h-1}\in [k]}\Bigg(
\prod_{i = 1}^{h-1} p_{x_i}
\times
\prod_{(i,j)\in \edges(H)\backslash \{(h-1,h)\}}
Q_{x_i,x_j}\times\sum_{x_h\in [k]}Q_{x_{h-1},x_h}\Bigg)  = 0,
\end{align*}
where we use the fully unbiased condition for $x = x_{h-1}.$
\end{proof}

Throughout, we only verify condition \eqref{eq:statisticalscaling}, but do not compute the variance of the planted model.

\begin{theorem}[Failure of All Signed Trees]
\label{examplethm:notodddegree}
There exists an SBM distribution $\SBM(n;p,Q)$ on $k=2$ communities with $p = (1/2,1/2)$ which has the following properties.
For any graph $H$ with an odd degree vertex,
$\Fourier_{\SBM(p,Q)}(H) = 0.$ 
For any graph $H$ without odd degree vertices,
$\Fourier_{\SBM(p,Q)}(H) = 1.$ 
Furthermore, as $n$ grows,  $\SBM(n;p,Q)$ can be distinguished from $\ergraphhalf$ with high probability via the signed 4-cycle count, for example. Finally, 
any signed subgraph count distinguishes $\SBM(n;p, |Q|)$ and $\ergraphhalf$ with high probability as $n$ grows.
\end{theorem}
\begin{proof}[Construction] Take $k = 2, p = (1/2, 1/2)$
and $Q = \begin{pmatrix}1 & -1\\-1 & 1\end{pmatrix}.$
\end{proof}

\begin{theorem}[Balanced 2-SBMs in which 1-Stars Dominate]
\label{examplethm:1starsnotrianglenofourcycle}
There exists a stochastic block model $\SBM(n; p,Q)$ on 2 balanced communities with the following property. It can be distinguished from $\ergraphhalf$ with high probability via the signed count of 1-stars, but not via the signed counts of any other connected
graph of constant size.
\end{theorem}
\begin{proof}[Construction]
Consider the following SBM model on $k = 2$ communities. 
$p = (1/2,1/2)$ and 
$Q = \begin{pmatrix}
     n^{-\beta} & 0 \\
     0 & n^{-\beta}
\end{pmatrix}$ where 
$\beta \in (3/4,1).$
For any connected graph $H,$
$\Fourier_{\SBM(n;p,Q)}(H) = O(n^{-\beta |\edges(H)|}).$ Hence, if $H$ has at least 2 edges and is connected, 
$\Psi_{\SBM(n;p,Q)}(H) = O(n^{-\beta \frac{|\edges(H)|}{|\vertices(H)|}}) = O(n^{-2\beta/3}) = o(n^{-1/2})$ as $\beta \ge 3/4.$ Then, $\Fourier_{\SBM(n;p,Q)}(\Star_1) = \Theta(n^{-\beta})$ and $\Psi_{\SBM(n;p,Q)}(\Star_1) = \Theta(n^{-\beta/3}) = \omega(n^{-1/2})$ as $\beta\le 1.$ 
\end{proof}

\begin{theorem}[Balanced 2-SBMs in which 2-Stars Dominate]
\label{examplethm:2starsnotrianglenofourcycle}
There exists a stochastic block model $\SBM(n; p,Q)$ on 2 balanced communities with the following property. It can be distinguished from $\ergraphhalf$ with high probability via the signed count of 2-stars, but not via the signed counts of any other connected
graph of constant size.
\end{theorem}
\begin{proof}[Construction]
Consider the following SBM model on $k = 2$ communities. 
$p = (1/2,1/2)$ and 
$Q = \begin{pmatrix}
     n^{-\beta} & 0 \\
     0 & -n^{-\beta}
\end{pmatrix}$ where 
$\beta \in (2/3,3/4).$
For any connected graph $H,$
$\Fourier_{\SBM(n;p,Q)}(H) = O(n^{-\beta |\edges(H)|}).$ Hence, if $H$ has at least 3 edges and is connected, 
$\Psi_{\SBM(n;p,Q)}(H) = O(n^{-\beta \frac{|\edges(H)|}{|\vertices(H)|}}) = O(n^{-3\beta/4}) = o(n^{-1/2})$ as $\beta \ge 2/3.$ Then, $\Fourier_{\SBM(n;p,Q)}(\Star_2) = \Theta(n^{-2\beta})$ and $\Psi_{\SBM(n;p,Q)}(\Star_2) = \Theta(n^{-2\beta/3}) = \omega(n^{-1/2})$ as $\beta\le 3/4.$ Finally, 
$\Fourier_{\SBM(n;p,Q)}(\Star_1) = 0.$
\end{proof}

\begin{theorem}[2-SBMs in which Large Stars Dominate]
\label{examplethm:largestarsnotrianglenofourcycle}
Let $D\ge 3$ be a natural number. 
There exists a stochastic block model $\SBM(n; p,Q)$ on 2 communities with the following property. 
It can be distinguished from $\ergraphhalf$ with high probability via the signed count of $D$-stars, but not via the signed count of any other 
connected graph on at most $D$ edges.
\end{theorem}
\begin{proof}[Construction]
Consider the following SBM model on $k = 2$ communities. 
$p = (n^{-\alpha},1-n^{-\alpha})$ and 
$Q = \begin{pmatrix}
     0 & n^{-\beta} \\
     n^{-\beta} & 0
\end{pmatrix}$ where 
$\alpha= 3/4, \beta = \frac{2D-1}{4D} -\frac{1}{8D(D-1)}.$
Let $H$ be any connected graph. If $H$ is not bipartite, clearly $\Fourier_{\SBM(p,Q)}(H) = 0.$ Now, suppose that $H$ is bipartite and has parts of sizes $u$ and $v$ where $1\le u\le v$ and $u+v = |\vertices(H)| \le D+1.$ A simple calculation shows that 
$$
\Fourier_{\SBM(p,Q)}(H) = 
\Theta(
p_1^u n^{- \beta |\edges(H)|})
= 
\Theta(n^{-\alpha u - \beta |\edges(H)|})
\Longrightarrow
\Psi_{\SBM(p,Q)}(H) = 
\Theta(n^{-\alpha\frac{u}{|\vertices(H)|} - \beta\frac{|\edges(H)|}{|\vertices(H)|} }).
$$
Now, suppose first that $H$ is not a tree. A non-tree bipartite graph on at most $D$ edges has at most $D$ vertices and each part in the bipartition is of size at least 2. Thus,
$$
\Psi_{\SBM(p,Q)}(H) = 
O(n^{-\frac{3}{2D} - \frac{2D-1}{4D} + \frac{1}{8D(D-1)}})
O(n^{-\frac{1}{2} - \frac{10D - 11}{8D(D-1)}})
=o(n^{-1/2}). 
$$
Next, if $H$ is a tree and $u\ge 2,$
$$
\Psi_{\SBM(p,Q)}(H) =  
O(n^{-\frac{3}{2}\frac{1}{|\vertices(H)|} - \beta \times 
\frac{|\vertices(H)|-1}{|\vertices(H)|}
}).
$$
As $\beta<3/2,$ the convex combination $\frac{3}{2}\frac{1}{|\vertices(H)|} + \beta \times 
\frac{|\vertices(H)|-1}{|\vertices(H)|}$ decreases with $|\vertices(H)|,$ hence is maximized at $|\vertices(H)| = D+1,$ where
$$
n^{-\frac{3}{2}\frac{1}{|\vertices(H)|} - \beta \times 
\frac{|\vertices(H)|-1}{|\vertices(H)|}
} = 
n^{-\frac{3}{2}\times \frac{1}{D+1} - (\frac{2D-1}{4D} -\frac{1}{8D(D-1)}) \times 
\frac{D}{D+1}
} =
n^{-\frac{1}{2}- \frac{6D - 7}{8(D-1)(D+1)}}
=o(n^{-1/2}).
$$
Finally, if $H$ is a tree and $u = 1,$ we obtain 
$$
\Psi_{\SBM(p,Q)}(H) =  
\Theta(n^{-\frac{3}{4}\frac{1}{|\vertices(H)|} - \beta \times 
\frac{|\vertices(H)|-1}{|\vertices(H)|}
}).
$$
When $|\vertices(H)| = D+1,$ the case of $\Star_D,$
$$
n^{-\frac{3}{4}\frac{1}{|\vertices(H)|} - \beta \times 
\frac{|\vertices(H)|-1}{|\vertices(H)|}
} = 
n^{-\frac{3}{4}\frac{1}{D+1} - 
(\frac{2D-1}{4D} -\frac{1}{8D(D-1)})\frac{D}{D+1}
} = 
n^{-\frac{1}{2} + \frac{1}{8(D-1)(D+1)}} = \omega(n^{-1/2}).
$$
When $|\vertices(H)| <D+1,$ the case of smaller stars,
$$
n^{-\frac{3}{4}\frac{1}{|\vertices(H)|} - \beta \times 
\frac{|\vertices(H)|-1}{|\vertices(H)|}
}\le 
n^{-\frac{3}{4}\frac{1}{D} - 
(\frac{2D-1}{4D} -\frac{1}{8D(D-1)})\frac{D-1}{D}
} = 
n^{-\frac{1}{2} - \frac{1}{8D^2}} = o(n^{-1/2}).
$$
Thus, 
$\Star_D$ is the only connected graph on at most $D$ edges such that
$\Psi_{\SBM(n;p,Q)} = \omega(n^{-1/2}).$
\end{proof}

\begin{theorem}[2-SBMs in which 4-Cycles Dominate]
\label{examplethm:fourcyclenostarnotriangle}
There exists an SBM distribution $\SBM(p,Q)$ with the following properties. All Fourier coefficients corresponding to graphs with at least one leaf (including all stars) are zero and the Fourier coefficients corresponding to graphs with an odd number of edges (including triangles) are zero. Furthermore,
the 4-cycle is the unique connected graph of constant size 
which can distinguish $\SBM(n;p,Q)$ and $\ergraphhalf$ with high probability.
\end{theorem}

\begin{proof}[Construction] The construction is inspired by the ``quiet planting'' distribution in \cite{kothari2023planted}. Take some $q\in \mathbb{N}$ such that $q = \omega(n^{5/8}), q = o(n^{2/3})$ and consider an SBM on $k = q^2$ communities labeled by $[q]\times[q],$ where $p_{(a,b)} = 1/q^2\quad\forall (a,b)\in [q]\times [q]$ and
\begin{align*}
    Q_{(a_1,b_1),(a_2,b_2)} = 
    \begin{cases}
     1, & a_1 = b_1 \text{ and }a_2 \neq b_2,\\
    -1, & a_1 \neq b_1 \text{ and }a_2 = b_2,\\
    \ 0, & \text{otherwise}.
    \end{cases}
\end{align*}
The Fourier coefficient corresponding to any graph with a leaf (including stars) is 0 as the SBM is fully unbiased as in \cref{prop:fullyunbiased}.

The Fourier coefficient corresponding to any graph with an odd number of edges (including triangles) is also 0. This follows since the measure-preserving map on labels $(a,b)\longrightarrow (b,a)$ acts by $Q_{b,a} = - Q_{a,b}.$ This is enough via \eqref{eq:explicitfourier}.

Finally, the Fourier coefficient of any connected graph $H$ is $O(q^{1-|\vertices(H)|})$ and the Fourier coefficient of the signed 4-cycle is $\Theta(q^{-3}).$ In particular, this means that 
if $H$ is connected, $\Psi_{\SBM(p,Q)}(H) = O(q^{-\frac{|\vertices(H)|-1}{|\vertices(H)|}}).$ As $q = \omega(n^{5/8}),$ this means that 
$\Psi_{\SBM(p,Q)}(H) = o(n^{-1/2})$ if $|\vertices(H)|\ge 5.$ On the other hand, 
$\Psi_{\SBM(p,Q)}(\cycle_4) = \Theta(q^{-3/4}) = \omega(n^{-1/2})$ as $q = o(n^{2/3}).$ Finally, 
any other connected graph $H$ on at most 4 vertices has either an odd number of edges or an odd degree vertex, so 
$\Psi_{\SBM(p,Q)}(H) = 0.$
\end{proof}

\begin{theorem}[2-SBMs in which Triangles Dominate]
\label{examplethm:trianglenostarnotfourcycle}
There exists a stochastic block model $\SBM(n; p,Q)$ which can be distinguished from $\ergraphhalf$ with high probability via the signed triangle count, but not via the signed count of any other connected graph on a constant number of vertices.
\end{theorem}
\begin{proof}[Construction] One such construction appears in \cite{kothari2023planted} and with a slight modification in \cite{bangachev24fourier}. We give full detail for completeness. 
Take $k = n^{\alpha}$ for some $\alpha \in (2/3, 3/4).$ 
Consider an SBM on $k$ communities with $p_i = 1/k = n^{-\alpha}\quad \forall i\in [k]$ and 
\begin{equation*}
    Q_{i,j} = \begin{cases}
          1, & i = j,\\
         -\frac{1}{k-1}, & i \neq j.
    \end{cases}
\end{equation*}
Let $H$ be any connected graph.
One can again observe that the SBM is fully unbiased, which implies
that $\Fourier_{\SBM(p,Q)}(H) = 0$ if $H$ is a tree
via \cref{prop:fullyunbiased}. Similarly, one can show that if $H$ has any edge the removal of which makes the graph disconnected (i.e., $H$ is not 2-connected), $\Fourier_{\SBM(p,Q)}(H) = 0.$ 

Finally, if $H$ is 2-connected, that is $H$ is connected and continues to be so after the removal of any edge, one can show that $\Fourier_{\SBM(p,Q)}(H)\asymp k^{1-|\vertices(H)|}.$ Hence, 
$\Psi_{\SBM(p,Q)}(H)\asymp k^{-\frac{|\vertices(H)|-1}{|\vertices(H)|}}.$ This quantity is decreasing in $|\vertices(H)|.$ 
As $k  = o(n^{3/4})$ and $k = \omega(n^{2/3}),$ the signed triangle count is the unique signed count of a connected constant-sized graph sufficient for testing.
\end{proof}

\begin{theorem}[One-to-One Comparisons of Fourier Coefficients Fail]
\label{examplethm:onetoneinsufficient}
There exist two stochastic block models $\SBM(n;p^1,Q^1)$ and 
$\SBM(n;p^2, Q^2)$ on 2 communities and a connected graph $H$ on $6$ vertices and 8 edges with the following two properties:
\begin{enumerate}
    \item In $\SBM(n;p^1,Q^1),$ 
    $
    \Psi_{\SBM(n;p^1,Q^1)}(\Star_t) = 
    o(\Psi_{\SBM(n;p^1,Q^1)}(H))
    $ for any $t\in \mathbb{N}.$
    \item In $\SBM(n;p^2,Q^2),$
    $
    \Psi_{\SBM(n;p^2,Q^2)}(\cycle_4) = 
    o(\Psi_{\SBM(n;p^2,Q^2)}(H)).
    $ 
\end{enumerate}
But, by \cref{thm:maximizingpartition2sbm}, for any 2-SBM $\SBM(p,Q),$ it holds that
$$
\Psi_{\SBM(n;p,Q)}(H) \lesssim 
\max\Bigg(
\Psi_{\SBM(n;p,Q)}(\cycle_4), 
\max_{1\le t \le 8}
\Psi_{\SBM(n;p,Q)}(\Star_t) 
\Bigg).$$

\end{theorem}
\begin{proof}
Take $H = \complete_{2,4}.$ Let $\SBM(n;p^1,Q^1)$ be just the SBM from \cref{examplethm:notodddegree}. Then, for any star $\Star_t,$
$
\displaystyle\Fourier_{\SBM(p^1,Q^1)}(\Star_t) = 0.
$
However, as $\complete_{2,4}$ has only even degree vertices,
$
\Fourier_{\SBM(p^1,Q^1)}(\complete_{2,4}) = 
1.
$

Now, let $\SBM(n;p^2,Q^2)$ be the SBM with $p^2 = (n^{-\alpha}, 1- n^{-\alpha})$ where 
$\alpha \in (0,1)$
and 
$Q^2 = \begin{pmatrix}
    0 & 1 \\
    1 & 0 \\
\end{pmatrix}.$
Then, 
$$
\Fourier_{\SBM(p^2,Q^2)}(\cycle_4)\asymp 
\Fourier_{\SBM(p^2,Q^2)}(\complete_4)\asymp (p^2_1)^2 = n^{-2\alpha}.
$$
Thus, 
$$
|\Fourier_{\SBM(n;p^2,Q^2)}(\cycle_4)|^{\frac{1}{|\vertices(\cycle_4)|}} \asymp n^{-\alpha/2} 
 = 
    o(|\Fourier_{\SBM(n;p^2,Q^2)}(\complete_{2,4})|^{\frac{1}{|\vertices(H))|}})=  o(n^{-\alpha/3}). \qedhere
$$
\end{proof}

\section{Some Omitted Details}
\subsection{Proof of Inequality \texorpdfstring{\eqref{eq:decreasinginv}}{Norms Inequality}}
\label{sec:smpleinequality}
We rewrite as follows. 

\begin{lemma}
Suppose that $(p_1,p_2,\ldots, p_k)$ is a pmf on $[k].$ Let $q_1, q_2, \ldots, q_k$ be real numbers in $[0,1].$ Then,
$$
v\longrightarrow \big(\sum_{i = 1}^k p_i^{v}q_i^{v-1}\big)^{1/v}
$$
is non-increasing on $[1,+\infty].$
\end{lemma}
\begin{proof} Define the function 
    $\displaystyle
    f(v)\coloneqq \log \big(\sum_{i = 1}^k p_i^{v}q_i^{v-1}\big)^{1/v} = 
    \frac{1}{v}\log \big(\sum_{i = 1}^k p_i^{v}q_i^{v-1}\big).
    $
    It is enough to show that $f(v)$ is non-increasing. Equivalently, we need to show that $f'(v)\le 0.$ 
    \begin{align*}
        & f'(v) \le 0 \Longleftrightarrow\\
        & -\frac{1}{v^2}\log \big(\sum_{i = 1}^k p_i^{v}q_i^{v-1}\big) + 
        \frac{1}{v}\frac{\sum_{i = 1}^k \log(p_iq_i)p_i^{v}q_i^{v-1}}{\sum_{i = 1}^k p_i^{v}q_i^{v-1}}\le 0  \Longleftrightarrow\\
        & v\sum_{i = 1}^k \log(p_iq_i)p_i^{v}q_i^{v-1}\le 
        \log \big(\sum_{i = 1}^k p_i^{v}q_i^{v-1}\big)\times 
        \sum_{i = 1}^k p_i^{v}q_i^{v-1} \Longleftrightarrow\\
        & \sum_{i = 1}^k \log(p^v_iq^v_i)p_i^{v}q_i^{v-1}\le  
        \sum_{i = 1}^k \Bigg(\log \big(\sum_{i = 1}^k p_i^{v}q_i^{v-1}\big)\times p_i^{v}q_i^{v-1} \Bigg).\\
    \end{align*}
The last inequality follows from the fact that $q_i\in [0,1]$ since for any $i,$
$$
\log(p^v_iq^v_i)\le 
\log(p^v_iq^{v-1}_i)\le 
\log \big(\sum_{i = 1}^k p_i^{v}q_i^{v-1}\big).\qedhere
$$
\end{proof}

\end{document}